\documentclass[11pt,letterpaper]{article}
\usepackage[utf8]{inputenc}
\usepackage{amsmath}
\usepackage{amssymb}
\usepackage{amsthm}
\usepackage[english]{babel}
\usepackage{booktabs}
\usepackage{graphicx}
\usepackage[colorlinks,allcolors=blue]{hyperref}
\usepackage{pgfplots}
\pgfplotsset{compat=1.14}
\usepackage{xcolor}
\usepackage[margin=1in]{geometry}
\usepackage{caption}
\usepackage{subcaption}
\usepackage{nicefrac} % To write in a better way the exponential fractions
\usepackage[affil-it]{authblk}
\usepackage{stmaryrd}
\usepackage{mathtools}
\usepackage[ruled,linesnumbered]{algorithm2e}
\usepackage{bm}
\usepackage{array}
\usepackage{empheq}

\usepackage{marginnote} % margin notes for revised version

\usepackage{tabularx}

\usepackage[shortlabels]{enumitem}

\numberwithin{equation}{section}    %The number of the equations change according to the section

\renewcommand{\arraystretch}{1.2}   % To modify the heigh of tables' rows

\usepackage{xpatch} % Name of theorems are bold
\makeatletter
   \xpatchcmd{\@thm}{\fontseries\mddefault\upshape}{}{}{} % same font as thm-header
\makeatother

\usepackage{float}

\let\oldnl\nl
\newcommand{\nonl}{\renewcommand{\nl}{\let\nl\oldnl}}

\usepackage[normalem]{ulem}

\usepackage{newtxtext}
\usepackage{newtxmath}

\tikzstyle{arrow} = [thick,->,>=stealth]

\definecolor{cadmiumgreen}{rgb}{0.0, 0.42, 0.24}
\definecolor{bole}{rgb}{0.47, 0.27, 0.23}
\definecolor{cyan(process)}{rgb}{0.0, 0.72, 0.92}

\pgfarrowsdeclare{mylatex'}{mylatex'}
{
	\newdimen\len
	\len=\pgfgetarrowoptions{mylatex'}
	\pgfarrowsleftextend{-0.4\len}
	\pgfarrowsrightextend{0.6\len}
}
{
	\newdimen\len
	\len=\pgfgetarrowoptions{mylatex'}
	\pgfpathmoveto{\pgfqpoint{0.6\len}{0\len}}
	\pgfpathcurveto
	{\pgfqpoint{0.35\len}{0.05\len}}
	{\pgfqpoint{-0.1\len}{0.15\len}}
	{\pgfqpoint{-0.4\len}{0.375\len}}
	\pgfpathcurveto
	{\pgfqpoint{-0.15\len}{0.1\len}}
	{\pgfqpoint{-0.15\len}{-0.1\len}}
	{\pgfqpoint{-0.4\len}{-0.375\len}}
	\pgfpathcurveto
	{\pgfqpoint{-0.1\len}{-0.15\len}}
	{\pgfqpoint{0.35\len}{-0.05\len}}
	{\pgfqpoint{0.6\len}{0\len}}
	\pgfusepathqfill
}
\pgfsetarrowoptions{mylatex'}{8pt}
\pgfkeys{/tiplen/.default=8pt, /tiplen/.code={\pgfsetarrowoptions{mylatex'}{#1}}}

\newcommand{\el}{\bm{E}} % (vector) electric field
\newcommand{\mg}{\bm{B}} % (vector) magnetic field
\newcommand{\gendom}{U} % general domain (open set)
\newcommand{\permitt}{\varepsilon} % electric permittivity coefficient
\newcommand{\permeab}{\mu} % magnetic permeability coefficient
\newcommand{\indices}[2]{\llbracket #1, #2\rrbracket} % set of integer indices
\newcommand{\QT}[1]{\widehat{\mathbb{QT}}_{#1}} % quasi-Trefftz space
\newcommand{\QTopt}[1]{\mathbb{QT}_{#1}}
\newcommand{\pol}[1]{\mathbb{P}_{#1}} % polynomial space
\newcommand{\hompol}[1]{\tilde{\mathbb{P}}_{#1}} % homogeneous polynomial space
\newcommand{\funel}[1]{\bm{#1}^E} % general vector function connected with the electric field
\newcommand{\funmg}[1]{\bm{#1}^B} % general vector function connected with the magnetic field
\newcommand{\taylortrun}[2]{\bm{\mathcal{T}}_{#1}^{#2}} % (vector) Taylor taylortruncation
\newcommand{\taylortrunsc}[2]{\mathcal{T}_{#1}^{#2}} % (scalar) Taylor taylortruncation
\newcommand{\curl}{\bm{\nabla}\times} % curl operator
\newcommand{\diver}{\bm{\nabla}\cdot} % divergence operator
\newcommand{\grad}{\bm{\nabla}} % gradient operator
\newcommand{\homdivfreeall}[1]{\tilde{\mathbb S}_{#1}} % homogeneous divergence free (solenoidal) polynomial vector fields
\newcommand{\homcurlfreeall}[1]{\tilde{\mathbb I}_{#1}} % homogeneous curl free (irrotational) polynomial vector fields
\newcommand{\homdivfree}[1]{\tilde{\mathbb S}_{#1}^*} % homogeneous divergence free (solenoidal) but not curl free polynomial vector fields
\newcommand{\homcurlfree}[1]{\tilde{\mathbb I}_{#1}^*} % homogeneous curl free (irrotational) but not divergence free polynomial vector fields
\newcommand{\homdivfreebis}[1]{\tilde{\mathbb S}_{#1}^{+}} 
\newcommand{\homcurlfreebis}[1]{\tilde{\mathbb I}_{#1}^{+}}
\newcommand{\homharm}[1]{\tilde{\mathbb H}_{#1}}

\newcommand{\harm}[1]{\mathbb H_{#1}}
\newcommand{\divfree}[1]{{\mathbb S}_{#1}^*}
\newcommand{\curlfree}[1]{{\mathbb I}_{#1}^*}
\newcommand{\airy}{\text{Ai}} % Airy function
\newcommand{\imag}{\mathrm{i}} % roman "i" for imaginaty unit

\newtheorem{theorem}{Theorem}[section]
\newtheorem{proposition}[theorem]{Proposition}
\newtheorem{lemma}[theorem]{Lemma}
\newtheorem{corollary}[theorem]{Corollary}
\newtheorem{assumption}[theorem]{Assumption}
\theoremstyle{remark}
\newtheorem{remark}[theorem]{Remark}
\theoremstyle{definition}
\newtheorem{definition}[theorem]{Definition}

\DeclareCaptionLabelFormat{andtable}{#1~#2  \&  \tablename~\thetable}

\title{Quasi-Trefftz spaces for the first-order time-harmonic Maxwell's equations}

\author[1]{Ilaria Fontana} %\thanks{Corresponding author}}
\author[1]{Lise-Marie Imbert-Gérard}

\affil[1]{Department of Mathematics, University of Arizona, USA}

\date{}

\begin{document}
\maketitle

\tableofcontents

\newpage

\begin{abstract}
    The goal of this article is to study optimal high-order Taylor-based polynomial quasi-Trefftz spaces for first-order time-harmonic Maxwell's equations. 
    Compared with standard polynomial spaces, the purpose of quasi-Trefftz spaces is to retain %\IF{Can we use "retain" instead of enjoy?} 
    the same high-order best approximation property %as standard polynomial spaces 
    but more efficiently: their dimension is considerably smaller.
    To achieve this goal, they approximate only smooth solutions to the governing equation rather than to approximate general smooth functions. %\IF{Is this sentence complete? A final . was missing.}
    While there is a natural definition of Taylor-based polynomial quasi-Trefftz spaces for scalar equations, the situation is fundamentally different for vector-valued equations, as the naive quasi-Trefftz property does not systematically define the smallest space with the desired best approximation property. 
    This article identifies the optimal quasi-Trefftz spaces for first-order time-harmonic Maxwell's equations, proposes two approaches to construct quasi-Trefftz bases, and proves the spaces' best approximation property. The theoretical results, which strongly leverage polynomial Helmholtz decompositions and the exactness of the polynomial de Rham sequences, are validated with numerical experiments.
    %We develop and analyze novel high-order polynomial spaces for the first-order time-harmonic Maxwell system. These locally defined spaces, called \emph{quasi-Trefftz} spaces as they are an extension of Trefftz spaces, consist of local approximate solutions of the system in the sense of Taylor truncation. 
    %% local approximate solutions obtained via Taylor truncation of the governing equations.
    %We prove that these spaces hold the same best approximation properties as standard complete polynomial spaces, but with a significant reduction in their dimension. Two algorithms for the construction of functions and bases of these quasi-Trefftz spaces are proposed, both leveraging the homogeneous underlying structure of polynomial spaces, the Helmholtz decomposition of polynomial vector fields, and the exactness of the polynomial de Rham sequence. We validate the theoretical results with numerical simulations with both constant and variable coefficients.

    %\IF{Should we put this in a single paragraph instead of splitting it into three parts?}
\end{abstract}

\section{Introduction}

Quasi-Trefftz (QT) Discontinuous Galerkin (DG) methods \cite{Imbert-Gerard_2013,IGMS_2023,Gomez_2024, Imbert-Gerard_2025} are numerical methods for Partial Differential Equation (PDE) problems that discretize a weak formulation  with problem-dependent local discrete spaces.
For context, several other methods for numerical PDEs also rely on problem-dependent discrete spaces. This is, for instance, the case of 
Trefftz methods \cite{PhD_Cessenat_1996,doi:10.1137/S0036142995285873,PHD_Moiola_2011,Hiptmair_2016, Hiptmair_2013}, 
embedded Trefftz methods {\cite{Lehrenfeld_2023,lederer2025embeddedtrefftzdgframework,stocker2026embeddedtrefftzdgmethod,gomez2026embeddedtrefftzdgmethod}}, discontinuous Petrov-Galerkin methods {\cite{DemkGopPI,DemkGopPII,DemkGopPIII,DemkGopPIV,demkowicz2025discontinuous,BOOK_Brenner_2008, BOOK_Demkowicz_2023}}, or multiscale methods \cite{GALLISTL20151,Peterseim2014EliminatingTP,10.1007/978-3-319-51954-8_6,Altmann_Henning_Peterseim_2021}.
Two key features of QT-DG methods for scalar problems are the following:
\begin{itemize}
    \item minimal numerical work is required to construct local basis functions, since this can be implemented simply by (forward) substitutions rather than numerical solves;
    \item the QT discrete spaces preserve high-order best approximation properties similar to those of standard polynomial spaces but with a significant reduction in dimension. %TO BE COMPLETED {\sc preserving the high-order best approx properties of standard discontinuous pol spaces with a significant (optimal) reduction in dimension}.
\end{itemize}

%Besides the choice of an adequate DG weak formulation,  QT-DG methods 
As DG methods, QT-DG methods rely on two fundamental pillars: (1) an adequate DG weak formulation and (2) QT discrete spaces satisfying best-approximation properties. 
While Trefftz DG methods leverage spaces of local exact solutions to the governing equation(s), the specificity of QT-DG methods is to leverage spaces of local approximate solutions to the governing equation(s).
For comparison, if the governing PDE(s) can be written as $\mathcal L u =0$, a function $\varphi$ might satisfy a local either Trefftz property or quasi-Trefftz property:
$$
\text{ either }\mathcal L \varphi = 0 \text{ or } \mathcal L \varphi \approx 0.
$$
More precisely, this work relies on a Taylor-based notion of the quasi-Trefftz property: given a point $\bm{x}_0\in \mathbb R^d$, the quasi-Trefftz property states that a certain Taylor polynomial centered at $\bm{x}_0$ of the quantity $\mathcal L \varphi$ is equal to zero.

Given the inherently local aspect of both Taylor expansions and DG discrete spaces, QT-DG discrete spaces are naturally defined as a product of local spaces defined element-wise on a mesh of the computational domain.
%this work focuses on local quasi-Trefftz spaces.
Key properties of QT-DG discrete spaces stem from their local properties on individual mesh elements;
given a mesh element $T$ and its centroid $\bm{x}_T\in T$, discrete spaces on $T$ are defined locally via Taylor expansions centered at $\bm{x}_T$.

%parag on what has been done: scal, recent both pol and GPW, and systm with Lagardere
Two main types of local Taylor-based QT-DG spaces have been studied for problems governed by scalar PDEs. One type relied on a so-called Generalized Plane Wave (GPW) ansatz \cite{Imbert-Gerard_2013,Imbert-Gerard_2024}, with higher-order polynomial terms in either the phase or the amplitude of a classical plane wave; another type simply relied on a polynomial ansatz \cite{Imbert-Gerard_2024,Imbert-Gerard_2025}. Recently \cite{Imbert-Gerard_2025-2,FIG}, an abstract framework has been proposed to unify the study of these quasi-Trefftz spaces, based on the kernel of a certain operator called the quasi-Trefftz operator; this operator is linear in the polynomial case, but not in the GPW case.
In the context of scalar PDEs \cite{Imbert-Gerard_2025-2}, the so-called principal part of the QT operator, referring to the highest-order differential operator terms with frozen coefficients, is always surjective, and this is absolutely fundamental to both the theoretical study and practical aspects of QT spaces and approximation. 
However, the situation is deeply different in the context of PDE systems, as such operators are generally not surjective. In \cite{Imbert-Gerard_2026-pre}, a first work on a PDE system tackled a first-order formulation of the Helmholtz equation in dimension 2. 
Another contribution \cite{Imbert-Gerard_2025-3} addresses the second-order time-harmonic Maxwell's equation.

%system of 6 PDEs (and 2 div conds?)
This work focuses on Taylor-based polynomial QT spaces for the first-order time-harmonic Maxwell's equations \cite{BOOK_Muller_1969, BOOK_Nedelec_2001, BOOK_Monk_2003}:
\begin{equation}\label{eq:first-order Maxwell}
    \begin{cases}
        \curl \el -\imag \kappa \permeab \mg = \bm{0}, %-(\imag k)^{-1} \bm{J} 
        %&\quad \text{in } \Omega
        \\
        \curl \mg + \imag \kappa \permitt \el = \bm{0} ,
        %&\quad \text{in } \Omega,
    \end{cases}
\end{equation}
on some open set $\Omega\subset\mathbb R^3$,
where $\el$ and $\mg$ are the unknown electric and magnetic field,  $\imag$ is the imaginary unit, $\kappa > 0$ is a fixed wavenumber, $\permitt$ is the electric permittivity and $\permeab$ is the magnetic permeability, not necessarily constant, i.e., $ \permitt(\bm{x})$, $ \permeab(\bm{x})$. 
For simplicity, in this work, we only consider the case in which these two variable coefficients are scalar functions. Associated divergence conditions and de Rham complexes are often fundamental in the analysis of PDE problems driven by Maxwell's equations and in the numerical analysis of numerical methods for such problems alike. 

Numerous numerical methods have been developed for both time-harmonic and time-dependent problems, in first- and second-order formulations.
Finite element methods \cite{BOOK_Monk_2003}
 as well as {mixed methods and DG formulations  \cite{Monk_1991, Hiptmair_2002, Cockburn_2004, BOOK_Hesthaven_2008, Houston_2005, Houston_2005-3} have been extensively studied.
 In the continuous setting, Finite Element Exterior Calculus focuses on constructing Finite Element spaces compatible with the de Rham sequence.
In the discontinuous context, divergence conditions can be neglected; they can be penalized either with constant or varying weights, or controlled based on mixed-DG formulations. 
Additionally, a Trefftz DG approach has been analyzed for the second-order time-harmonic equations with constant coefficients \cite{Hiptmair_2013}.}
By contrast, in this setting, Trefftz bases 
satisfy the divergence condition element-wise exactly, but not globally.

In \cite{Imbert-Gerard_2025-3}, a first work on Maxwell's equations considered a second-order formulation for constant permeability. Most importantly, it leverages a de Rham complex and the associated Helmholtz decomposition between spaces of homogeneous polynomials that are key to the study of the QT operator beyond the scalar PDE case. It also includes the definition of QT spaces and the study of the structure of the so-called QT operator to find their dimension. 
Besides defining QT spaces, the current work also studies their best approximation properties, their dimension, and the explicit construction of bases, all thanks to the structure of the QT operator.

In the scalar case, for a PDE of the form $\mathcal L u=0$, the QT property simply consists of the quasi-PDE: $\taylortrun{q}{\ \bm{x}_T} [\mathcal L \varphi] = 0$ where $\taylortrun{q}{\ \bm{x}_T}$ denotes the Taylor polynomial operator of degree $q$ with center $\bm{x}_T$, see for instance \cite{Imbert-Gerard_2025, Imbert-Gerard_2025-2}.
By contrast, for \eqref{eq:first-order Maxwell}, divergence conditions will play a central role in the study of the QT spaces, as it was the case in \cite{Imbert-Gerard_2025-3}.
Most importantly, to avoid any case where the divergence condition associated with \eqref{eq:first-order Maxwell} could degenerate, we introduce the following assumption.

\begin{assumption}\label{ass:coefficients}
    Given an open set $T\subset \Omega$ and $\bm{x}_T\in T$, assume that the value of two coefficients $\permitt, \permeab$ at $\bm{x}_T$ is non-zero:
    \begin{equation*}
        \permitt(\bm{x}_T) \neq 0
        \qquad\text{and}\qquad \permeab(\bm{x}_T) \neq 0.
    \end{equation*}
\end{assumption}

The exact definition of the quasi-Trefftz property will include two components:
\begin{itemize}
    \item systems of quasi-PDEs, including
\begin{itemize}
    \item a Faraday quasi-Trefftz property,
    $\taylortrun{q}{\ \bm{x}_T} [\curl \funel{Q} - \imag \kappa \permeab \funmg{Q}] = \bm{0} $,
    \item a Amp\`ere quasi-Trefftz property,
    $\taylortrun{q}{\ \bm{x}_T} [\curl \funmg{Q} + \imag \kappa \permitt \funel{Q}] = \bm{0}$;
\end{itemize}
    \item quasi-divergence conditions, including
\begin{itemize}
    \item a Gauss quasi-Trefftz property,
    $\taylortrunsc{\ell}{\,\bm{x}_T} [\diver (\permitt\funel{Q})] = 0$,
    \item a Gauss quasi-Trefftz property for magnetism,
    $\taylortrunsc{\ell}{\,\bm{x}_T} [\diver (\permeab\funmg{Q})] = 0$.
\end{itemize}
\end{itemize}
Because of symmetry in the PDEs, it is natural to choose the same degree $q$ for the two quasi-PDEs, and similarly the same degree $\ell$ for the quasi-divergence conditions.
We introduce here the formal definition of the QT property in this context, specifying how $\ell$ compares to $q$ under appropriate smoothness assumptions on the PDE coefficients. 
%\lmig{we talked about the definition of the global space, but actually (now that the content of the article is pretty much complete) it is not used anywhere, so I would like us to think about that one more time}
%\IF{Since we mention the DG formulation, I believe it is fine, for completeness to keep the definition of the global space as well.}
As a side note, although the setting consists of an open set $T$ and a point $\bm{x}_T\in T$, the proposed notation for the local QT spaces depends on $T$ rather than on $\bm{x}_T$, reflecting the fact that these spaces are meant to be defined element by element on a mesh of the computational domain for the discretization of a DG weak formulation.

\begin{definition}[Local quasi-Trefftz space] \label{def:QT local space-divergence}
    Given a positive integer $p$, an open set $T\subset\Omega\subset\mathbb R^3$, a point $\bm{x}_T\in T$, assuming that $(\permitt, \permeab)\in \left(C^{p}(T; \mathbb C)\right)^2$, consider the first-order system \eqref{eq:first-order Maxwell}.
    For any pair of smooth vector-valued functions $(\el,\mg)$, the {\em quasi-Trefftz property of order $p$}\footnote{This refers to the order of the remainder.} %\footnote{In the literature, this property can equivalently be described as "of degree $p-1$", referring to the degree of the Taylor polynomial, and "of order $p$", referring to the order of the remainder.} 
    is
    $$
    \left\{
    \begin{aligned}
        &\taylortrun{p-1}{\ \bm{x}_T} [\curl \el - \imag \kappa \permeab \mg] = \bm{0}, \\
        &\taylortrun{p-1}{\ \bm{x}_T} [\curl \mg + \imag \kappa \permitt \el] = \bm{0}, \\
        &\taylortrunsc{p-1}{\,\bm{x}_T} [\diver (\permitt \el)] = 0 ,\\
        &\taylortrunsc{p-1}{\,\bm{x}_T} [\diver (\permeab \mg)] = 0.
    \end{aligned}
    \right.
    $$
    The {\em local polynomial quasi-Trefftz space of order $p$} is
    \begin{equation}\label{eq:QT local space-divergence}
        \QTopt{p}(T) \coloneq \left\{(\funel{Q}, \funmg{Q}) \in (\pol{p})^3\times (\pol{p})^3 \ :\ 
        \begin{aligned}
            &\taylortrun{p-1}{\ \bm{x}_T} \left[\curl \funel{Q} - \imag \kappa \permeab \funmg{Q}\right] = \bm{0} \\
            &\taylortrun{p-1}{\ \bm{x}_T} \left[\curl \funmg{Q} + \imag \kappa \permitt \funel{Q}\right] = \bm{0} \\
            &\taylortrunsc{p-1}{\,\bm{x}_T} \left[\diver (\permitt\funel{Q})\right] = 0 \\
            &\taylortrunsc{p-1}{\,\bm{x}_T} \left[\diver (\permeab\funmg{Q})\right] = 0
        \end{aligned}
        \right\},
    \end{equation}
    where $\pol{p}$ denotes the complex-valued scalar polynomial space of degree at most $p$.
 \end{definition}

 \begin{remark}[Quasi-Trefftz operator]\label{rmk:qTop}
    As in the scalar case, the local quasi-Trefftz space can be characterized as the kernel of the so-called quasi-Trefftz operator defined by: 
    $$
        \mathcal D_p:(\funel{Q}, \funmg{Q}) \in (\pol{p})^3\times (\pol{p})^3 \ \mapsto\ \left(
        \begin{aligned}
            &\taylortrun{p-1}{\ \bm{x}_T} \left[\curl \funel{Q} - \imag \kappa \permeab \funmg{Q}\right]  \\
            &\taylortrun{p-1}{\ \bm{x}_T} \left[\curl \funmg{Q} + \imag \kappa \permitt \funel{Q}\right]  \\
            &\taylortrunsc{p-1}{\,\bm{x}_T} \left[\diver (\permitt\funel{Q})\right]  \\
            &\taylortrunsc{p-1}{\,\bm{x}_T} \left[\diver (\permeab\funmg{Q})\right] 
        \end{aligned}\right)
    $$
 \end{remark}

\begin{remark}[Global QT-DG space] 
    In the DG context, a global mesh-dependent discrete space is defined piecewise from discrete local spaces, and
    it is similar in the QT-DG context. If $\mathcal T_h$ is a mesh of $\Omega$ and each $T\in\mathcal T_h$ includes a given point $\bm{x}_T\in T$, the {\em global polynomial quasi-Trefftz space of order $p$} is then defined by 
    \begin{equation*}
        \QTopt{p}(\mathcal{T}_h) \coloneq \left\{(\funel{Q}, \funmg{Q}) \ :\ (\funel{Q}, \funmg{Q})|_T \in\QTopt{p}(T)\ \forall T\in\mathcal{T}_h \right\}.
    \end{equation*}
\end{remark}

% \modTC{
% We will exhibit an efficient algorithm to construct polynomial vector fields in this space, which is inspired by the approach developed for the scalar case. However, for this vector-valued case, developing a well-posed algorithm requires more sophisticated mathematical tools, such as the Helmholtz decomposition, the exactness of the homogeneous polynomial de Rham sequence, and the hypothesis that their constant coefficients in the Taylor expansion are non-zero.
% }

This article is organized as follows. Section~\ref{sec:preliminaries} introduces the main notation of the paper and collects some preliminary results, involving Taylor polynomials, the homogeneous polynomial Helmholtz decomposition, and the polynomial de Rham complex. Section \ref{sec:quasi-Trefftz} focuses on the QT space $\QTopt{p}(T)$. It includes its best approximation properties, an algorithm to construct elements of the space, as well as a procedure to construct bases for the space thanks to the previous algorithm. %we develop an efficient \lmig{this word efficient is weird, what does it mean? explain! and maybe explain better how it differs from the other Maxwell paper} algorithm for the construction of QT functions, and we show how to employ it to construct a QT basis. 
Section~\ref{sec:optimal properties} demonstrates the dimension optimality of $\QTopt{p}(T)$.
Section ~\ref{subsec:Implementation consideration} discusses computational aspects of the construction of QT bases, 
and Section~\ref{sec:numerical results} presents numerical results with both constant and variable coefficients, validating the theoretical results.

% \lmig{CHECK OUT WHERE TO PUT THIS IN THE FIRST SECTION - AND SAY ABOUT DIV CONDITION DEGENERATES IN THIS CASE
% \\
% AS IS THIS MIXES REGULARITY AND NON ZERO VALUE, IT IS MORE CONFUSING THAN ANYTHING
% \begin{assumption}\label{ass:coefficients}
%     Given $\bm{x}_T\in T$ the two coefficients $\permitt, \permeab\in C^p(T;\mathbb C)$, their Taylor truncation polynomials of degree $p$ are denoted by
%     \begin{equation*}
%         \taylortrunsc{p}{\,\bm{x}_T} [\permitt] (\bm{x}) = \sum_{k=0}^p \permitt_k(\bm{x})
%         \qquad\text{and}\qquad
%         \taylortrunsc{p}{\,\bm{x}_T} [\permeab] (\bm{x}) = \sum_{k=0}^p \permeab_k(\bm{x}),
%     \end{equation*}
%     where $\permitt_k, \permeab_k \in \hompol{k}$ for all $k\in\indices{0}{p}$, 
%     and we assume that their constant terms are non-trivial, i.e., 
%     \begin{equation*}
%         \permitt_0 = \permitt_0 (\bm{x}) = \permitt(\bm{x}_T) \neq 0
%         \qquad\text{and}\qquad
%         \permeab_0 = \permeab_0 (\bm{x}) = \permeab(\bm{x}_T) \neq 0.
%     \end{equation*}
% \end{assumption}
% }
% \IF{Theoretically, we only ask the constant coefficient ``centered'' in the barycenter to be nonzero, but in practice it would be better if this is true for all points in the domain we are considering. For example, the varying coefficient numerical test, we put the entire domain away from the plane $x=1$.}

\section{Preliminaries}\label{sec:preliminaries}

This section establishes the notation and preliminary results required for the analysis of the QT space %\eqref{eq:QT local space-divergence} 
and the development of an efficient basis-generation algorithm.

\subsection{Notation}\label{subsec:notation}
This work uses standard notation for three-dimensional polynomial vector fields.

%\lmig{here I got rid of the "we" everywhere, I also added def of $\bm m$ and of components of $\bm x$}
%Throughout this paper, 
The sets of positive and non-negative integers are, respectively, denoted by $\mathbb N$ and $\mathbb N_0$. Given two integers $n_1\leq n_2$, $\indices{n_1}{n_2}$ stands for the list of all integers ranging from $n_1$ to $n_2$.
Given two multi-indices $\bm{i} = (i_1, i_2, i_3) \in\mathbb N_0^3$, $\bm{m} = (m_1, m_2, m_3) \in\mathbb N_0^3$ and a point $\bm{x} = (x_1, x_2, x_3)\in\mathbb R^3$:
\begin{equation*}
    \lvert\bm{i}\rvert = i_1 + i_2 + i_3,
    \qquad\quad
    \bm{i}! \coloneqq i_1!\, i_2!\, i_3!,
    \qquad\quad
    \binom{\bm{i}}{\bm{m}} = \binom{i_1}{m_1} \binom{i_2}{m_2} \binom{i_3}{m_3},
    \qquad\quad
    \bm{x}^{\bm{i}} \coloneqq x_1^{i_1} x_2^{i_2} x_3^{i_3},
\end{equation*}
and $\bm{m}\leq \bm{i}$ means that $m_j \leq i_j$ for all $j\in\{1, 2, 3\}$. 
The three vectors of the canonical basis of $\mathbb R^3$ are denoted by $\bm{e}_1$, $\bm{e}_2$, and $\bm{e}_3$.  %The three multi-indices of length one are denoted with $\bm{e}_1$, $\bm{e}_2$, and $\bm{e}_3$.
The polynomial variable in three dimensions is denoted $\bm X$.

The spaces of scalar- and vector-valued polynomials of degree at most equal to $s\in\mathbb N_0$ are, respectively, denoted by $\pol{s}$ and $(\pol{s})^3$. 
% respectively, which can be written as graded spaces
% \IF{Here is a possible modification including the discussion about center of the polynomials}
Given a point $\bm{x}_0\in\mathbb R^3$, an element of these spaces can be written as
\begin{equation*}
    Q(\bm{X}) = \sum_{\vert\bm{i}\rvert\leq s} c_{\bm{i}} (\bm{X} - \bm{x}_0)^{\bm{i}}
    \qquad\text{and}\qquad
    \bm{Q}(\bm{X}) = \begin{pmatrix}
        Q_1(\bm{X}) \\
        Q_2(\bm{X}) \\
        Q_3(\bm{X})
    \end{pmatrix}= \begin{pmatrix}
        \displaystyle\sum_{\vert\bm{i}\rvert\leq s} c_{1, \bm{i}} (\bm{X} - \bm{x}_0)^{\bm{i}} \\
        \displaystyle\sum_{\vert\bm{i}\rvert\leq s} c_{2, \bm{i}} (\bm{X} - \bm{x}_0)^{\bm{i}} \\
        \displaystyle\sum_{\vert\bm{i}\rvert\leq s} c_{3, \bm{i}} (\bm{X} - \bm{x}_0)^{\bm{i}}
    \end{pmatrix},
\end{equation*}
respectively. Correspondingly, we write the polynomial spaces as graded spaces
\begin{equation}\label{eq:graded space decomposition}
    \pol{p} = \bigoplus_{k=0}^p \hompol{k}
    \qquad\qquad\text{and}\qquad\qquad
    (\pol{p})^3 = \bigoplus_{k=0}^p \left(\hompol{k}\right)^3,
\end{equation}
% \IF{Mention of center? It is equivalent with a change of coordinates, but for the Taylor polynomial we specify the center}
where $\hompol{k} = Span \{ (\bm{X}-\bm{x}_0)^{\bm{i}}, \lvert\bm{i}\rvert = k \}$ is the homogeneous polynomial space of degree $k$ centered at $\bm{x}_0$.
For simplicity, the notation for the homogeneous polynomial spaces does not explicitly mention the point $\bm{x}_0$. 
Yet, given two distinct points $(\bm{x}_1,\bm{x}_2)\in(\mathbb R^3)^2$,  $\bm x_1\neq \bm x_2$, of course $$Span \{ (\bm{X}-\bm{x}_1)^{\bm{i}}, \lvert\bm{i}\rvert = k \}\neq Span \{ (\bm{X}-\bm{x}_2)^{\bm{i}}, \lvert\bm{i}\rvert = k \}.$$
\begin{remark}\label{rmk:indepcenter}
    %\lmig{Ilaria: check if you agree with this}\IF{Seems find to me}
    Even though the $\hompol{k}$ spaces %, as well as their various subspaces defined below,
    are distinct when centered at distinct points, there are a number of aspects of these spaces that can be studied independently of the center $\bm x_0$, and, most importantly, many calculations can be performed once (for instance centered at the origin) and then applied at any other center $\bm x_0$.
\end{remark}

Given a sufficiently regular scalar function $f$ and a sufficiently regular three-dimensional vector field $\bm{F}$, we write
\begin{equation*}
    D^{\bm{i}} f \coloneqq \partial_{x_1}^{i_1} \partial_{x_2}^{i_2} \partial_{x_3}^{i_3} f
    \qquad\qquad\text{and}\qquad\qquad
    D^{\bm{i}} \bm{F} \coloneqq \begin{pmatrix}
        D^{\bm{i}} F_1 \\ D^{\bm{i}} F_2 \\ D^{\bm{i}} F_3
    \end{pmatrix}.
\end{equation*}
%\IF{In some other works, the degree written in the Taylor polynomial operator symbol is the order of the remainder and not the degree of the polynomial, like in \cite{Imbert-Gerard_2025}. I think the notation we are using now is fine since we define everything (it is also coherent with your preprints \cite{Imbert-Gerard_2025-2, Imbert-Gerard_2025-3}, but I can easily change this if we want.}
Then, their Taylor polynomial of degree $s\in\mathbb N_0$ at $\bm{x}_0\in\mathbb R^3$, $\taylortrunsc{s}{\,\bm{x}_0} [f] \in\pol{s}$ and $\taylortrun{s}{\,\,\bm{x}_0} [\bm{F}] \in (\pol{s})^3$, as
\begin{equation}\label{eq:scalar Taylor polynomial}
    \taylortrunsc{s}{\,\bm{x}_0} [f] (\bm{x}) = \sum_{\lvert\bm{i}\rvert\leq s} \frac{1}{\bm{i}!} D^{\bm{i}} f(\bm{x}_0) (\bm{X}-\bm{x}_0)^{\bm{i}},
\end{equation}
\begin{equation}\label{eq:vector Taylor polynomial}
    \taylortrun{s}{\,\,\bm{x}_0} [\bm{F}] (\bm{x}) = \sum_{\lvert\bm{i}\rvert \leq s} \frac{1}{\bm{i}!} D^{\bm{i}} \bm{F}(\bm{x}_0) (\bm{X}-\bm{x}_0)^{\bm{i}}
    = \begin{pmatrix}
        \taylortrunsc{s}{\,\bm{x}_0} [F_1](\bm{x}) \\
        \taylortrunsc{s}{\,\bm{x}_0} [F_2](\bm{x}) \\
        \taylortrunsc{s}{\,\bm{x}_0} [F_3](\bm{x})
    \end{pmatrix}.
\end{equation}

Given an open set $\gendom\subseteq \mathbb R^3$ and a vector field $\bm{F}\in  C^s(\gendom; \mathbb C^3)$, we denote the standard $ C^s$ norm and seminorms, $s\in \mathbb N_0$, $s\leq p$, as follows:
\begin{equation}\label{eq:norms}
    \lVert\bm{F}\rVert_{ C^0(\gendom)} \coloneqq \sup_{\bm{x}\in \gendom} \Vert\bm{F}\rVert, 
    \qquad\text{and}\qquad
    \vert\bm{F}\rvert_{ C^s(\gendom)} \coloneqq \max_{\bm{i}\in\mathbb N_0^3, \lvert\bm{i}\rvert = s} \left\lVert D^{\bm{i}} \bm{F}\right\rVert_{ C^0(\gendom)},
\end{equation}
where $\lVert\,\cdot\,\rVert$ is the Euclidean norm.
%\IF{Moved here as suggested}

For any $k > 0$ and given a point $\bm{x}_0\in\mathbb R^3$, several subspaces of $(\hompol{k})^3$ are of interest in this work. Consider here:
%the homogeneous polynomial spaces $\homharm{k}$, $\homdivfreeall{k}$, $\homcurlfreeall{k}$, $\homdivfree{k}$, and $\homcurlfree{k}$ are defined as subspaces of $(\hompol{k})^3$ as follows for $k > 0$:
% \lmig{it seems that it would be more clear to define spaces (as kernels) separately from the direct sum properties} 
% \IF{Alternative way below}
% \begin{itemize}
%     \item $\homharm{k}$ is the space of ${\it harmonic}$ polynomials, that is, the divergence-free and curl-free homogeneous polynomial vector field space of degree $k$; 
%     \item $\homdivfreeall{k} = \homdivfree{k} \oplus \homharm{k}$ is the space of {\it solenoidal} polynomials, that is, the divergence-free homogeneous polynomial vector field space of degree $k$;
%     \item $\homcurlfreeall{k} = \homcurlfree{k} \oplus \homharm{k}$ is the space of {\it irrotational} polynomials, that is, the curl-free homogeneous polynomial vector field space of degree $k$.
% \end{itemize}
% \IF{Alternative introduction of these spaces:} \lmig{yes, much better! make sure to state this for $k>0$}
%\modTC{
\begin{itemize}
    \item $\homdivfreeall{k} \coloneqq \ker (\diver |_{(\hompol{k})^3})$ is the space of {\it solenoidal} polynomials, i.e., the divergence-free homogeneous polynomial vector field space of degree $k$;
    \item $\homcurlfreeall{k} \coloneqq \ker (\curl |_{(\hompol{k})^3})$ is the space of {\it irrotational} polynomials, i.e., the curl-free homogeneous polynomial vector field space of degree $k$;
    \item $\homharm{k} \coloneqq \homdivfreeall{k} \cap \homcurlfreeall{k}$ is the space of ${\it harmonic}$ polynomials, i.e., the divergence-free and curl-free homogeneous polynomial vector field space of degree $k$; 
    \item $\homdivfree{k}$ is the complement of $\homharm{k}$ in $\homdivfreeall{k}$, i.e., $\homdivfreeall{k} = \homdivfree{k} \oplus \homharm{k}$;
    \item $\homcurlfree{k}$ is the complement of $\homharm{k}$ in $\homcurlfreeall{k}$, i.e., $\homcurlfreeall{k} = \homcurlfree{k} \oplus \homharm{k}$;
\end{itemize}
%}
The case $k = 0$ is simpler, as all constant vector fields are harmonic and accordingly we define $\homharm{0} \coloneqq (\hompol{0})^3$.
For simplicity, we also denote by $\harm{p}$, $p\in\mathbb N_0$, the space of all harmonic polynomial vector fields of degree at most $p$, i.e., $$\harm{p} \coloneqq \bigoplus_{k=0}^p \homharm{k},$$
and for all $p>0$
$$
    \curlfree{p} \coloneqq \bigoplus_{k=1}^p \homcurlfree{k}
    \quad\text{ and }\quad
    \divfree{p} \coloneqq \bigoplus_{k=1}^p \homdivfree{k}.
$$
As a reminder, the dimensions of key polynomial spaces are recorded in Table \ref{tab:all spaces dimensions}.
\begin{table}[tb]
    \centering
    \renewcommand{\arraystretch}{1.5}
    \begin{tabular}{ccccccc}
        \toprule
        Space & $(\hompol{k})^3$ & $\homharm{k}$ & $\homdivfreeall{k}$ & $\homcurlfreeall{k}$ & $\homdivfree{k}$ & $\homcurlfree{k}$ \\ \midrule
        Dim. & $\frac{3}{2} (k+1)(k+2)$ & $2k+3$ & $(k+1)(k+3)$ & $\frac{1}{2}(k+2)(k+3)$ & $k(k+2)$ & $\frac{1}{2} k(k+1)$ \\ \bottomrule
    \end{tabular}
    \caption{Dimensions of various homogeneous polynomial spaces for $k>0$.}
    \label{tab:all spaces dimensions}
\end{table}

% \begin{table}[tb]
%     \centering
%     \begin{tabular}{c|c}
%         Space & Dimension \\ \hline
%         $(\hompol{k})^3$ & $\frac{3}{2} (k+1)(k+2)$ \\
%         $\homharm{k}$ & $2k+3$ \\
%         $\homdivfreeall{k}$ & $(k+1)(k+3)$ \\
%         $\homcurlfreeall{k}$ & $\frac{1}{2}(k+2)(k+3)$ \\
%         $\homdivfree{k}$ & $k(k+2)$ \\
%         $\homcurlfree{k}$ & $\frac{1}{2} k(k+1)$
%     \end{tabular}
%     \caption{Dimensions of the spaces defined in Subsection \ref{subsec:Helmholtz}.}
%     \label{tab:all spaces dimensions}
% \end{table}

\begin{remark}
%\lmig{check this one too}\IF{Yes, should be fine as well}
    Following Remark \ref{rmk:indepcenter} and the definitions of $\homdivfreeall{k}$, $\homcurlfreeall{k}$,  $\homharm{k}$ (except $\homharm{0}$), $\homdivfree{k}$ and  $\homcurlfree{k}$, these spaces also depend on the center $\bm{x}_0$. However, $\harm{p}$ is the full space of polynomial harmonic vector fields of degree at most equal to $p$, and it does not depend on $\bm{x}_0$; similarly, $\curlfree{p}$ and $\divfree{p}$ do not depend on $\bm{x}_0$ either.
\end{remark}

\subsection{De Rham complex and more
%Homogeneous Polynomial Helmholtz Decomposition and De Rham Complex
}
\label{ssec:HPHD}

This section presents key results regarding the homogeneous polynomial \emph{de Rham complex} and  \emph{Helmholtz decomposition}. 
For comparison, polynomial decompositions are standard in Galerkin methods for Maxwell's equations, both for time-dependent and time-harmonic formulations; see for instance Remarks 1 and 15 in~\cite{Moiola_2018} in the context of Trefftz methods or \cite{BOOK_Monk_2003} in the context of Finite Elements. As an illustration of Remark \ref{rmk:indepcenter}, none of the results below depend on the center $\bm x_0$ in the definition of homogeneous polynomial spaces. 

As a reminder, the following homogeneous polynomial \emph{Helmholtz decomposition} was stated and proved in~\cite{Imbert-Gerard_2025-3}.
%The homogeneous polynomial \emph{Helmholtz decomposition} of \cite[Section 3]{Imbert-Gerard_2025-3} will again play a key role in this work. It is stated here for convenience.
\begin{proposition}[Unique Helmholtz decomposition]
    Let $k\in\mathbb N$ and a point $\bm{x}_0\in\mathbb R^3$. Any homogeneous polynomial vector field $\bm{Q}\in(\hompol{k})^3$ of degree $k$ can be uniquely decomposed into the sum of a solenoidal polynomial vector field $\bm{F}\in\homdivfree{k}$, an irrotational polynomial vector field $\bm{G}\in\homcurlfree{k}$, and a harmonic polynomial vector field $\bm{H}\in\homharm{k}$:
    \begin{equation}\label{eq:Helmholtz}
        \setlength{\arraycolsep}{2pt}
        \begin{array}{c c c c c c c}
            (\hompol{k})^3 & = & \homdivfree{k} & \oplus & \homcurlfree{k} & \oplus & \homharm{k} \\
            \bm{Q} & = & \bm{F} & + & \bm{G} & + & \bm{H}.
        \end{array}
    \end{equation}
\end{proposition}

%One key point to be able to develop an efficient algorithm for constructing quasi-Trefftz functions in $\QTopt{p}(T)$ \eqref{eq:QT local space-divergence} is the exactness of the homogeneous polynomial de Rham complex. 
In the polynomial context, the common exactness result involves the standard polynomial spaces, see for instance \cite{Douglas_2006}, and can be summarized as:
\begin{equation*}
    \mathbb R \hookrightarrow \pol{p+3} \xrightarrow{\grad} (\pol{p+2})^3 \xrightarrow{\curl} (\pol{p+1})^3 \xrightarrow{\diver} \pol{p} \to \{0\}.
\end{equation*}
Its extension to homogeneous spaces reads:
\begin{equation}\label{eq:De Rham complex}
    \{0\} \to \hompol{k+3} \xrightarrow{\grad} (\hompol{k+2})^3 \xrightarrow{\curl} (\hompol{k+1})^3 \xrightarrow{\diver} \hompol{k} \to \{0\}.
\end{equation}
The next result focuses on the two differential operators that come into play in the definition of the quasi-Trefftz space.
%\modTC{
\begin{lemma}[Surjective curl and divergence operators]\label{prop:surjectivity operators}
    Let $k\in\mathbb N$ and a point $\bm{x}_0\in\mathbb R^3$. Then,
    \begin{enumerate}[(a)]
        \item\label{curlsurj} the operator $\curl|_{\homdivfreeall{k}} \colon \homdivfreeall{k} \to \homdivfreeall{k-1}$ is surjective; \label{it:surjective curl}
        \item the operator $\diver|_{\homcurlfreeall{k}} \colon \homcurlfreeall{k} \to \hompol{k-1}$ is surjective.
    \end{enumerate}
\end{lemma}
\begin{proof}
    %In this proof, we use the space dimensions shown in Table \ref{tab:all spaces dimensions}. \\
    \noindent
    %\ref{curlsurj} 
    The range of the $\curl|_{\homdivfreeall{k}}$ operator is indeed equal to the $\homdivfreeall{k-1}$ space, because:
    \begin{itemize}
        \item on the one hand, $\mathcal R(\curl|_{\homdivfreeall{k}}) \subset  \homdivfreeall{k-1}$,
        \\a direct consequence of the exactness of the de Rham sequence,
        \item on the other hand, $\text{rk}(\curl|_{\homdivfreeall{k}}) = \dim \homdivfreeall{k-1}$,\\
        which follows from the rank-nullity theorem combined with $\ker (\curl|_{\homdivfreeall{k}}) = \homharm{k} $.
    \end{itemize}
    %(a) For all $\bm{F}\in\homdivfreeall{k}$, $\diver\curl\bm{F} = 0$, and, as a consequence,
    % \begin{equation}\label{eq:surjectivity 1}
    %     \mathcal R %\text{Range} 
    %     (\curl|_{\homdivfreeall{k}}) \subseteq \homdivfreeall{k-1}
    %     \qquad\text{and}\qquad
    %     \text{rk}%\text{rank} 
    %     (\curl|_{\homdivfreeall{k}}) \leq \dim \homdivfreeall{k-1} = k (k+2).
    % \end{equation}
    % Moreover, 
    % \begin{equation*}
    %     \bm{F} \in \ker (\curl|_{\homdivfreeall{k}})
    %     \qquad\Leftrightarrow\qquad
    %     \diver\bm{F} = 0\ \text{and}\ \curl\bm{F} = \bm{0} 
    %     \qquad\Leftrightarrow\qquad
    %     \bm{F} \in \homharm{k}
    % \end{equation*}
    % which implies
    % \begin{equation*}
    %     \text{nullity} (\curl|_{\homdivfreeall{k}}) = \dim \homharm{k} = 2k+3.
    % \end{equation*}
    % By applying the rank-nullity theorem, we get
    % \begin{equation}\label{eq:surjectivity 2}
    %     \text{rank} (\curl|_{\homdivfreeall{k}}) = \dim \homdivfreeall{k} - \text{nullity} (\curl|_{\homdivfreeall{k}}) = (k+1) (k+3) - (2k+3) = k (k+2).
    % \end{equation}
    % Combining \eqref{eq:surjectivity 1} and \eqref{eq:surjectivity 2}, we conclude that $\curl \colon \homdivfreeall{k} \to \homdivfreeall{k-1}$ is surjective.\\
    % \noindent
    %
    Similarly, the range of the $\diver|_{\homcurlfreeall{k}}$ operator is equal to the $\hompol{k-1}$ space, because:
    \begin{itemize}
        \item on the one hand, $\mathcal R(\diver|_{\homcurlfreeall{k}}) \subset  \hompol{k-1}$,
        \item on the other hand, $\text{rk}(\diver|_{\homcurlfreeall{k}}) = \dim \hompol{k-1}$,\\
        which follows again from the rank-nullity theorem combined with $\ker (\diver|_{\homcurlfreeall{k}}) = \homharm{k} $.
    \end{itemize}
    % (b) can be shown similarly to \ref{it:surjective curl}, with the difference that the image of $\diver$ is the entire polynomial space $\hompol{k-1}$.
\end{proof}
% }
%\lmig{comment on next result?}
Finally, the following proposition will be fundamental for the construction of QT basis functions in Section~\ref{subsec:construction}.
% \modTC{
\begin{proposition}[Bijective curl and divergence operators]\label{prop:bijectivity operators}
    Let $k\in\mathbb N$ and a point $\bm{x}_0\in\mathbb R^3$. %, and let $\homharm{k}, \homdivfreeall{k}, \homdivfree{k}, \homcurlfreeall{k}, \homcurlfree{k}$ the polynomial vector field spaces introduced in Subsection \ref{subsec:notation}. %Additionally, let $\curl$ and $D_{k-1}$ denote the curl and divergence operators restricted to $(\hompol{k})^3$, as in \cite{Imbert-Gerard:25-Maxwell}. 
    Then,
    \begin{enumerate}[(a)]
        \item the operator $\curl|_{\homdivfree{k}} \colon \homdivfree{k} \to \homdivfreeall{k-1}$ is bijective; \label{it:bijective curl}
        \item the operator $\diver|_{\homcurlfree{k}} \colon \homcurlfree{k} \to \hompol{k-1}$ is bijective.
    \end{enumerate}
\end{proposition}
\begin{proof}
    These results directly follow from Lemma \ref{prop:surjectivity operators} together with the facts that
    \begin{equation*}
        \homdivfreeall{k} = \homdivfree{k} \oplus \ker (\curl|_{\homdivfree{k}})
        \qquad\text{and}\qquad
        \homcurlfreeall{k} = \homcurlfree{k} \oplus \ker( \diver|_{\homcurlfree{k}} ).
    \end{equation*}
    %By definition of the spaces, $\ker(\curl|_{\homdivfree{k}}) = \{\bm{0}\}$ and $\ker(\diver|_{\homcurlfree{k}}) = \{\bm{0}\}$. The conclusion follows from the fact that 
    % \begin{equation*}
    %     \dim \homdivfree{k} = \dim \homdivfreeall{k-1},
    %     \qquad%\text{and}\qquad
    %     \dim \homcurlfree{k} = \dim \hompol{k=1},
    % \end{equation*}
\end{proof}

For comparison, the curl and divergence operators defined between convenient polynomial spaces are also considered in \cite{Imbert-Gerard_2025-3}. There, Lemmas 2 and 3 focus on these operators defined on standard polynomial spaces, while Proposition 5 corresponds to Proposition \ref{prop:bijectivity operators}-(b). The latter is proved again here simply for the sake of completeness.

% \IF{I checked the results presented in \cite{Imbert-Gerard_2025-3}:
% \begin{itemize}
%     % \item The existence and uniqueness of the Helmholtz decomposition \eqref{eq:Helmholtz} is proved and stated exactly like here, \cite[Proposition 3]{Imbert-Gerard_2025-3}
%     \item Curl operator:
%     \begin{itemize}
%         \item \cite[Lemma 3]{Imbert-Gerard_2025-3} shows how to find a basis of the range of the curl operator defined on the entire polynomial space
%         \item There is no explicit results about the restricted operators
%     \end{itemize}
%     \item Divergence operator:
%     \begin{itemize}
%         \item \cite[Lemma 2]{Imbert-Gerard_2025-3} states and proves the surjectivity of the operator defined on the full polynomial space
%         \item \cite[Proposition 5]{Imbert-Gerard_2025-3} is exactly Proposition \ref{prop:bijectivity operators}-(b)
%     \end{itemize}
% \end{itemize}
% As a consequence, I would leave the results as they are presented here: no proof for the unique Helmholtz decomposition, and short proof for Proposition~\ref{prop:surjectivity operators} and \ref{prop:bijectivity operators}.}

\section{Quasi-Trefftz space}\label{sec:quasi-Trefftz}

This section shows the main result of this paper, studying the properties of the QT space $\QTopt{p}(T)$ and developing an efficient algorithm for the construction of a QT function and basis. In particular, we will prove that this space has a significantly lower dimension compared to the complete polynomial space $(\pol{p})^3\times (\pol{p})^3$, see Corollary~\ref{cor:dim} and Figure~\ref{fig:dimensions}, while maintaining the same best approximation properties, see Theorem~\ref{th:approximation properties QTopt}.

From now on, we will use the natural decomposition of the electric and magnetic polynomial vector fields:
\begin{equation*}
    \funel{Q} = \displaystyle\sum_{k=0}^p \funel{Q}_k
    \qquad\text{and}\qquad
    \funmg{Q} = \displaystyle\sum_{k=0}^p \funmg{Q}_k,
\end{equation*}
where $\funel{Q}_k, \funmg{Q}_k$ are the homogeneous components of $\funel{Q}, \funmg{Q}$, respectively, of degree $k$ centered at $\bm{x}_T$, 
%\lmig{since nothin is said about T in the previous sentence, I wouldn't say anything here about xT - since it is not a formal result I think it is ok like that}\IF{Here, do we need to specify that the point belongs to $T$ (hence the subscript), which is an open set?},
i.e., $\funel{Q}_k, \funmg{Q}_k \in (\hompol{k})^3$, $k\in\indices{0}{p}$.
Similarly, we decompose the Taylor polynomials of the two-variable coefficients:
\begin{equation}\label{eq:Taylor polynomials of coefficients p-1}
    \taylortrunsc{p-1}{\bm{x}_T} [\permitt] = \displaystyle\sum_{k=0}^{p-1} \permitt_k
    \qquad\text{and}\qquad
    \taylortrunsc{p-1}{\bm{x}_T} [\permeab] = \displaystyle\sum_{k=0}^{p-1} \permeab_k,
\end{equation}
where $(\permitt_k,\permeab_k)\in (\hompol{k})^2$ for each value of $k$.
Then, the local polynomial QT property in \eqref{eq:QT local space-divergence} can be rewritten as
\begin{equation}\label{eq:QT conditions rewritten}
    \begin{cases}
        \curl \funel{Q}_{k} - \imag \kappa \displaystyle\sum_{j=0}^{k-1} \permeab_j \funmg{Q}_{k-1-j} = \bm{0} \smallskip\\
        \curl \funmg{Q}_{k} + \imag \kappa \displaystyle\sum_{j=0}^{k-1} \permitt_j \funel{Q}_{k-1-j} = \bm{0} 
        \smallskip \\
        %\modTC{
        \diver \left( \displaystyle\sum_{j=0}^k \permitt_j \funel{Q}_{k-j} \right) = 0
        %}
        \smallskip \\
        %\modTC{
        \diver \left( \displaystyle\sum_{j=0}^k \permeab_j \funmg{Q}_{k-j} \right) = 0
        %}
    \end{cases}
    \qquad \forall k\in\indices{1}{p}.
\end{equation}
%\IF{I added the divergence conditions to represent the full QT. This way of rewriting is never mentioned later. Maybe it would be worth it referencing it in the Algorithm . See my comment there. If not, I would still keep the equation but without a number and label.}

\subsection{Best approximation properties 
%- Second attempt
}
\label{ssec:APPProp}

The standard polynomial space achieves the optimal algebraic best approximation order, and Taylor polynomials realize that order locally. %Taylor polynomials give locally near-optimal approximations of smooth functions, in terms of rate of convergence, and as a consequence standard polynomial spaces enjoy optimal-order best approximation properties.
The key property of the QT space associated with a given PDE is that it includes a particular Taylor polynomial of smooth exact PDE solutions. 
As a consequence, the QT space $\QTopt{p}(T)$ will also achieve the optimal algebraic approximation order.

\begin{proposition}[Taylor polynomial and quasi-Trefftz space]\label{prop:TaylinQT}
     Given $p\in\mathbb N$, let $T\subset \mathbb R^3$ be an open set, $\bm{x}_T \in T$, and suppose that $\varepsilon, \mu\in C^{p}(T; \mathbb C)$. 
    Consider the local polynomial QT space $\QTopt{p}(T)$.
    If $(\el, \mg) \in C^{p+1}(T; \mathbb C^3) \times C^{p+1}(T; \mathbb C^3)$ satisfies the Maxwell system \eqref{eq:first-order Maxwell}, then 
    $$\left(\taylortrun{p}{\ \bm{x}_T} [\el], \taylortrun{p}{\ \bm{x}_T} [\mg]\right) \in \QTopt{p}(T).$$
\end{proposition}
\begin{proof}
    % Given any $(\el, \mg) \in C^{p+1}(T; \mathbb C^3) \times C^{p+1}(T; \mathbb C^3)$, $(\taylortrun{p}{\ \bm{x}_T} [\el], \taylortrun{p}{\ \bm{x}_T} [\mg]) \in (\pol{p})^3\times (\pol{p})^3$ by definition, so the
     The proof relies on verifying that $\left(\taylortrun{p}{\ \bm{x}_T} [\el], \taylortrun{p}{\ \bm{x}_T} [\mg]\right)$ satisfies the four QT properties \eqref{eq:QT local space-divergence}. % the Faraday QT property, the Amp\`ere QT property, the Gauss QT property, and the divergence-free QT property.
     The following statements follow directly from Proposition \ref{prop:divergence+taylor} and \ref{prop:opertors+taylor w/ coefficients}: %Lemma \ref{lmm:minimal}:
    \begin{itemize}
        \item 
    %    {\bf The Faraday QT property.}
    %     \\Given any $\bm{i}\in\mathbb N_0^3$ such that $\lvert\bm{i}\rvert \leq p-1$, since
    % \begin{equation*}%\label{eq:approx th-div 1}
    %     \begin{split}
    %     &D^{\bm{i}} \left(\curl \taylortrun{p}{\ \bm{x}_T} [\el] - \imag \kappa \permeab \taylortrun{p}{\ \bm{x}_T} [\mg]\right) (\bm{x}_T) 
    %     \\&= D^{\bm{i}} \curl \taylortrun{p}{\ \bm{x}_T} [\el] (\bm{x}_T) - \imag \kappa D^{\bm{i}} (\permeab \taylortrun{p}{\ \bm{x}_T} [\mg]) (\bm{x}_T)
    %     \\&= D^{\bm{i}} \curl \taylortrun{p}{\ \bm{x}_T} [\el] (\bm{x}_T) - \imag \kappa \sum_{\substack{\bm{m}\in \mathbb N_0^3,\\ \bm{m}\leq\bm{i}}} \binom{\bm{i}}{\bm{m}} D^{\bm{m}} \permeab(\bm{x}_T) D^{\bm{i}-\bm{m}} \taylortrun{p}{\ \bm{x}_T} [\mg] (\bm{x}_T) \\
    %         &= D^{\bm{i}} \curl \el (\bm{x}_T) - \imag \kappa \sum_{\substack{\bm{m}\in \mathbb N_0^3,\\ \bm{m}\leq\bm{i}}} \binom{\bm{i}}{\bm{m}} D^{\bm{m}} \permeab(\bm{x}_T) D^{\bm{i}-\bm{m}} \mg (\bm{x}_T)
    %        \\
    %         &= D^{\bm{i}} \left(\curl \el  - \imag \kappa  \permeab \mg \right)(\bm{x}_T)
    %     \end{split}
    % \end{equation*}
    given any $\bm{i}\in\mathbb N_0^3$ such that $\lvert\bm{i}\rvert \leq p-1$, since
    $D^{\bm{i}} (\permeab \taylortrun{p}{\ \bm{x}_T} [\mg]) (\bm{x}_T) =  D^{\bm{i}} \left(\permeab \mg \right)(\bm{x}_T)$
    and similarly $D^{\bm{i}} (\permitt \taylortrun{p}{\ \bm{x}_T} [\el]) (\bm{x}_T) =  D^{\bm{i}} \left(\permitt \el \right)(\bm{x}_T)$, then
    $$
    \begin{cases}
        D^{\bm{i}} \left(\curl \taylortrun{p}{\ \bm{x}_T} [\el] - \imag \kappa \permeab \taylortrun{p}{\ \bm{x}_T} [\mg]\right) (\bm{x}_T)  = D^{\bm{i}} \left(\curl \el  - \imag \kappa  \permeab \mg \right)(\bm{x}_T),
        \\
        D^{\bm{i}} \left(\curl \taylortrun{p}{\ \bm{x}_T} [\mg] + \imag \kappa \permitt \taylortrun{p}{\ \bm{x}_T} [\el]\right) (\bm{x}_T)  = D^{\bm{i}} \left(\curl \mg  + \imag \kappa  \permitt \el \right)(\bm{x}_T);
    \end{cases}
    $$
        %\item 
        % {\bf The Amp\`ere QT property.}
    % Given any $\bm{i}\in\mathbb N_0^3$ such that $\lvert\bm{i}\rvert \leq p-1$, since
    % $D^{\bm{i}} (\permitt \taylortrun{p}{\ \bm{x}_T} [\el]) (\bm{x}_T) =  D^{\bm{i}} \left(\permitt \el \right)(\bm{x}_T)$, then
    % $$D^{\bm{i}} \left(\curl \taylortrun{p}{\ \bm{x}_T} [\mg] + \imag \kappa \permitt \taylortrun{p}{\ \bm{x}_T} [\el]\right) (\bm{x}_T)  = D^{\bm{i}} \left(\curl \mg  + \imag \kappa  \permitt \el \right)(\bm{x}_T).$$
        \item 
        %{\bf The Gauss QT property.}
        given any $\bm{i}\in\mathbb N_0^3$ such that $\lvert\bm{i}\rvert \leq p-1$, then
    $$
    \begin{cases}
        D^{\bm{i}} \diver\left(\permitt \taylortrun{p}{\ \bm{x}_T} [\el]\right) (\bm{x}_T)  = D^{\bm{i}} \diver\left( \permitt \el \right)(\bm{x}_T),
        \\
        D^{\bm{i}} \diver\left(\permeab \taylortrun{p}{\ \bm{x}_T} [\mg]\right) (\bm{x}_T)  = D^{\bm{i}} \diver\left( \permeab \mg \right)(\bm{x}_T).
    \end{cases}$$
    %     \item 
    %     %{\bf The Gauss QT property for magnetism.}
    %     Given any $\bm{i}\in\mathbb N_0^3$ such that $\lvert\bm{i}\rvert \leq p-1$, then
    % $$D^{\bm{i}} \diver\left(\permeab \taylortrun{p}{\ \bm{x}_T} [\mg]\right) (\bm{x}_T)  = D^{\bm{i}} \diver\left( \permeab \mg \right)(\bm{x}_T).$$
    \end{itemize}
    The result then simply follows when the pair $(\el, \mg)$ satisfies the Maxwell system \eqref{eq:first-order Maxwell}. %: $(\taylortrun{p}{\ \bm{x}_T} [\el], \taylortrun{p}{\ \bm{x}_T} [\mg]) \in \QTopt{p}(T)$.
\end{proof}

%\lmig{not sure this sentence is really useful} The best approximation properties of QT spaces are as follows.
\begin{theorem}[Best approximation properties.]\label{th:approximation properties QTopt}
    Given $p\in\mathbb N$, let $T\subset \mathbb R^3$ be an open set, $\bm{x}_T\in T$ be such that $T$ is star-shaped about $\bm{x}_T$, and suppose that $\varepsilon, \mu\in C^{p}(T; \mathbb C)$. 
    Consider the local polynomial QT space $\QTopt{p}(T)$. % the defined by \eqref{eq:QT local space-divergence},
    If $(\el, \mg) \in C^{p+1}(T; \mathbb C^3) \times C^{p+1}(T; \mathbb C^3)$ satisfies the Maxwell system \eqref{eq:first-order Maxwell}, then 
    \begin{equation*}
        \begin{split}
            \inf_{(\funel{Q}, \funmg{Q})\in\QTopt{p}(T)} \Bigl(\left\lvert \el - \funel{Q}\right\rvert_{C^0(T)} &+ \left\lvert \mg - \funmg{Q} \right\rvert_{C^0(T)}\Bigr)  \\
            % \leq&\ \lvert \el - \taylortrun{p}{\,\,\bm{x}_T} [\el] \rvert_{C^s(T)} + \lvert \mg - \taylortrun{p}{\,\,\bm{x}_T} [\mg] \rvert_{C^s(T)} \\
            \leq&\ \frac{3^{p+1}}{(p+1)!} h_T^{p+1} \Bigl( \left\lvert\el\right\rvert_{C^0(T)} + \left\lvert\mg\right\rvert_{C^0(T)} \Bigr),
        \end{split}
    \end{equation*}
    and for all $s\in\mathbb N, s\leq p$
    \begin{equation*}
        \begin{split}
            \inf_{(\funel{Q}, \funmg{Q})\in\QTopt{p}(T)} \Bigl(\left\lvert \el - \funel{Q}\right\rvert_{C^s(T)} &+ \left\lvert \mg - \funmg{Q} \right\rvert_{C^s(T)}\Bigr)  \\
            % \leq&\ \lvert \el - \taylortrun{p}{\,\,\bm{x}_T} [\el] \rvert_{C^s(T)} + \lvert \mg - \taylortrun{p}{\,\,\bm{x}_T} [\mg] \rvert_{C^s(T)} \\
            \leq&\ \frac{3^{p+1-s}}{(p+1-s)!} h_T^{p+1-s} \Bigl( \left\lvert\el\right\rvert_{C^s(T)} + \left\lvert\mg\right\rvert_{C^s(T)} \Bigr)
        \end{split}
    \end{equation*}
    %\label{it:approximation QTopt 2}
    %\IF{Put the case $s=0$ to a separate point as (ii) and then the current second to (iii)}
    % \lmig{I suppose your comment would adjust here as: split this in two statements to have the s=0 case apart from the rest}
\end{theorem}
As in the scalar case,  see for instance \cite[Theorem 2.4]{Imbert-Gerard_2025}, the following proof relies on approximation properties of the Taylor polynomials of exact solutions.
%This proof relies on standard arguments developed for scalar PDEs, see for instance \cite[Theorem 2.4]{Imbert-Gerard_2025}, and no major difficulty arises in the vector-valued case. 
\begin{proof}
Following from Proposition \ref{prop:TaylinQT}, the Taylor polynomials of degree $p$ of the exact solutions $(\el, \mg)$ satisfy:
%Proposition \ref{prop:TaylinQT} \lmig{not a fan of the word automatically: it can lead to a lot of things, not sure why this would be automatic} automatically leads to
    \begin{equation*}
        \begin{split}
            \inf_{(\funel{Q}, \funmg{Q})\in\QTopt{p}} \Bigl(\left\lvert \el - \funel{Q}\right\rvert_{C^s(T)} &+ \left\lvert \mg - \funmg{Q} \right\rvert_{C^s(T)}\Bigr)  \\
            \leq&\ \left\lvert \el - \taylortrun{p}{\,\,\bm{x}_T} [\el] \right\rvert_{C^s(T)} + \left\lvert \mg - \taylortrun{p}{\,\,\bm{x}_T} [\mg] \right\rvert_{C^s(T)}. \\
            % \leq&\ \frac{3^{p+1-s}}{(p+1-s)!} h_T^{p+1-s} \Bigl( \lvert\el\rvert_{C^s(T)} + \lvert\mg\rvert_{C^s(T)} \Bigr)
        \end{split}
    \end{equation*}
    Moreover,
    \begin{equation*}
        \begin{split}
            \left\lvert \el - \taylortrun{p}{\,\,\bm{x}_T}\right. &\left.[\el] \right\rvert_{C^s(T)} + \left\lvert \mg - \taylortrun{p}{\,\,\bm{x}_T} [\mg] \right\rvert_{C^s(T)} \\
            =&\ \max_{\bm{i}\in\mathbb N_0^3, \left\lvert\bm{i}\right\rvert = s} \left\lVert D^{\bm{i}} \el - \taylortrun{p-s}{\ \bm{x}_T} \left[D^{\bm{i}} \el\right] \right\rVert_{ C^0(T)} + \max_{\bm{i}\in\mathbb N_0^3, \left\lvert\bm{i}\right\rvert = s} \left\lVert D^{\bm{i}} \mg - \taylortrun{p-s}{\ \bm{x}_T} \left[D^{\bm{i}} \mg\right] \right\rVert_{ C^0(T)} \\
            \leq&\ \max_{\bm{i}\in\mathbb N_0^3, \left\lvert\bm{i}\right\rvert = s} \sum_{\substack{\bm{j}\in \mathbb N_0^3,\\ \left\lvert\bm{j}\right\rvert = p+1-s}} \frac{1}{\bm{j}!} \sup_{\bm{x}, \bm{x}_*\in T} \left\lVert D^{\bm{i} + \bm{j}} \el(\bm{x}_*) (\bm{x}-\bm{x}_T)^{\bm{j}} \right\rVert \\
            &\ + \max_{\bm{i}\in\mathbb N_0^3, \left\lvert\bm{i}\right\rvert = s} \sum_{\substack{\bm{j}\in \mathbb N_0^3,\\ \left\lvert\bm{j}\right\rvert = p+1-s}} \frac{1}{\bm{j}!} \sup_{\bm{x}, \bm{x}_*\in T} \left\lVert D^{\bm{i} + \bm{j}} \mg(\bm{x}_*) (\bm{x}-\bm{x}_T)^{\bm{j}} \right\rVert \\
            \leq&\ \frac{3^{p+1-s}}{(p+1-s)!} h_T^{p+1-s} \Bigl( \left\lvert\el\right\rvert_{C^s(T)} + \left\lvert\mg\right\rvert_{C^s(T)} \Bigr).
        \end{split}
    \end{equation*}
    This concludes the proof.
\end{proof}

\subsection{Construction of quasi-Trefftz polynomials}\label{subsec:construction}

%\IF{For the Algorithm, maybe it is worth to mention the equations rewritten as \eqref{eq:QT conditions rewritten}. See the magenta part as a possibility. However, I am also fine not mentioning this if you prefer.}

According to Definition \ref{def:QT local space-divergence}, a QT function is a $(\pol{p})^3\times (\pol{p})^3$ polynomial satisfying the QT property of order $p$.
%The graded structure of polynomial spaces provides a fantastic basis for the construction of QT functions, in the vector case as in the scalar case. 
%Yet in the vector case the principal part of a QT operator is not necessarily surjective as in the scalar case.
%
In Algorithm \ref{alg:QT construction} we propose an explicit procedure to construct QT polynomials, relying on key ingredients:
\begin{itemize}
    \item the graded structure of polynomial spaces, %\modTC{which allow us to rewrite the QT property as in \eqref{eq:QT conditions rewritten}},
    \item the homogeneous polynomial Helmholtz decomposition. %\modTC{, which further splits the homogeneous components into harmonic, solenoidal, and irrotational ones}.
\end{itemize}
The former leads to the QT property expression in \eqref{eq:QT conditions rewritten}, while the latter leads to well-posed curl and divergence equations.
%In this subsection, we develop an explicit algorithm for constructing polynomials in $(\pol{p})^3\times (\pol{p})^3$ that satisfy the QT property defined in \eqref{eq:QT local space-divergence}. This algorithm works layer-by-layer on spaces of homogeneous polynomials with increasing degree using the natural decomposition of $(\pol{p})^3$ as a graded space \eqref{eq:graded space decomposition}. %and has $p+1$ steps.
% and it is based on the local QT property written as in \eqref{subeq:curl E}, \eqref{subeq:curl E} and \eqref{eq:divergence with separation}.
%To prove the main results of this section, we will work under the following assumption:

\begin{algorithm}[ht]
    \caption{
    %Construction of a QT polynomial
    $\left(\funel{Q},\funmg{Q}\right) = {\tt ConstructQTpol}\left(p,\bm{x}_T,f_\permitt, f_\permeab, (\funel{H},\funmg{H})\right)$
    }\label{alg:QT construction}
    %\lmig{the expilons and mus should also be input! and review how spaces are used below}\IF{I changed the description of the input adding the missing pieces and describing the coefficients as Taylor expansions rather than functions}

    {\bf Input:} {$p\in\mathbb N$,  %$\permitt, \permeab$ satisfying Assumption~\ref{ass:coefficients}}
    $\bm{x}_T\in\mathbb R^3$,\\
    \nonl\phantom{\bf Input:}
    $f_\permitt = \displaystyle\sum_{k=0}^p \permitt_k$ and $f_\permeab = \displaystyle\sum_{k=0}^p \permeab_k$, with $(\permitt_k, \permeab_k) \in\left(\hompol{k}\right)^2$  $\forall k\in\llbracket 0,p\rrbracket$ and $\varepsilon_0\mu_0\neq 0$, \\%and $\permitt_0. \permeab_0 \neq 0$;
    \nonl\phantom{\bf Input:}
    $\left(\funel{H},\funmg{H}\right) = \left(\displaystyle\sum_{k=0}^p \funel{H}_k,   \displaystyle\sum_{k=0}^p \funmg{H}_k \right)
    %\in \harm{p}\times \harm{p}
    $
    with $(\funel{H}_k,\funmg{H}_k )\in \homharm{k}\times \homharm{k}$ $\forall k\in\llbracket 0,p\rrbracket$
    }

    \For{$k\ {\rm from}\ 1\ {\rm to}\ p$}{
        {\bf find} 
        $\left(\funel{F}_k, \funmg{F}_k\right) \in \left(\homdivfree{k}\right)^2$ such that
        \begin{gather}
            \curl \funel{F}_k = {\color{white}-}\imag \kappa \,\Biggl(\, \displaystyle\sum_{j=0}^{k-2} \permeab_j \left(\funmg{F}_{k-j-1} + \funmg{G}_{k-j-1} + \funmg{H}_{k-j-1} \right) + \permeab_{k-1} \funmg{H}_0 \Biggr) \label{eq:alg 1}
            \\
            \curl \funmg{F}_k = - \imag \kappa \,\Biggl(\, \displaystyle\sum_{j=0}^{k-2} \permitt_j \left(\funel{F}_{k-j-1} +\funel{G}_{k-j-1} + \funel{H}_{k-j-1} \right) + \permitt_{k-1} \funel{H}_0 \Biggr) \label{eq:alg 2}
        \end{gather}

        {\bf find}
        $\left(\funel{G}_k, \funmg{G}_k \right)\in\left(\homcurlfree{k}\right)^2$ such that
        \begin{gather}
            \diver \funel{G}_k = - (\permitt_0)^{-1} \diver \Biggl(\, \displaystyle\sum_{j=1}^{k-1} \permitt_j \left(\funel{F}_{k-j} +\funel{G}_{k-j} + \funel{H}_{k-j} \right) + \permitt_k \funel{H}_0 \Biggr) \label{eq:alg 3}
            \\
            \diver \funmg{G}_k = - (\permeab_0)^{-1} \diver \Biggl(\, \displaystyle\sum_{j=1}^{k-1} \permeab_j \left(\funmg{F}_{k-j} +\funmg{G}_{k-j} + \funmg{H}_{k-j} \right) + \permeab_k \funmg{H}_0 \Biggr) \label{eq:alg 4}
        \end{gather}
    }
    
    {\bf Output: } {$\left(\funel{Q},\funmg{Q}\right) \in (\pol{p})^3 \times (\pol{p})^3$  where
    \\
    \nonl\phantom{\bf Output:}$\funel{Q} = \funel{H}_0 +  \displaystyle\sum_{k=1}^p \left(\funel{F}_k + \funel{G}_k + \funel{H}_k\right)$ and $\funmg{Q} = \funmg{H}_0 + \displaystyle\sum_{k=1}^p \left(\funmg{F}_k + \funmg{G}_k + \funmg{H}_k\right)$
    }
\end{algorithm}

Two natural questions arise about Algorithm \ref{alg:QT construction}, and will be answered in the next theorem.
First, does the proposed procedure work?
Indeed, between spaces of homogeneous polynomials, the divergence operator is surjective, so \eqref{eq:alg 3} and \eqref{eq:alg 4} always have a solution; but the curl operator is not surjective, so \eqref{eq:alg 1} and \eqref{eq:alg 2} might not always have a solution.
Second, is the output indeed a QT polynomial? The following result answers both questions. 
% \lmig{I don't understand: saying that the answer is the next theorem is not a repetition but rather an introduction}\IF{There is a bit of repetition here about the fact that we answer the questions in the next Theorem, but maybe it is fine since it is an important result.}

As a side note, it is obvious from \eqref{eq:alg 3}-\eqref{eq:alg 4} that the algorithm cannot be applied if $\varepsilon_0 \mu_0=0$, and this is how Assumption \ref{ass:coefficients} comes into play here.

\begin{theorem}[Construction of a quasi-Trefftz polynomial]\label{th:construction QT function} 
    Algorithm~\ref{alg:QT construction} is well-defined.
    Moreover, let $T\subset \mathbb R^3$ be an open set, $\bm{x}_T\in T$, $p\in\mathbb N$,  $\left(\funel{H}, \funmg{H}\right)\in \harm{p}\times \harm{p}$ be two harmonic polynomials, and consider Maxwell's equation \eqref{eq:first-order Maxwell} with coefficients {$(\permitt, \permeab)\in \left(C^{p}(T; \mathbb C)\right)^2$}. %\modTC{$\varepsilon, \mu\in C^{p}(T; \mathbb C)$}.
    %\lmig{I kept it coherent by reverting both occurences to the more precise statement}\IF{In the previous results, we used the magenta version. I would choose the magenta version only for coherence.}
    %missing epsilon and mu -> corrected
    Under Assumption~\ref{ass:coefficients}, 
    %if $\{(\epsilon_k,\mu_k), 0\leq k\leq p\}$ are the homogeneous components of $\left(\taylortrunsc{p}{\,\bm{x}_T} [\permitt],\taylortrunsc{p}{\,\bm{x}_T} [\permeab]\right)$,  
    the output $\left(\funel{Q},\funmg{Q}\right) = {\tt ConstructQTpol}\left(p,\bm{x}_T,\taylortrunsc{p}{\,\bm{x}_T} [\permitt], \taylortrunsc{p}{\,\bm{x}_T} [\permeab], (\funel{H},\funmg{H})\right)$ is a quasi-Trefftz polynomial: $(\funel{Q}, \funmg{Q})\in \QTopt{p}(T)$. %In particular, this is the unique quasi-Trefftz function in $\QTopt{p}(T)$ with $\funel{H}$ and $\funmg{H}$ as harmonic components for $\funel{Q}$ and $\funmg{Q}$, respectively.
\end{theorem}

% Alternative way of writing the algorithm
% \begin{algorithm}[ht]
%     \caption{Construction of QT functions}\label{alg:QT construction}
%     {\bf Input:} {$\funel{H} = \displaystyle\sum_{k=0}^p \funel{H}_k, \funmg{H} = \displaystyle\sum_{k=0}^p \funmg{H}_k \in \harm{p}$}

%     {\bf set} $\funel{Q}_0 = \funel{H}_0$ and $\funmg{Q}_0 = \funmg{H}_0$

%     \For{$k\ {\rm from}\ 1\ {\rm to}\ p$}{
%         {\bf find} 
%         $\funel{F}_k, \funmg{F}_k \in \homdivfree{k}$ such that
%         \begin{gather}
%             \curl \funel{F}_k = \imag \kappa \displaystyle\sum_{j=0}^{k-1} \permeab_j \funmg{Q}_{k-j-1} \label{eq:alg 1}
%             \\
%             \curl \funmg{F}_k = - \imag \kappa \displaystyle\sum_{j=0}^{k-1} \permitt_j \funel{Q}_{k-j-1} \label{eq:alg 2}
%         \end{gather}

%         {\bf find}
%         $\funel{G}_k, \funmg{G}_k \in\homcurlfree{k}$ such that
%         \begin{gather}
%             \diver \funel{G}_k = - (\permitt_0)^{-1} \diver \displaystyle\sum_{j=1}^{k} \permitt_j \funel{Q}_{k-j} \label{eq:alg 3}
%             \\
%             \diver \funmg{G}_k = - (\permeab_0)^{-1} \diver \displaystyle\sum_{j=1}^{k} \permeab_j \funmg{Q}_{k-j} \label{eq:alg 4}
%         \end{gather}

%         {\bf set} $\funel{Q}_k = \funel{F}_k + \funel{G}_k + \funel{H}_k$ and $\funmg{Q}_k = \funmg{F}_k + \funmg{G}_k + \funmg{H}_k$
%     }
    
%     {\bf Output: } $\funel{Q} = \displaystyle\sum_{k=0}^p \funel{Q}_k$, $\funmg{Q} = \displaystyle\sum_{k=0}^p \funmg{Q}_k \in \pol{p}$
% \end{algorithm}

\begin{proof} 
    First, the algorithm is well-defined since each algorithm equation from \eqref{eq:alg 1} to \eqref{eq:alg 4} admits a unique solution for any $k>0$: the existence and uniqueness are ensured by Proposition~\ref{prop:bijectivity operators}, the fact that $\permitt_0\permeab_0\neq 0$, and that, at each step $k\in\indices{1}{p}$, the right-hand sides of \eqref{eq:alg 1} and \eqref{eq:alg 2} both belong to $\homdivfreeall{k-1}$. 
    More precisely, regarding the latter, for $k=1$, we have $(\imag \kappa \permeab_0 \funmg{H}_0, -\imag \kappa \permitt_0 \funel{H}_0) \in %\homdivfreeall{0}\times\homdivfreeall{0} = 
    \left(\pol{0}\right)^3\times\left(\pol{0}\right)^3$, while, for $k\in\indices{2}{p}$, the right-hand sides automatically satisfy the divergence-free condition thanks to \eqref{eq:alg 3} and \eqref{eq:alg 4} satisfied at the previous step.

    Next, consider $(\funel{Q}, \funmg{Q})\in (\pol{p})^3\times (\pol{p})^3$, the output of the algorithm, to prove that it is indeed a QT polynomial by showing that it satisfies the four QT conditions defining $\QTopt{p}$. 
    In what follows, the Helmholtz decomposition comes into play repeatedly since, for all $k$ from $1$ to $p$, $(\funel{F}_k, \funel{G}_k,\funel{H}_k  )\in \homdivfree{k}\times \homcurlfree{k}\times \homharm{k}$ implies that $\curl\funel{G}_k = \bm{0}$, $\curl\funel{H}_k = \bm{0}$, $\diver\funel{F}_k = 0$ and $ \diver\funel{H}_k = {0}$. % As %the decomposition of the output, %the Taylor truncation polynomials defined in Assumption~\ref{ass:coefficients}, and 
    % the fact that $\curl\funel{G}_k = \curl\funel{H}_k = \bm{0}$ since $\funel{G}_k\in\homcurlfree{k}$ and $\funel{H}_k\in\homharm{k}$ for all $k\in\indices{1}{p}$, 
    In particular, we have
    \begin{equation*}
        \begin{split}
            \taylortrun{p-1}{\ \bm{x}_T} \left[\curl\funel{Q} - \imag\kappa \permeab \funmg{Q}\right] =&\ \taylortrun{p-1}{\ \bm{x}_T} \biggl[\curl\biggl(\sum_{k=0}^p \Bigl(\funel{F}_k + \funel{G}_k + \funel{H}_k \Bigr)\biggr) - \imag\kappa \permeab \sum_{k=0}^p \Bigl(\funmg{F}_k + \funmg{G}_k + \funmg{H}_k\Bigr) \biggr] \\
            =&\ \sum_{k=1}^p \curl \funel{F}_k - \imag\kappa \sum_{l=0}^{p-1} \sum_{j=0}^l \permeab_j \Bigl(\funel{F}_{l-j} + \funel{G}_{l-j} + \funel{H}_{l-j} \Bigr) \\
            =&\ \sum_{k=1}^p \biggl( \curl \funel{F}_k - \imag\kappa \sum_{j=0}^{k-1} \permeab_j \Bigl(\funel{F}_{k-j-1} + \funel{G}_{k-j-1} + \funel{H}_{k-j-1} \Bigr) \biggr).
        \end{split}
    \end{equation*}
    Then, this is equal to zero since \eqref{eq:alg 1} holds for all $k\in\indices{1}{p}$. Analogously, 
    \begin{equation*}
        \taylortrun{p-1}{\ \bm{x}_T} \left[\curl \funmg{Q} + \imag \kappa \permitt \funel{Q}\right] = \bm{0}
    \end{equation*}
    because \eqref{eq:alg 2} holds for all $k\in\indices{1}{p}$.
    Similarly, as $\permitt_0\neq 0$, we also have
    \begin{equation*}
        \begin{split}
            \taylortrunsc{p-1}{\,\bm{x}_T} \left[\diver \left(\permitt\funel{Q}\right)\right] =&\ \taylortrunsc{p-1}{\,\bm{x}_T} \biggl[\diver \biggl( \permitt \sum_{k=0}^p \left(\funel{F}_k + \funel{G}_k + \funel{H}_k\right) \biggr)\biggr] \\
            =&\ \sum_{k=1}^p \sum_{j=0}^k \diver \Bigl(\permitt_j \Bigl( \funel{F}_{k-j} + \funel{G}_{k-j} + \funel{H}_{k-j} \Bigr) \Bigr) \\
            =&\ \permitt_0 \sum_{k=1}^p \biggl( \diver \funel{G}_k + (\permitt_0)^{-1} \displaystyle\sum_{j=1}^k \diver \left(\permitt_j \left(\funel{F}_{k-j} +\funel{G}_{k-j} + \funel{H}_{k-j} \right)\right) \biggr) .
        \end{split}
    \end{equation*}
    Then, this is also equal to zero since \eqref{eq:alg 3} holds for all $k\in\indices{1}{p}$.
    Analogously, 
    \begin{equation*}
        \taylortrunsc{p-1}{\,\bm{x}_T} \left[\diver \left(\permeab\funmg{Q}\right)\right] = 0
    \end{equation*}
    follows from the facts that $\permeab_0\neq 0$ and that  \eqref{eq:alg 4} holds for all $k\in\indices{1}{p}$.
    This concludes the proof since $\left(\funel{Q}, \funmg{Q}\right)\in\QTopt{p}(T)$. %since it satisfies all four QT conditions.

%    \lmig{I don't understand the need for this paragraph} Finally, the uniqueness statement follows from the definition of the output and the uniqueness of the Helmholtz decomposition \eqref{eq:Helmholtz}. %and the fact that, by Proposition \ref{prop:surjectivity operators} and the fact that $\permitt_0\neq 0$ and $\permeab_0 \neq 0$, all the equations in Algorithm \ref{alg:QT construction} have a unique solution.
%    More in details, thanks to the Helmholtz decomposition \eqref{eq:Helmholtz}, a polynomial $\bm{Q}$ is uniquely defined by specifying for all $k\in\indices{0}{p}$ its harmonic components in $\homharm{k}$, and for all $k\in\indices{1}{p}$ its solenoidal components in $\homdivfree{k}$ and its irrotational components in $\homcurlfree{k}$. For the two polynomials $(\funel{Q}, \funmg{Q})\in (\pol{p})^3\times (\pol{p})^3$ constructed by Algorithm~\ref{alg:QT construction}, the harmonic components are fixed by the input of the algorithm, while the other components by equations \eqref{eq:alg 1}--\eqref{eq:alg 4}, each of which has a unique solution in the space $\homdivfree{k}$ or $\homcurlfree{k}$.
\end{proof}

To summarize, the procedure proposed in Algorithm \ref{alg:QT construction} is well-defined and under Assumption \ref{ass:coefficients} it produces a QT polynomial. 
While the surjectivity of the QT operator's principal part (as a reminder, this refers to the highest-order differential operator with frozen coefficients) is a fundamental building block for the construction of QT functions for scalar PDEs \cite{Imbert-Gerard_2025-2, FIG}, for a vector-valued PDE the QT operator's principal part is not always surjective. This is the case for Maxwell's equations: the curl operators are not surjective.
Instead, Algorithm \ref{alg:QT construction} strongly relies on the exactness of the polynomial de Rham complex. % \eqref{eq:De Rham complex}.
%\lmig{I think it is better to emphasize this in the abstract, introduction and/or conclusion rather than in a remark here}\IF{Maybe make the last two sentences a remark to highlight a novelty of this analysis?}

% \begin{remark}[Curl Operator] \lmig{this should not be here}
%     \IF{Maybe move this to the introduction and/or the conclusion}
%     The well-posedness of Algorithm~\ref{alg:QT construction} is not automatic, but it is ensured by the exactness of the polynomial De Rham complex \eqref{eq:De Rham complex}. This fact establishes a main difference between the scalar linear differential equation case and this vector-valued system of differential equations. In particular, it has recently been shown \cite{Imbert-Gerard_2025-2, Fontana_2025} that the surjectivity of the highest order differential operator with frozen coefficients of the equation is a fundamental brick for the construction of QT functions. While for scalar equations this is always true, see \cite[Proposition 4]{Imbert-Gerard_2025-2}, for the Maxwell system \eqref{eq:first-order Maxwell} the curl operator is not surjective on the entire polynomial space. However, the property $\text{Range}(\curl|_{\hompol{k}}) = \text{Range}(\curl|_{\homdivfree{}{k}}) = \ker (\diver) = \homdivfreeall{k}$ for all $k\in\mathbb N_0$ is sufficient to ensure existence and uniqueness of the output $(\funel{Q}, \funmg{Q})$.
% \end{remark}

\subsection{Dimension and bases}
Thanks to Algorithm \ref{alg:QT construction} and Theorem \ref{th:construction QT function}, 
we can now determine the dimension of the QT space and introduce the definition of a basis.
Practical aspects of the construction of a basis will be postponed to Section \ref{subsec:Implementation consideration}.

\begin{corollary}[Quasi-Trefftz space dimension]
\label{cor:dim}
    Let $T\subset \mathbb R^3$ be an open set, $\bm{x}_T\in T$, $p\in\mathbb N$, and consider Maxwell's equation \eqref{eq:first-order Maxwell} with coefficients {$(\permitt, \permeab)\in \left(C^{p}(T; \mathbb C)\right)^2$} %\modTC{$\varepsilon, \mu\in C^{p}(T; \mathbb C)$}.
    %\IF{In the previous results, we used the magenta version. I would choose the magenta version only for coherence.} 
    satisfying Assumption~\ref{ass:coefficients}.
    A QT polynomial is uniquely determined by its harmonic components.
    As a consequence,
    $$
    \dim \QTopt{p}(T) = 2(p+1)(p+3).
    $$
\end{corollary}
\begin{proof}
    The fact that for all $k\in\mathbb N$ and all $(\mathbf F,\mathbf G,\mathbf H)\in \homdivfree{k} \times \homcurlfree{k} \times \homharm{k}$
    $$ 
    \curl (\mathbf F+\mathbf G+\mathbf H) = \curl \mathbf F
    \qquad\text{ and }\qquad
    \diver (\mathbf F+\mathbf G+\mathbf H) = \diver \mathbf G$$
    shows that the set of equations \eqref{eq:alg 1}-\eqref{eq:alg 2}-\eqref{eq:alg 3}-\eqref{eq:alg 4} for $k$ from $1$ to $p$ is equivalent to the QT property. 
    Hence, from Algorithm \ref{alg:QT construction}, a QT polynomial is indeed uniquely determined by its harmonic components and
    $\dim \QTopt{p}(T) = \dim (\harm{p}\times \harm{p})$.
    This concludes the proof.
\end{proof}

A basis for the QT space $\QTopt{p}(T)$ can then be constructed from any basis for $\harm{p}\times\harm{p}$ as follows.
\begin{theorem}[Quasi-Trefftz basis]\label{thm:QTbasis}
    Let $T\subset \mathbb R^3$ be an open set, $\bm{x}_T\in T$, $p\in\mathbb N$, and consider Maxwell's equation \eqref{eq:first-order Maxwell} with coefficients {$(\permitt, \permeab)\in \left(C^{p}(T; \mathbb C)\right)^2$} % \modTC{$\varepsilon, \mu\in C^{p}(T; \mathbb C)$}.
    %\IF{In the previous results, we used the magenta version. I would choose the magenta version only for coherence.}   
    satisfying Assumption \ref{ass:coefficients}.
    Given any basis $\{
    \varphi_j%(\varphi_j,\phi_j) 
    \}_{j\in \llbracket 1, 2(p+1)(p+3)\rrbracket}$ of $\harm{p}\times\harm{p}$, then 
    \begin{equation}\label{eq:1stB}
        \left\{ 
        {\left(\funel{Q}_j, \funmg{Q}_j\right) = \tt ConstructQTpol}\left(p,\bm{x}_T,\taylortrunsc{p}{\,\bm{x}_T} [\permitt], \taylortrunsc{p}{\,\bm{x}_T} [\permeab], 
        \varphi_j%\varphi_j,\phi_j
        \right)
        \right\}_{j\in \llbracket 1, 2(p+1)(p+3)\rrbracket}
    \end{equation}
    forms a basis of $\QTopt{p}(T)$.
\end{theorem}
\begin{proof}
    First, the set defined in \eqref{eq:1stB} is linearly independent, for three reasons:
    \begin{itemize}
        \item the vector polynomial space can be decomposed as
    $\displaystyle
    \left(\left(\pol{p}\right)^3\right)^2  =  \left(\divfree{p}  \oplus  \curlfree{p} \right)^2 \oplus  \left(\harm{p}\right)^2 
    %,
    %\left(\pol{p}\right)^3  =  \left(\divfree{p}  \oplus  \curlfree{p} \right) \oplus  \harm{p} .
    %%\left(\pol{p}\right)^3 \times \left(\pol{p}\right)^3  =  \left(\divfree{p}  \oplus  \curlfree{p} \right)^2 \oplus  \left(\harm{p}\right)^2.
    $;
    \item $\varphi_j$ is the $\left(\harm{p}\right)^2$-component of $\left(\funel{Q}_j, \funmg{Q}_j\right)$;
    \item the set $\{
    \varphi_j%(\varphi_j,\phi_j) 
    \}_{j\in\indices{1}{2(p+1)(p+3)}}$ is linearly independent.
    \end{itemize}
    So the set has $2(p+1)(p+3)$ linearly independent elements. They all belong to $\QTopt{p}(T)$ from Theorem \ref{th:construction QT function}. According to Corollary \ref{cor:dim}, $2(p+1)(p+3)$ is also the dimension $\QTopt{p}(T)$. Hence the set indeed forms a basis for $\QTopt{p}(T)$.
\end{proof}

\begin{figure}[tb]
    \centering
    \begin{minipage}{0.40\textwidth}
        % This file was created with tikzplotlib v0.10.1.
\begin{tikzpicture}[scale=0.85]

\definecolor{darkgray176}{RGB}{176,176,176}
\definecolor{darkorange25512714}{RGB}{255,127,14}
\definecolor{forestgreen4416044}{RGB}{44,160,44}
\definecolor{steelblue31119180}{RGB}{31,119,180}

\begin{axis}[
legend style={
  fill opacity=.7,
  draw opacity=1,
  text opacity=1,
  at={(0.05,0.95)},
  anchor=north west,
  draw=white!10!black
},
legend cell align={left},
tick align=outside,
tick pos=left,
x grid style={darkgray176},
xlabel={$p$},
xmajorgrids,
xmin=0.3, xmax=15.7,
xtick style={color=black},
y grid style={darkgray176},
ylabel={Dimension},
ymajorgrids,
ymin=-228, ymax=5140,
ytick={0, 1000, 2000, 3000, 4000, 5000},
yticklabels={$0$, $10^3$, $2\cdot 10^3$, $3\cdot 10^3$, $4\cdot 10^3$, $5\cdot 10^3$},
ytick style={color=black}
]
\addplot [semithick, steelblue31119180, mark=*, mark size=2, mark options={solid}]
table {%
1 16
2 30
3 48
4 70
5 96
6 126
7 160
8 198
9 240
10 286
11 336
12 390
13 448
14 510
15 576
};
\addlegendentry{$\dim\QTopt{p}(T)$}
\addplot [semithick, darkorange25512714, mark=*, mark size=2, mark options={solid}]
table {%
1 18
2 36
3 60
4 90
5 126
6 168
7 216
8 270
9 330
10 396
11 468
12 546
13 630
14 720
15 816
};
\addlegendentry{$\dim\QT{p}(T)$}
\addplot [semithick, forestgreen4416044, mark=*, mark size=2, mark options={solid}]
table {%
1 24
2 60
3 120
4 210
5 336
6 504
7 720
8 990
9 1320
10 1716
11 2184
12 2730
13 3360
14 4080
15 4896
};
\addlegendentry{$\dim ((\pol{p})^3 \times (\pol{p})^3)$}
\end{axis}

\end{tikzpicture}%{Figures/dimensions_plot.tikz}
    \end{minipage}
    \hfill
    \begin{minipage}{0.45\textwidth}
        \vspace{-1.3cm}
        \begin{tabularx}{\textwidth}{cccc}
            \toprule
            $p$ & $\QTopt{p}(T)$ & $\QT{p}(T)$ & $(\pol{p})^3\times (\pol{p})^3$ \\
            \midrule
            1 & 16 & 18 & 24 \\
            2 & 30 & 36 & 60 \\
            3 & 48 & 60 & 120 \\
            4 & 70 & 90 & 210 \\
            5 & 96 & 126 & 336 \\
            10 & 286 & 396 & 1716 \\
            15 & 576 & 816 & 4896 \\
            \bottomrule
        \end{tabularx}
    \end{minipage}

    % \smallskip
    % \centering
    % \begin{tabularx}{\textwidth}{c*{15}{>{\centering\arraybackslash}X}}
    % \toprule
    % $p$ & 1 & 2 & 3 & 4 & 5 & 6 & 7 & 8 & 9 & 10 & 11 & 12 & 13 & 14 & 15 \\
    % \midrule
    % $\QT{p}(T)$ & 18 & 36 & 60 & 90 & 126 & 168 & 216 & 270 & 330 & 396 & 468 & 546 & 630 & 720 & 816 \\
    % $\QTopt{p}(T)$ & 16 & 30 & 48 & 70 & 96 & 126 & 160 & 198 & 240 & 286 & 336 & 390 & 448 & 510 & 576 \\
    % Pol & 24 & 60 & 120 & 210 & 336 & 504 & 720 & 990 & 1320 & 1716 & 2184 & 2730 & 3360 & 4080 & 4896 \\
    % \bottomrule
    % \end{tabularx}
    
    \caption{Space dimensions for $p\in\indices{1}{15}$: $\dim\QTopt{p}(T) = 2(p+1)(p+3)$ (Corollary~\ref{cor:dim}), $\dim \QT{p}(T) = 3(p+1)(p+2)$ (Corollary~\ref{th:space dimension}), and $\dim\left((\pol{p})^3 \times (\pol{p})^3\right) = (p+1)(p+2)(p+3)$.}
    \label{fig:dimensions}
\end{figure}

% \begin{table}[tb]
%     \centering
%     \begin{tabularx}{\textwidth}{c*{15}{>{\centering\arraybackslash}X}}
%     \toprule
%     $p$ & 1 & 2 & 3 & 4 & 5 & 6 & 7 & 8 & 9 & 10 & 11 & 12 & 13 & 14 & 15 \\
%     \midrule
%     $\QT{p}(T)$ & 18 & 36 & 60 & 90 & 126 & 168 & 216 & 270 & 330 & 396 & 468 & 546 & 630 & 720 & 816 \\
%     $\QTopt{p}(T)$ & 16 & 30 & 48 & 70 & 96 & 126 & 160 & 198 & 240 & 286 & 336 & 390 & 448 & 510 & 576 \\
%     Pol & 24 & 60 & 120 & 210 & 336 & 504 & 720 & 990 & 1320 & 1716 & 2184 & 2730 & 3360 & 4080 & 4896 \\
%     \bottomrule
%     \end{tabularx}
%     \caption{Space dimensions for $p \in \indices{1}{15}$.}
%     \label{tab:dimensions}
% \end{table}

Figure \ref{fig:dimensions} compares the dimensions of the quasi-Treffz space $\QTopt{p}(T)$  and the classical full polynomial space $(\pol{p})^3\times (\pol{p})^3$, illustrating the significant reduction in the QT case. 
%For reference, the dimension of the naive QT space $\QT{p}(T)$ is also depicted (see Definition \ref{eq:QT local space} and Corollary~\ref{th:space dimension}). % for a proof that this dimension is $3(p+1)(p+2)$.
It illustrates the significant reduction in dimension in the high-order regime, as $\dim \QTopt{p}(T) = O(p^2)$ %and $\dim \QT{p}(T) = O(p^2)$, 
while $\dim \left((\pol{p})^3\times (\pol{p})^3\right) = O(p^3)$.  
Comparison with the naive QT space $\QT{p}(T)$, see Definition \ref{eq:QT local space}, is provided for reference.

The next section turns to the dimension optimality of $\QTopt{p}(T)$.
%In addition, $\dim \QTopt{p}(T) < \dim\QT{p}(T)$ for all $p\in\mathbb N$.

% \begin{corollary}[Construction of a QT basis]
%     Let $\bm{x}_T\in T$, $p\in\mathbb N_0$ and let Assumption~\ref{ass:coefficients} hold.
%     Then, a basis of the quasi Trefftz space $\QTopt{p}(T)$ \eqref{eq:QT local space-divergence} can be constructed by choosing a basis of $\harm{p} \times \harm{p}$ and applying Algorithm~\ref{alg:QT construction} to each element of this basis.
%     % \IF{TBC}
% \end{corollary}

\section{Optimality
%of QT Spaces
}\label{sec:optimal properties}

%\lmig{instead explain that we stand in between standard pol spaces and the PDE-only qt condition}

To justify our choice in the order of the quasi-divergence conditions, consider the best approximation property discussed in Section \ref{ssec:APPProp}.
As already discussed, the best approximation property of the QT space $\QTopt{p}(T)$ relies on the fact that a certain Taylor polynomial of all smooth exact solutions to the PDE belongs to the space. 
The same Taylor polynomial of all exact solutions to the PDE clearly belongs to the standard polynomial space $(\pol{p})^3\times (\pol{p})^3$, and therefore the standard polynomial space enjoys the same best approximation property. %; but the Taylor polynomial of the same order of any smooth function also belongs to $(\pol{p})^3\times (\pol{p})^3$.
In summary, compared to standard polynomial space, the QT space can approximate only smooth solutions to the governing PDE, rather than all smooth functions, with the same order of convergence, but it does so more efficiently, as it requires fewer degrees of freedom.

% In the context of approximation of smooth solutions of a PDE, what is the smallest possible subspace of $(\pol{p})^3\times (\pol{p})^3$ that enjoys the same best approximation property?
We will now see that $\QTopt{p}(T)$ is the smallest possible subspace of $(\pol{p})^3\times (\pol{p})^3$ that enjoys the same best approximation property. In that sense, $\QTopt{p}(T)$ has an optimal dimension, and this justifies its definition.

 \subsection{Naive quasi-Trefftz space}
 For comparison, consider the naive QT space defined as follows.
 
\begin{definition}[Naive quasi-Trefftz space] %\label{def:naiveQT local space-divergence}
    Given a positive integer $p$, an open set $T\subset\Omega\subset\mathbb R^3$, a point $\bm{x}_T\in T$, assuming that $(\permitt, \permeab)\in \left(C^{p-1}(T; \mathbb C)\right)^2$, consider the first-order system \eqref{eq:first-order Maxwell}.
    The {\em naive local polynomial quasi-Trefftz space of order $p$} is
\begin{equation}\label{eq:QT local space}
    \QT{p}(T) \coloneq \left\{\left(\funel{Q}, \funmg{Q}\right) \in (\pol{p})^3\times (\pol{p})^3 \ :\ 
    \begin{aligned}
        &\taylortrun{p-1}{\ \bm{x}_T} \left[\curl \funel{Q} - \imag \kappa \permeab \funmg{Q}\right] = \bm{0} \\
        &\taylortrun{p-1}{\ \bm{x}_T} \left[\curl \funmg{Q} + \imag \kappa \permitt \funel{Q}\right] = \bm{0}
    \end{aligned}
    \right\}.
\end{equation}
\end{definition}
This definition is naive since it does not explicitly impose quasi-divergence conditions; instead, it only includes quasi-PDE conditions. However, following Proposition~\ref{prop:divergence+taylor}-\ref{it:Taylor-divergence condition}, implicit divergence conditions follow from the quasi-PDE conditions, and in fact the space can also be described as follows:
    \begin{equation*}
        \QT{p}(T) = \left\{\left(\funel{Q}, \funmg{Q}\right) \in (\pol{p})^3\times (\pol{p})^3 \ :\ 
        \begin{aligned}
            &\taylortrun{p-1}{\ \bm{x}_T} \left[\curl \funel{Q} - \imag \kappa \permeab \funmg{Q}\right] = \bm{0} \\
            &\taylortrun{p-1}{\ \bm{x}_T} \left[\curl \funmg{Q} + \imag \kappa \permitt \funel{Q}\right] = \bm{0} \\
            &\taylortrunsc{p-2}{\,\bm{x}_T} \left[\diver (\permitt\funel{Q})\right] = 0 \\
            &\taylortrunsc{p-2}{\,\bm{x}_T} \left[\diver (\permeab\funmg{Q})\right] = 0
        \end{aligned}
        \right\}.
    \end{equation*}
    Compared to the QT space $\QTopt{p}(T)$, the only difference is the order of the quasi-divergence conditions.

 From the definitions, it is obvious that 
$$
\QTopt{p}(T)\subset \QT{p}(T).
$$
The following theorem shows that the QT space $\QT{p}(T)$ 
also has the same approximation property as $\QTopt{p}(T)$, although the theorem assumes less regularity for the PDE coefficients, i.e., $\varepsilon, \mu\in C^{p-1}(T; \mathbb C)$. 

\begin{theorem}[Approximation properties of $\QT{p}(T)$]\label{th:approximation properties QT}
    Let $T\subset \mathbb R^3$ be an open set, $\bm{x}_T\in T$, $p\in\mathbb N$, %and $\QT{p}$ , 
    and suppose that $\varepsilon, \mu\in C^{p-1}(T; \mathbb C)$. If $(\el, \mg) \in C^{p+1}(T; \mathbb C^3) \times C^{p+1}(T; \mathbb C^3)$ satisfies the Maxwell system \eqref{eq:first-order Maxwell}, then 
    \begin{enumerate}[(i)]
        \item $\left(\taylortrun{p}{\ \bm{x}_T} [\el], \taylortrun{p}{\ \bm{x}_T} [\mg]\right) \in \QT{p}(T)$; \label{th-it:approximation properties QT 1}
        \item under the assumption that $T$ is star-shaped with respect to $\bm{x}_T$, \label{th-it:approximation properties QT 2}
        \begin{equation*}
            \begin{split}
                \inf_{(\funel{Q}, \funmg{Q})\in\QT{p}(T)} \Bigl(\left\lvert \el - \funel{Q}\right\rvert_{C^s(T)} &+ \left\lvert \mg - \funmg{Q} \right\rvert_{C^s(T)}\Bigr)  \\
                % \leq&\ \lvert \el - \taylortrun{p}{\,\,\bm{x}_T} [\el] \rvert_{C^s(T)} + \lvert \mg - \taylortrun{p}{\,\,\bm{x}_T} [\mg] \rvert_{C^s(T)} \\
                \leq&\ \frac{3^{p+1-s}}{(p+1-s)!} h_T^{p+1-s} \Bigl( \left\lvert\el\right\rvert_{C^s(T)} + \left\lvert\mg\right\rvert_{C^s(T)} \Bigr)
            \end{split}
        \end{equation*}
        for all $s\in\mathbb N_0$ with $s\leq p$.
    \end{enumerate}
\end{theorem}

\begin{proof}
    (i) By the definition of the two QT spaces % spaces $\QT{p}(T)$ and $\QTopt{p}(T)$ \
    and by Theorem~\ref{th:approximation properties QTopt} we have
    \begin{equation*}
        \left(\taylortrun{p}{\ \bm{x}_T} [\el], \taylortrun{p}{\ \bm{x}_T} [\mg]\right) \in \QTopt{p}(T) \subset \QT{p}(T).
    \end{equation*}

    \noindent
    (ii) The approach is the same as the proof of Theorem~\ref{th:approximation properties QTopt}.    
    %An immediate consequence of \ref{th-it:approximation properties QT 1} is
    % \begin{equation*}
    %     \begin{split}
    %         \inf_{(\funel{Q}, \funmg{Q})\in\QT{p}} \Bigl(\lvert \el - \funel{Q}\rvert_{C^s(T)} &+ \lvert \mg - \funmg{Q} \rvert_{C^s(T)}\Bigr)  \\
    %         \leq&\ \lvert \el - \taylortrun{p}{\,\,\bm{x}_T} [\el] \rvert_{C^s(T)} + \lvert \mg - \taylortrun{p}{\,\,\bm{x}_T} [\mg] \rvert_{C^s(T)}
    %     \end{split}
    % \end{equation*}

    % \IF{To complete: Convergence rate.}
\end{proof}
Hence $\QT{p}(T)$ and $\QTopt{p}(T)$ enjoy the same best approximation properties; yet they have different dimensions, as stated by the next theorem and illustrated in Figures~\ref{fig:dimensions}.

\begin{lemma}[Naive quasi-Trefftz  space dimension]\label{th:space dimension}
    Let $T\subset \mathbb R^3$ be an open set, $\bm{x}_T\in T$, $p\in\mathbb N$, suppose that $\varepsilon, \mu\in C^{p-1}(T; \mathbb C)$ satisfy Assumption~\ref{ass:coefficients}. The dimension of the naive quasi-Trefftz space $\QT{p}(T)$ is
    \begin{equation*}
        \dim \QT{p}(T) = 3(p+1)(p+2).
    \end{equation*}
    In particular, $\dim \QT{p}(T) > \dim \QTopt{p}(T)$ for all $p\in\mathbb N$.
\end{lemma}

\begin{proof}
    Algorithm~\ref{alg:QT construction} can be modified to construct a QT function in $\QT{p}(T)$ as follows: the component $(\funel{G}_p, \funmg{G}_p)\in\homcurlfree{p}\times\homcurlfree{p}$ is fixed as input of the algorithm, and Equations \eqref{eq:alg 3}-\eqref{eq:alg 4} are removed for $k=p$. As a consequence, 
    \begin{equation*}
        \begin{split}
            \dim \QT{p}(T) =&\ \dim \QTopt{p} + 2 \dim\homcurlfree{p}
            % = 2\sum_{k=0}^p (2k+3) + p(p+1) \\ 
            % =&\ 2 p(p+1) + 6(p+1) + p(p+1) = (3p+6) (p+1) = 3 (p+1)(p+2).
            = 2(p+1)(p+3) + p(p+1) = 3(p+1)(p+2).
        \end{split}
    \end{equation*}
    This implies that $\dim \QT{p}(T) > \dim \QTopt{p}(T)$ for all $p\in\mathbb N$.
\end{proof}

So the naive space, even though it enjoys the same best approximation properties as standard polynomial space, is not the smallest possible space to do so.

\subsection{Minimal dimension}

% Note that, while both the curl $\curl$ and the divergence $\diver$ are first-order partial differential operators, but the order of the Taylor truncation in the QT condition defining $\QT{p}$  differs from the one in the divergence conditions \eqref{eq:divergence conditions}. Indeed, they are $p-1$ and $p-2$ respectively.
% Another possible choice for the definition of a QT space for the first-order Maxwell's system, inspired by \cite{Imbert-Gerard:25-Maxwell}, is

The QT space $\QTopt{p}(T)$ is smaller than $\QT{p}(T)$, because of a stronger QT property.  
However, Theorems~\ref{th:approximation properties QT} and \ref{th:approximation properties QTopt} show that the best approximation properties are preserved in the smaller space, since the $h$-convergence rates are the same. In particular, these approximation properties hold as soon as $(\taylortrun{p}{\ \bm{x}_T} [\el], \taylortrun{p}{\ \bm{x}_T} [\mg])$ is an element of the finite-dimensional space.
Consequently, a natural question would be: can we do better? Or, in other words, can we define a smaller vector space with the same best approximation properties? 
The following Theorem demonstrates that the answer is no, since $\QTopt{p}(T)$ has the optimal dimension among spaces of the general form
    \begin{equation*}
        \mathbb V_{p, q_1, q_2, q_3, q_4} \coloneqq \left\{\left(\funel{Q}, \funmg{Q}\right) \in (\pol{p})^3\times (\pol{p})^3 \ :\ 
        \begin{aligned}
            &\taylortrun{q_1}{\ \bm{x}_T} \left[\curl \funel{Q} - \imag \kappa \permeab \funmg{Q}\right] = \bm{0} \\
            &\taylortrun{q_2}{\ \bm{x}_T} \left[\curl \funmg{Q} + \imag \kappa \permitt \funel{Q}\right] = \bm{0} \\
            &\taylortrunsc{q_3}{\,\bm{x}_T} \left[\diver (\permitt\funel{Q})\right] = 0 \\
            &\taylortrunsc{q_4}{\,\bm{x}_T} \left[\diver (\permeab\funmg{Q})\right] = 0
        \end{aligned}
        \right\}
    \end{equation*}
    that share the same best approximation properties.

\begin{theorem}[Minimal space]
    Given $p\in\mathbb N$, an open set $T\subset\mathbb R^3$, and a point $\bm{x}_T\in T$, suppose that $\varepsilon, \mu\in C^{\infty}(T; \mathbb C)$ and that Assumption~\ref{ass:coefficients} holds.  Then, the dimension of the quasi-Trefftz space $\QTopt{p}(T)$ is optimal, in the sense that it is the smallest $\mathbb V_{p, q_1, q_2, q_3, q_4}$ space with the best approximation properties defined by Theorem~\ref{th:approximation properties QTopt}.
\end{theorem}

\begin{proof}
    The QT space can be described as
    $\QTopt{p}(T) = \mathbb V_{p, p-1, p-1, p-1, p-1}$. %As particular cases of $\mathbb V_{p, q_1, q_2, q_3, q_4}$ spaces, %\lmig{what is the need to talk about the naive space here?} the naive spaces $\QT{p}$ are characterized by $q_1 = q_2 = p-1$ and $q_3 = q_4 = p-2$, while 
    %%\QTopt{p}$ are characterized by $q_1 = q_2 = q_3 = q_4 = p-1$. Furthermore, choosing any $q_j < p-1$, $j\in\indices{1}{4}$ will result in a space bigger than $\QTopt{p}$.
    % \begin{itemize}
    %     \item by Proposition~\ref{prop:divergence+taylor}-\ref{it:Taylor-divergence condition}, we automatically have $q_3 \geq q_2-1$ and $q_4 \geq q_1-1$;
    %     \item $\QT{p}$ is characterized by $q_1 = q_2 = p-1$ and $q_3 = q_4 = p-2$, while $\QTopt{p}$ by $q_1 = q_2 = q_3 = q_4 = p-1$;
    %     \item choosing any $q_j < p-1$, $j\in\indices{1}{4}$ will result in a space bigger than $\QTopt{p}$.
    % \end{itemize}
    Next, we show that if $q_j > p-1$ for some $j\in\indices{1}{4}$ then $(\taylortrun{p}{\ \bm{x}_T} [\el], \taylortrun{p}{\ \bm{x}_T} [\mg]) \notin \mathbb V_{p, q_1, q_2, q_3, q_4}$.
    Assume that $(\el, \mg) \in C^{p+1}(T; \mathbb C^3) \times C^{p+1}(T; \mathbb C^3)$ satisfies the Maxwell system \eqref{eq:first-order Maxwell}, and 
    fix $\bm{i}\in\mathbb N_0^3$ such that $\lvert\bm{i}\rvert = p$. 
    \begin{itemize}
        \item If $q_1 > p-1$, then Propositions~\ref{prop:divergence+taylor}-\ref{it:curl-Taylor action} and \ref{prop:multi-index derivative}-\ref{it:derivative vector Taylor} together with the Leibniz product rule give
    \begin{equation*}
        \begin{split}
            D^{\bm{i}} 
            \left(\curl \taylortrun{p}{\ \bm{x}_T} [\el] - \imag \kappa \permeab \taylortrun{p}{\ \bm{x}_T} [\mg]\right) (\bm{x}_T) =&\ D^{\bm{i}} \curl \taylortrun{p}{\ \bm{x}_T} [\el] (\bm{x}_T) - \imag \kappa D^{\bm{i}} \left(\permeab \taylortrun{p}{\ \bm{x}_T} [\mg]\right) (\bm{x}_T) \\
            =&\ -\imag \kappa D^{\bm{i}} (\permeab \mg) (\bm{x}_T)
            \neq \bm{0} \text{ in general}.
        \end{split}
    \end{equation*}
        \item If $q_2 > p-1$, then similarly
    \begin{equation*}
            D^{\bm{i}} \left(\curl \taylortrun{p}{\ \bm{x}_T} [\mg] + \imag \kappa \permitt \taylortrun{p}{\ \bm{x}_T} [\el]\right) (\bm{x}_T) = \imag \kappa D^{\bm{i}} (\permitt \el) (\bm{x}_T) \neq \bm{0} \text{ in general}.
    \end{equation*}
        \item If $q_3 > p-1$, then 
        %\lmig{this does not bring anything to the argument} proceeding as in the proof of Theorem~\ref{th:approximation properties QTopt}, 
        Proposition \ref{prop:multi-index derivative}-\ref{it:derivative scalar Taylor} and the Leibniz product rule give
    \begin{equation*}
        \begin{split}
            D^{\bm{i}} \left(\diver \left(\permitt \taylortrun{p}{\ \bm{x}_T}[\el] \right)\right) (\bm{x}_T) 
            =&\ D^{\bm{i}} \left(\grad\permitt \cdot \taylortrun{p}{\ \bm{x}_T}[\el]\right) (\bm{x}_T) + D^{\bm{i}} \left(\permitt \diver\taylortrun{p}{\ \bm{x}_T} [\el]\right) (\bm{x}_T) \\
            =&\ D^{\bm{i}} \left(\grad\permitt \cdot \el\right) (\bm{x}_T) + \sum_{\substack{\bm{m} \leq \bm{i},\\ \bm{m}\neq \bm{0}}} \binom{\bm{i}}{\bm{m}} D^{\bm{m}} \permitt(\bm{x}_T) D^{\bm{i}-\bm{m}} \diver \el (\bm{x}_T) \\
            =&\ D^{\bm{i}} (\diver (\permitt \el )) (\bm{x}_T) - \permitt (\bm{x}_T) D^{\bm{i}} \diver \el (\bm{x}_T) \\
            =&\ - \permitt (\bm{x}_T) D^{\bm{i}} \diver \el (\bm{x}_T) \neq 0 \text{ in general}.
        \end{split}
    \end{equation*}
        \item If $q_4 > p-1$, then similarly 
    \begin{equation*}
        D^{\bm{i}} \left(\diver \left(\permeab \taylortrun{p}{\ \bm{x}_T}[\mg] \right)\right) (\bm{x}_T) = - \permeab (\bm{x}_T) D^{\bm{i}} \diver \mg (\bm{x}_T) \neq 0 \text{ in general}.
    \end{equation*}
    \end{itemize}
    As a consequence, the best approximation properties do not hold for all solutions to Maxwell's system \eqref{eq:first-order Maxwell} if $q_j > p-1$ for any index $j\in\indices{1}{4}$.
\end{proof}

\section{Computational aspects}\label{subsec:Implementation consideration}

Practical aspects of the construction of QT bases are fundamental in the context of QT-DG methods. Indeed, a local basis needs to be constructed on each element of the computational mesh. 
%It is therefore particularly important to take into account practical aspects of the corresponding computations in order to limit the computational complexity . 
%
As in previous works, QT functions can naturally be constructed homogeneous components by homogeneous component, in increasing degree. This requires the resolution of polynomial differential equations, namely solutions to two curl equations and two divergence equations for each value of the degree $k$ from $1$ to $p$. It also requires a basis of harmonic vector fields as illustrated in Theorem \ref{thm:QTbasis}.

In this section, we propose two different approaches to construct individual QT functions. The first one implements Algorithm~\ref{alg:QT construction}, but requires numerical solves and the construction of a basis for $\homdivfree{k}$ and $\homcurlfree{k}$ for all $k$ from 1 to $p$ by two Singular Value Decompositions (SVDs) at each step. 
As a more computationally efficient alternative, the second one proposes a variation of that algorithm but can be implemented fully explicitly.
%\IF{As a comment later in the numerical section, in the current implementation Algorithm 2 also involves some numerical solver since the action of $\curl$ and $\diver$ is ``easy'' to implement, while the inverse is a bit more complex. Noted this, would you keep the wording ``fully explicit"? I think it is still applicable since the bases of $\homdivfreebis{k}$ and $\homcurlfreebis{k}$ are explicit, and I believe a fully explicit implementation, even if complex, would be possible.
%A side note: I checked online, and in python it is more efficient to solve two linear systems with the same basis, rather than computing the inverse matrix and then performing two matrix-vector multiplications.}
Additionally, the section ends with a discussion on the construction of full QT bases focusing on harmonic components, as these play a critical role in both approaches.

\subsection{First implementation}

In Algorithm~\ref{alg:QT construction}, solutions to the divergence and curl equations are sought, respectively, in $\homcurlfree{k}$ and $\homdivfree{k}$. 
Given bases for these spaces, one can simply assemble matrices for the operators $\diver|_{\homcurlfree{k}}$ and $\curl|_{\homdivfree{k}}$, which would be invertible, and simply solve a linear system numerically to construct the desired solutions.
To our knowledge, no explicit bases for these spaces can be found in the literature.

However, since
$$
    \begin{cases}
        \displaystyle (\hompol{k})^3 = \homdivfreeall{k} \oplus \homcurlfree{k},\\
        \displaystyle \homdivfreeall{k} = \ker \left( \diver|_{(\hompol{k})^3} \right),
    \end{cases}
    \text{ and }\qquad
    \begin{cases}
        \displaystyle (\hompol{k})^3 = \homcurlfreeall{k} \oplus \homdivfree{k},\\
        \displaystyle \homcurlfreeall{k} = \ker \left( \curl|_{(\hompol{k})^3} \right),
    \end{cases}
$$
 bases for $\homcurlfree{k}$ and $\homdivfree{k}$ can be computed, respectively, from the SVD of any matrix of the operators $\diver|_{(\hompol{k})^3}$ and $\curl|_{(\hompol{k})^3}$, as the set of right-singular vectors associated to non-zero singular values.

 \begin{remark}\label{rem:HSI independence of x0}
    %\lmig{Ilaria: are you satisfied with this remark?}
     While all spaces of homogeneous polynomials in this work depend on $\bm x_0$, see Remark \ref{rmk:indepcenter}, using canonical bases centered at $\bm x_0$ can evidence properties that are independent of $\bm x_0$, and these can be exploited for efficient computation in the context of a QT-DG method.
     Indeed, matrices of the differential operators $\diver|_{(\hompol{k})^3}$ and $\curl|_{(\hompol{k})^3}$ in canonical bases $\{(\bm X-\bm x_0)^{\bm i},|\bm i|=k\}$ are themselves independent of $\bm x_0$. For example
     $$
     \diver
     \begin{pmatrix}
     (\bm{X}-\bm{x}_0)^{4\bm e_1}\\
     \bm 0\\
     \bm 0
     \end{pmatrix}
     = 
     4(\bm{X}-\bm{x}_0)^{3\bm e_1}
     \qquad\text{ and }\qquad
     \diver 
     \begin{pmatrix}
     \bm 0\\
     (\bm{X}-\bm{x}_0)^{\bm e_1+3\bm e_2}\\
     \bm 0
     \end{pmatrix}
     = 
     3(\bm{X}-\bm{x}_0)^{\bm e_1+2\bm e_2}.
     $$
     Therefore, in order to construct QT bases at the centroid of each cell throughout a computational domain, computing only two SVDs per value of $k$ is sufficient independently of the number of mesh elements. 
 \end{remark}
 
% for $k\in\mathbb N$

% Given $\bm R\in(\hompol{k-1})^3$
%       find  $\bm{G} \in \homcurlfree{k}$ such that
% $\diver \bm{G} =\bm R$

% Given $f\in\hompol{k-1}\cap\mathcal R(\curl)$
%       find  $\bm{F} \in \homdivfree{k}$ such that
% $\curl {F} =f$

\subsection{Alternative implementation}
In order for a polynomial vector field to satisfy the QT property, it is actually not necessary to compute the $\homcurlfree{k}$- and $\homdivfree{k}$-components of the Helmholtz decomposition. Instead, it is sufficient to find any components $\left(\funel{F}_k, \funmg{F}_k\right) \in \left(\hompol{k}\right)^3 \times \left(\hompol{k}\right)^3$ and $\left(\funel{G}_k, \funmg{G}_k \right) \in \left(\hompol{k}\right)^3 \times \left(\hompol{k}\right)^3$ that satisfy, respectively, the curl and divergence equations.
This can be achieved explicitly via the definition of subspaces $\homdivfreebis{k}\subset \homdivfreeall{k}$ and $\homcurlfreebis{k}\subset\homcurlfreeall{k}$ that are, respectively, in bijection with 
$\mathcal R(\diver|_{\hompol{k}^3})$ and $\mathcal R(\curl|_{\hompol{k}^3})$ 
%\lmig{no}\IF{Shouldn't be kernel here and not range?},
and for which explicit bases are known. Indeed, then any matrices of the differential operators $\diver|_{\homdivfreebis{k}}$ and $\curl|_{\homcurlfreebis{k}}$ are invertible, they can be assembled explicitly, and solutions to the divergence and curl equations can be found respectively in $\homdivfreebis{k}$ and $\homcurlfreebis{k}$ without any SVD computation. 
{Moreover, the corresponding linear systems are trivial in relevant bases, so solutions can be found explicitly without any numerical solve and then easily expressed in canonical bases via explicit changes of bases.}
%\lmig{you chose not to do it in your code, but it is true that it can be done!}\IF{As the comment above, in practice there is a linear solver involved.}

Follows the definition of such sets of subspaces of homogeneous curl-free and divergence-free polynomials.

\begin{definition}[Curl-free and divergence-free functions]\label{def:curl-free and divergence-free functions}
    Given $k\geq 1$ and a point $\bm{x}_0\in\mathbb R^3$, consider the sets of curl-free, $\bm{\Phi}_{k, \bm{i}}$, and divergence-free, $\bm{\Psi}_{k,j, \bm{i}}$, functions defined as follows:
\begin{equation*}
    \bm{\Phi}_{k, \bm{i}} \coloneqq \begin{pmatrix}
        i_1 (\bm{X}-\bm{x}_0)^{\bm{i}-\bm{e}_1} \\
        i_2 (\bm{X}-\bm{x}_0)^{\bm{i}-\bm{e}_2} \\
        i_3 (\bm{X}-\bm{x}_0)^{\bm{i}-\bm{e}_3}
    \end{pmatrix}
    \qquad
    \forall \bm{i}\in\mathbb N_0^3, \lvert\bm{i}\rvert = k+1.
\end{equation*}
\begin{equation*}
    \bm{\Psi}_{k, 1, \bm{i}} \coloneqq \begin{pmatrix}
        (\bm{X}-\bm{x}_0)^{\bm{i}} \\ 0 \\ 0
    \end{pmatrix}
    \qquad\begin{matrix}
        \forall \bm{i}\in\mathbb N_0^3,\\ \lvert\bm{i}\rvert = k, %i_1=0,
    \end{matrix}
    \qquad\qquad\qquad
    \bm{\Psi}_{k, 2, \bm{i}} \coloneqq \begin{pmatrix}
        0 \\ (\bm{X}-\bm{x}_0)^{\bm{i}} \\ 0 
    \end{pmatrix}
    \qquad\begin{matrix}
        \forall \bm{i}\in\mathbb N_0^3,\\ \lvert\bm{i}\rvert = k, %i_2=0,
    \end{matrix} 
\end{equation*}
\begin{equation*}
    \bm{\Psi}_{k, 3, \bm{i}} \coloneqq \begin{pmatrix}
        0 \\ 0 \\ (\bm{X}-\bm{x}_0)^{\bm{i}}
    \end{pmatrix}
    \qquad\begin{matrix}
        \forall \bm{i}\in\mathbb N_0^3,\\ \lvert\bm{i}\rvert = k, %i_3=0,
    \end{matrix}   
\end{equation*}
\begin{equation*}
    \bm{\Psi}_{k, 4, \bm{i}} \coloneqq \begin{pmatrix}
        (i_3+1) (\bm{X}-\bm{x}_0)^{\bm{i}} \\ 0 \\ -i_1 (\bm{X}-\bm{x}_0)^{\bm{i}-\bm{e}_1+\bm{e}_3}
    \end{pmatrix}
    \ \begin{matrix}
        \forall \bm{i}\in\mathbb N_0^3,\\ \lvert\bm{i}\rvert = k, i_1\geq 1,
    \end{matrix}
    \quad
    \bm{\Psi}_{k, 5, \bm{i}} \coloneqq \begin{pmatrix}
        0 \\ (i_3+1) (\bm{X}-\bm{x}_0)^{\bm{i}} \\ -i_2 (\bm{X}-\bm{x}_0)^{\bm{i}-\bm{e}_2+\bm{e}_3}
    \end{pmatrix}
    \ \begin{matrix}
        \forall \bm{i}\in\mathbb N_0^3,\\ \lvert\bm{i}\rvert = k, i_2\geq 1,
    \end{matrix}
\end{equation*}
\end{definition}

\begin{remark}[Zero-components of $\bm{\Phi}_{k, \bm{i}}$]
    A component of the basis element $\bm{\Phi}_{k, \bm{i}}$ can be 0. This happens when the corresponding coefficient is equal to zero: the $l$-th component is 0 if and only if $i_l = 0$, $l\in\{1, 2, 3\}$. However, for all multi-indices $\bm{i}$ there is at least one nonzero component since $|\bm{i} |= k+1$.
\end{remark}

While subsets of these form bases, respectively, for the sets of homogeneous curl-free and divergence-free polynomials, more precisely:
\begin{equation*}
    \left\{ \bm{\Phi}_{k, \bm{i}} ,{\bm{i}\in\mathbb N_0^3, \lvert\bm{i}\rvert = k+1} \right\} \text{ forms a basis for }\homcurlfreeall{k},\text{ and }
\end{equation*}
\begin{equation*}
\begin{split}
    \cup_{{\bm{i}\in\mathbb N_0^3, \lvert\bm{i}
    \rvert = k}}\Big\{ &
    \left\{\bm{\Psi}_{k,1, \bm{i}} ,%{\bm{i}\in\mathbb N_0^3, \lvert\bm{i}\rvert = k},
    i_1=0\right\}
    \cup
    \left\{\bm{\Psi}_{k,2, \bm{i}} ,
    %{\bm{i}\in\mathbb N_0^3, \lvert\bm{i}\rvert = k}, 
    i_2=0\right\}
    \cup
    \left\{\bm{\Psi}_{k,3, \bm{i}} ,
    %{\bm{i}\in\mathbb N_0^3, \lvert\bm{i}\rvert = k}, 
    i_3=0\right\}
    \\&
    \cup
    \left\{\bm{\Psi}_{k,4, \bm{i}} ,
    %{\bm{i}\in\mathbb N_0^3, \lvert\bm{i}\rvert = k}, 
    i_1\geq 1\right\}
    \cup
    \left\{\bm{\Psi}_{k,5, \bm{i}} ,
    %{\bm{i}\in\mathbb N_0^3, \lvert\bm{i}\rvert = k}, 
    i_2\geq 1\right\}
    \Big\} \text{ forms a basis for }\homdivfreeall{k},
\end{split}
\end{equation*}
see
Lemmas 2 and 3 of~\cite{Imbert-Gerard_2025-3} for a proof and  Appendix \ref{sec:CDbases} for a formal statement,
%\IF{The proof is already in the second-order Maxwell paper, I do not think it is worth it to put it here as well}
we will now define some smaller subsets $\homdivfreebis{k}$ and $\homcurlfreebis{k}$.%then show that they are in bijection with  $\homdivfree{k}$ and $\homcurlfree{k}$, and therefore, according to Proposition \ref{prop:bijectivity operators}, in bijection with $\homdivfreeall{k-1}$ and $\hompol{k-1}$.
\begin{definition}[Alternative irrotational and solenoidal subspaces]\label{def:alternative irr and sol}
    Given $k\geq 1$ and a point $\bm{x}_0\in\mathbb R^3$, the space $\homcurlfreebis{k}$ is defined by
    $$\homcurlfreebis{k} = Span\left\{\bm{\Phi}_{k,\bm{i}},{\bm{i}\in\mathbb N_0^3, \lvert\bm{i}\rvert = k+1, i_1\geq 2} \right\},
    $$
    and the space $\homdivfreebis{k} $ is defined either by
    $$
    \homdivfreebis{1} = Span\left\{ 
    \bm{\Psi}_{1, 1, \bm{e}_3}, 
    \bm{\Psi}_{1, 1, \bm{e}_2}, 
    \bm{\Psi}_{1, 2, \bm{e}_3}
    \right\}
    % \text{ or by }
    $$
    or by
\begin{equation*}
\begin{split}
    \homdivfreebis{k} 
    = Span &\cup_{{\bm{i}\in\mathbb N_0^3, \lvert\bm{i}
    \rvert = k}}\Big\{ 
    \left\{\bm{\Psi}_{k,2, \bm{i}} ,
    %{\bm{i}\in\mathbb N_0^3, \lvert\bm{i}\rvert = k}, 
    i_1>0,i_2=0,i_3>0\right\}
    \cup
    \left\{\bm{\Psi}_{k,3, \bm{i}} ,
    %{\bm{i}\in\mathbb N_0^3, \lvert\bm{i}\rvert = k}, 
    i_1>0,i_2>0,i_3=0\right\}
    \\&\phantom{ \cup_{{\bm{i}\in\mathbb N_0^3, \lvert\bm{i}\rvert = k}}\Big\{ }
    \cup
    \left\{\bm{\Psi}_{k,4, \bm{i}} ,
    %{\bm{i}\in\mathbb N_0^3, \lvert\bm{i}\rvert = k}, 
    i_1= 1,i_2i_3\neq 0\right\}
    \cup
    \left\{\bm{\Psi}_{k,4, \bm{i}} ,
    %{\bm{i}\in\mathbb N_0^3, \lvert\bm{i}\rvert = k}, 
    i_2= 1,i_1i_3\neq 0\right\}
    \\&\phantom{ \cup_{{\bm{i}\in\mathbb N_0^3, \lvert\bm{i}\rvert = k}}\Big\{ }
    \cup
    \left\{\bm{\Psi}_{k,4, \bm{i}} ,
    %{\bm{i}\in\mathbb N_0^3, \lvert\bm{i}\rvert = k}, 
    i_3= 1,i_1i_2\neq 0\right\}
    \cup
    \left\{\bm{\Psi}_{k,4, \bm{i}} ,
    %{\bm{i}\in\mathbb N_0^3, \lvert\bm{i}\rvert = k}, 
    (i_1-1)i_2i_3\neq 0\right\}
    \\&\phantom{ \cup_{{\bm{i}\in\mathbb N_0^3, \lvert\bm{i}\rvert = k}}\Big\{ }
    \cup
    \left\{\bm{\Psi}_{k,5, \bm{i}} ,
    %{\bm{i}\in\mathbb N_0^3, \lvert\bm{i}\rvert = k}, 
    i_2= 1,i_1i_3\neq 0\right\}
    \cup
    \left\{\bm{\Psi}_{k,5, \bm{i}} ,
    %{\bm{i}\in\mathbb N_0^3, \lvert\bm{i}\rvert = k}, 
    i_3=0,i_1(i_2-1)\neq 0\right\}
    \\&\phantom{ \cup_{{\bm{i}\in\mathbb N_0^3, \lvert\bm{i}\rvert = k}}\Big\{ }
    \cup
    \left\{\bm{\Psi}_{k,5, \bm{i}} ,
    %{\bm{i}\in\mathbb N_0^3, \lvert\bm{i}\rvert = k}, 
    i_1(i_2-1)i_3\neq 0\right\}
    \Big\}
    \\&
    \cup\{\bm{\Psi}_{k,1, k\bm e_3} ,\bm{\Psi}_{k,1, k\bm e_2} ,\bm{\Psi}_{k,3, k\bm e_2} ,\bm{\Psi}_{k,4,\bm{e_1} + (k-1)\bm{e_3}},\bm{\Psi}_{k,4,\bm{e_1} + (k-1)\bm{e_2}},\bm{\Psi}_{k,5,k\bm{e_2}}
    \},
\end{split}
\end{equation*}
for all $k>1$.
\end{definition}

The following two results then provide the theoretical foundation for an alternative algorithm to compute QT polynomial functions, with no need to solve any {linear system numerically}.

\begin{theorem}[Bijective divergence operator]\label{th:I**}
    Given $k\geq 1$ and a point $\bm{x}_0\in\mathbb R^3$, the divergence operator $\mathcal D_k$ defined by
    $$
        \begin{array}{rrcl}
            \mathcal D_k: &\homcurlfreebis{k} &\to&\hompol{k-1}  \\
             & \mathbf P&\mapsto&\diver\mathbf P
        \end{array}
    $$
    is bijective.
\end{theorem}

\begin{proof}
    Since the divergence is a linear transformation from $\homcurlfreebis{k}$ to $\hompol{k}$, it is sufficient to show that (a) $\dim \homcurlfreebis{k} = \dim\hompol{k-1}$ and (b) $\ker(\mathcal D_k) = \{\bm{0}\}$.

    \noindent
    (a) The set $\{\bm{\Phi}_{k, \bm{i}}\ :\ {\bm{i}\in\mathbb N_0^3, \lvert\bm{i}\rvert = k+1, i_1\geq 2}\}$ is linearly independent since it is a subset of a basis for $\homcurlfreeall{k}$. Moreover, that basis, namely $\left\{ \bm{\Phi}_{k, \bm{i}} ,{\bm{i}\in\mathbb N_0^3, \lvert\bm{i}\rvert = k+1} \right\}$, can be expressed as the union of two disjoint sets in the form:
    $$
         \{\bm{\Phi}_{k, \bm{i}}\ :\ {\bm{i}\in\mathbb N_0^3, \lvert\bm{i}\rvert = k+1, i_1\geq 2}\}
         \cup
         \{\bm{\Phi}_{k, \bm{i}}\ :\ {\bm{i}\in\mathbb N_0^3, \lvert\bm{i}\rvert = k+1, i_1 <2}\},
    $$
    where the basis has
    $\dim\homcurlfreeall{k} = 
    \frac{1}{2} (k+2)(k+3)$ elements and $ \{\bm{\Phi}_{k, \bm{i}}\ :\ {\bm{i}\in\mathbb N_0^3, \lvert\bm{i}\rvert = k+1, i_1 <2}\}$ has $ (k+2) + (k+1)$ elements.
    As a consequence, 
    \begin{equation*}
        %\begin{split}
            \dim \homcurlfreebis{k} 
            %=&\ \dim\homcurlfreeall{k} - \# \{\bm{\Phi}_{k, \bm{i}}\ :\ {\bm{i}\in\mathbb N_0^3, \lvert\bm{i}\rvert = k+1, i_1 = 0, 1}\} \\
            =%&\ 
            \frac{1}{2} (k+2)(k+3) - (2k+3) = \frac{1}{2} k(k+1),
        %\end{split}
    \end{equation*}
    which is precisely $\dim\hompol{k-1}$.
    
    \noindent
    (b) %We conclude the proof by showing that $\ker(\diver|_{\homcurlfreebis{k}}) = \{\bm{0}\}$. 
    Since $\{\bm{\Phi}_{k, \bm{i}}\ :\ {\bm{i}\in\mathbb N_0^3, \lvert\bm{i}\rvert = k+1, i_1\geq 2}\}$ is linearly independent it forms a basis for $\homcurlfreebis{k}$.
    Let $\bm{\Phi}\in \ker(\mathcal D_k)$. Then, there exist some coefficients $c_{\bm{i}}\in\mathbb C$ such that 
    \begin{equation*}
        \bm{\Phi} = \sum_{\substack{\bm{i}\in\mathbb N_0^3, \lvert\bm{i}\rvert = k+1 \\ i_1\geq 2}} c_{\bm{i}} \bm{\Phi}_{k, \bm{i}}.
    \end{equation*}
    We now show that all the coefficients $c_{\bm{i}}$ are zero by an iterative approach on increasing values of the first scalar index $i_1$.
    For the sake of compactness, in what follows, $\bm y$ denotes $\bm x - \bm x_0$.
    By hypothesis $\diver\bm{\Phi} = 0$, so, by linearity we have
    \begin{equation}\label{eq:kernel divergence element}
        \sum_{\substack{\bm{i}\in\mathbb N_0^3, \lvert\bm{i}\rvert = k+1 \\ i_1\geq 2}} c_{\bm{i}} \left( i_1 (i_1-1) \bm{y}^{\bm{i}-2\bm{e}_1} + i_2 (i_2-1) \bm{y}^{\bm{i}-2\bm{e}_2} + i_3 (i_3-1) \bm{y}^{\bm{i}-2\bm{e}_3} \right)  = 0
        % \begin{split}
        %     \diver\bm{\Phi} =& \sum_{\substack{\bm{i}\in\mathbb N_0^3, \lvert\bm{i}\rvert = k+1 \\ i_1\geq 2}} c_{\bm{i}} \diver\bm{\Phi}_{k, \bm{i}} \\
        %     =& \sum_{\substack{\bm{i}\in\mathbb N_0^3, \lvert\bm{i}\rvert = k+1 \\ i_1\geq 2}} c_{\bm{i}} \left( i_1 (i_1-1) \bm{y}^{\bm{i}-2\bm{e}_1} + i_2 (i_2-1) \bm{y}^{\bm{i}-2\bm{e}_2} + i_3 (i_3-1) \bm{y}^{\bm{i}-2\bm{e}_3} \right)  = 0.
        % \end{split}
    \end{equation}
    %The last expression is obtained by applying the divergence operator to each basis function \eqref{eq:I_k basis}. 
    Here the second subterm $i_2 (i_2-1) \bm{y}^{\bm{i}-2\bm{e}_2}$ in this expression is zero when $i_2\in\{0, 1\}$. Analogously, the third subterm $i_3 (i_3-1) \bm{y}^{\bm{i}-2\bm{e}_3}$ is zero when $i_3\in\{0, 1\}$.
    \begin{itemize}
        \item For all $\bm{j}\in\mathbb N_0^3$ such that $\lvert\bm{j}\rvert = k-1$ and $j_1=0$, the monomial $\bm{y}^{\bm{j}}$ only appears once in the sum \eqref{eq:kernel divergence element}, because $i_1\geq 2$. This element corresponds to the index $\bm{i} = \bm{j} + 2\bm{e}_1$. Therefore, $c_{\bm{i}} = 0$ for all multi-index $\bm{i}$ such that $i_1 = 2$, and we can rewrite \eqref{eq:kernel divergence element} as
        \begin{equation*}
            \diver\bm{\Phi} = \sum_{\substack{\bm{i}\in\mathbb N_0^3, \lvert\bm{i}\rvert = k+1 \\ i_1\geq 3}} c_{\bm{i}} \left( i_1 (i_1-1) \bm{y}^{\bm{i}-2\bm{e}_1} + i_2 (i_2-1) \bm{y}^{\bm{i}-2\bm{e}_2} + i_3 (i_3-1) \bm{y}^{\bm{i}-2\bm{e}_3} \right)  = 0.
        \end{equation*}

        \item We apply a similar procedure with increasing values of $i_1$. At a generic step $k+1 \geq m\geq 3$, assume that $c_{\bm{i}} = 0$ for all $\bm{i}$ with $i_1 < m$. Then,
        \begin{equation*}
            \diver\bm{\Phi} = \sum_{\substack{\bm{i}\in\mathbb N_0^3, \lvert\bm{i}\rvert = k+1 \\ i_1\geq m}} c_{\bm{i}} \left( i_1 (i_1-1) \bm{y}^{\bm{i}-2\bm{e}_1} + i_2 (i_2-1) \bm{y}^{\bm{i}-2\bm{e}_2} + i_3 (i_3-1) \bm{y}^{\bm{i}-2\bm{e}_3} \right)  = 0.
        \end{equation*}
        For all $\bm{j}\in\mathbb N_0^3$ such that $j_1 = m$, the monomial $\bm{y}^{\bm{j}}$ only appear once in this sum, corresponding to $\bm{i} = \bm{j}+2\bm{e}_1$. Therefore, $c_{\bm{i}} = 0$ for all multi-index $\bm{i}$ such that $i_1 = m$.
    \end{itemize}
    After these $k$ steps, we get $\bm{\Phi} = \bm{0}$. Therefore $\ker(\mathcal D_k) = \{\bm{0}\}$.
\end{proof}

\begin{theorem}[Bijective curl operator]\label{th:S**}
    Given $k\geq 1$ and a point $\bm{x}_0\in\mathbb R^3$, the curl operator $\mathcal C_k$ defined by
    $$
        \begin{array}{rrcl}
            \mathcal C_1: &\homdivfreebis{1} &\to&(\hompol{0} )^3 \\
             & \mathbf P&\mapsto&\curl\mathbf P
        \end{array}
        \qquad \text{ and }\qquad
        \begin{array}{rrcl}
            \mathcal C_k: &\homdivfreebis{k} &\to&\homdivfreeall{k-1} \\
             & \mathbf P&\mapsto&\curl\mathbf P
        \end{array}
    $$
    is bijective.
\end{theorem}

\begin{proof}
    Since the curl is a linear transformation defined on $\homdivfreebis{k}$, it is sufficient to show that (a) the dimensions of the domain and co-domain coincide and (b) $\mathcal C_k$ is surjective.

    \noindent
    (a) The selected elements are linearly independent since they are a subset of the basis of $\homdivfreeall{k}$. As a consequence, by counting the number of elements, we get
    \begin{equation*}
        \dim \homdivfreebis{k} = 2 + (k-1) + k + \frac{(k+2)(k-1)}{2} + \frac{k(k-1)}{2} = k(k+2) = \dim \homdivfreeall{k-1}
    \end{equation*}
    if $k\geq 2$, and $\dim\homdivfreebis{1} = 3 = \dim (\hompol{0})^3$.

    \noindent
    (b) To conclude the proof, we prove that $\mathcal C_k$ is surjective by showing that all elements of a basis of the co-domain can be represented as the curl of a linear combination of basis elements of $\homdivfreebis{k}$.  %For simplicity, we make the assumption that at each step we can use the elements already represented.
    If $k=1$, this simply follows from
\begin{equation*}%\label{eq:curl of S**_1}
    \curl\bm{\Psi}_{1, 1, \bm{e}_3} 
    = \bm e_2,
    \
    \curl\bm{\Psi}_{1, 1, \bm{e}_2} 
    =-\bm e_3, 
    \
    \curl\bm{\Psi}_{1, 2, \bm{e}_3}
    = -\bm e_1.
\end{equation*}
Consider next the case $k\geq 2$.
    Table \ref{tab:representation curl} explicitly shows the image of all elements of the basis of $\homdivfreebis{k}$ under the curl operator in the basis of $\homdivfreeall{k-1}$ for $k\geq 2$.
\begin{table}[htbp]
    \centering
    \renewcommand{\arraystretch}{1.5}
    \begin{tabular}{|c|c|c|}
        \hline
        Curl-operator action on $\homdivfreebis{k}$ & $\bm{i} = (i_1, i_2, i_3)$ & \# elements \\ \hline 
        
        $\curl \bm{\Psi}_{k,1,k\bm{e_3}} = k \bm{\Psi}_{k-1, 2, (k-1)\bm{e_3}}$ & $i_1=0, i_2=0, i_3=k$ & $1$ \\
        
        $\curl \bm{\Psi}_{k,1,k\bm{e_2}} = -k \bm{\Psi}_{k-1, 3, (k-1)\bm{e_2}}$ & $i_1=0, i_2=k, i_3=0$ & $1$ \\
        
        % $\curl\bm{\Psi}_{k,1,\bm{i}} = \bm{\Psi}_{k-1, 5, \bm{i} - \bm{e_3}}$ & $i_1=0, i_2>0, i_3>0$ \\ 
        \hline

        % $\curl\bm{\Psi}_{k,2,k\bm{e_3}} = -k \bm{\Psi}_{k-1, 1, (k-1)\bm{e_3}}$ & $i_1=0, i_2=0, i_3=k$ \\

        % $\curl\bm{\Psi}_{k,2,k\bm{e_1}} = k \bm{\Psi}_{k-1, 3, (k-1)\bm{e_1}}$ & $i_1=k, i_2=0, i_3=0$ \\
        
        $\curl \bm{\Psi}_{k,2,\bm{i}} = -\bm{\Psi}_{k-1, 4, \bm{i} - \bm{e_3}}$ & $i_1>0, i_2=0, i_3>0$ & $k-1$ \\ \hline

        $\curl \bm{\Psi}_{k,3,k\bm{e_2}} = k \bm{\Psi}_{k-1, 1, (k-1)\bm{e_2}}$ & $i_1=0, i_2=k, i_3=0$ & $1$ \\
        
        % $\curl\bm{\Psi}_{k,3,k\bm{e_1}} = -k \bm{\Psi}_{k-1, 2, (k-1)\bm{e_1}}$ & $i_1=k, i_2=0, i_3=0$ \\
        
        $\curl \bm{\Psi}_{k,3,\bm{i}} = i_2 \bm{\Psi}_{k-1, 4, \bm{i} - \bm{e_2}} - i_1 \bm{\Psi}_{k-1, 5, \bm{i} - \bm{e_1}}$ & $i_1>0, i_2>0, i_3=0$ & $k-1$ \\ \hline

        $\curl \bm{\Psi}_{k,4,\bm{e_1} + (k-1)\bm{e_3}} = k(k-1) \bm{\Psi}_{k-1,2,\bm{e_1} + (k-2)\bm{e_3}}$ & $i_1=1, i_2=0, i_3=k-1$ & $1$ \\
        
        $\begin{array}{r@{}l@{}} \curl \bm{\Psi}_{k,4,\bm{e_1} + (k-1)\bm{e_2}} =&\ -(k-1) \bm{\Psi}_{k-1,1,(k-2)\bm{e_2}+\bm{e_3}}\\
        &\ - (k-1) \bm{\Psi}_{k-1,3,\bm{e_1} + (k-2)\bm{e_2}}
        \end{array}$ & $i_1=1, i_2=k-1, i_3=0$ & $1$ \\

        % $\curl\bm{\Psi}_{k,4,k\bm{e_1}} = k(k-1) \bm{\Psi}_{k-1,2,(k-2)\bm{e_1} + \bm{e_3}}$ & $i_1=k, i_2=0, i_3=0$ \\
        
        $\begin{array}{r@{}l@{}} \curl \bm{\Psi}_{k,4,\bm{i}} =&\ -i_2 \bm{\Psi}_{k-1,1,\bm{i} -\bm{e_1} -\bm{e_2}+\bm{e_3}}\\
        &\ +(i_3+1) \bm{\Psi}_{k-1,5,\bm{i} -\bm{e_3}}
        \end{array}$ & $i_1=1, i_2>0, i_3>0$ & $k-2$ \\

        $\begin{array}{r@{}l@{}} \curl \bm{\Psi}_{k,4,\bm{i}} =&\ (i_3+1) i_3 \bm{\Psi}_{k-1,2,\bm{i}-\bm{e_3}} \\
        &\ + i_1(i_1-1) \bm{\Psi}_{k-1,2,\bm{i}-2\bm{e_1} + \bm{e_3}} \end{array}$ & $i_1>1, i_2=0, i_3>0$ & $k-2$ \\

        $\begin{array}{r@{}l@{}} \curl \bm{\Psi}_{k,4,\bm{i}} =&\ -i_2\bm{\Psi}_{k-1,3,\bm{i} - \bm{e_2}} \\ 
        &\ - \frac{i_1 i_2}{2} \bm{\Psi}_{k-1,4,\bm{i}-\bm{e_1}-\bm{e_2}+\bm{e_3}} \\
        & +\frac{i_1(i_1-1)}{2} \bm{\Psi}_{k-1,5,\bm{i}-2\bm{e_1}+\bm{e_3}} \end{array}$ & $i_1>1, i_2>0, i_3=0$ & $k-2$ \\

        $\begin{array}{r@{}l@{}} \curl \bm{\Psi}_{k,4,\bm{i}} =&\ (i_3+1) \bm{\Psi}_{k-1,5,\bm{i} -\bm{e_3}} \\
        &\ -\frac{i_1 i_2}{i_3+2} \bm{\Psi}_{k-1,4,\bm{i} -\bm{e_1} -\bm{e_2}+\bm{e_3}} \\
        &\ +\frac{i_1(i_1-1)}{i_3+2} \bm{\Psi}_{k-1,5,\bm{i}-2\bm{e_1} +\bm{e_3}}
        \end{array}$ & $i_1>1, i_2>0, i_3>0$ & $\frac{(k-2)(k-3)}{2}$ \\ \hline

        % $\curl\bm{\Psi}_{k,5,\bm{e_2} + (k-1)\bm{e_3}} = -k(k-1) \bm{\Psi}_{k-1,1,\bm{e_2} + (k-2)\bm{e_3}}$ & $i_1=0, i_2=1, i_3=k-1$ \\
        
        $\curl \bm{\Psi}_{k,5,k\bm{e_2}} = -k(k-1) \bm{\Psi}_{k-1,1,(k-2)\bm{e_2} + \bm{e_3}}$ & $i_1=0, i_2=k, i_3=0$ & $1$ \\

        % $\begin{array}{r@{}l@{}} \curl\bm{\Psi}_{k,5,(k-1)\bm{e_1} + \bm{e_2}} =&\ (k-1) \bm{\Psi}_{k-1,2,(k-2)\bm{e_1}+\bm{e_3}}\\
        % &\ + (k-1) \bm{\Psi}_{k-1,3,(k-2)\bm{e_1} + \bm{e_2}}
        % \end{array}$ & $i_1=k-1, i_2=1, i_3=0$ \\
        
        % $\begin{array}{r@{}l@{}} \curl\bm{\Psi}_{k,5,\bm{i}} =&\ -i_2(i_2-1) \bm{\Psi}_{k-1,1,\bm{i} -2\bm{e_2}+\bm{e_3}} -(i_3+1)i_3 \bm{\Psi}_{k-1,1,\bm{i} -\bm{e_3}}
        % \end{array}$ & $i_1=0, i_2>1, i_3>0$ \\

        $\begin{array}{r@{}l@{}} \curl \bm{\Psi}_{k,5,\bm{i}} =&\ i_1 \bm{\Psi}_{k-1,2,\bm{i}-\bm{e_1}-\bm{e_2}+\bm{e_3}} \\
        &\ - (i_3+1) \bm{\Psi}_{k-1,4,\bm{i}-\bm{e_3}} \end{array}$ & $i_1>0, i_2=1, i_3>0$ & $k-2$ \\

        $\begin{array}{r@{}l@{}} \curl \bm{\Psi}_{k,5,\bm{i}} =&\ i_1\bm{\Psi}_{k-1,3,\bm{i} - \bm{e_1}} \\
        &\ - \frac{i_2(i_2-1)}{2} \bm{\Psi}_{k-1,4,\bm{i}-2\bm{e_2}+\bm{e_3}} \\
        & +\frac{i_1 i_2}{2} \bm{\Psi}_{k-1,5,\bm{i}-\bm{e_1}-\bm{e_2}+\bm{e_3}} \end{array}$ & $i_1>0, i_2>1, i_3=0$ & $k-2$ \\

        $\begin{array}{r@{}l@{}} \curl \bm{\Psi}_{k,5,\bm{i}} =&\ -(i_3+1) \bm{\Psi}_{k-1,4,\bm{i} -\bm{e_3}} \\ 
        &\ - \frac{i_2(i_2-1)}{i_3+2} \bm{\Psi}_{k-1,4,\bm{i} -2\bm{e_2}+\bm{e_3}} \\
        &\ +\frac{i_1 i_2}{i_3+2} \bm{\Psi}_{k-1,5,\bm{i}-\bm{e_1}-\bm{e_2} +\bm{e_3}}
        \end{array}$ & $i_1>0, i_2>1, i_3>0$ & $\frac{(k-2)(k-3)}{2}$ \\ \hline
    \end{tabular}
    \caption{Action of the curl operator on basis elements of $\homdivfreebis{k}$ for $k\geq 2$.}
    \label{tab:representation curl}
\end{table}
    For clarity, denote by $\mathcal P$ the property of being the curl of a linear combination of basis elements of $\homdivfreebis{k}$. 
    The proof proceeds in the following order. %In the following list of all elements of the basis of $\homdivfreeall{k-1}$,
    \begin{itemize}
        \item All $\bm{\Psi}_{k-1, 2, \bm{j}}$ with $\bm{j}\in\mathbb N_0^3$, $\lvert\bm{i}\rvert = k-1$, $j_2=0$, satisfy $\mathcal P$ since:
        \begin{itemize}
            \item [$\circ$] $k \bm{\Psi}_{k-1, 2, (k-1)\bm{e}_3} = \curl\left[\bm{\Psi}_{k, 1, k\bm{e}_3}\right]$;
            \item [$\circ$] $k(k-1) \bm{\Psi}_{k-1, 2, \bm{e_1}+(k-2)\bm{e}_3} = \curl \left[ \bm{\Psi}_{k, 4, \bm{e}_1 + (k-1)\bm{e}_3} \right]$;
            \item [$\circ$] for $j_1$ increasing from $2$ to $k-1$, $\bm{j} = (j_1, 0, k-j_1)$ and
            \begin{equation*}
                (j_3+2) (j_3+1) \bm{\Psi}_{k-1, 2, \bm{j}} = \curl \left[ \bm{\Psi}_{k, 4, \bm{j}+\bm{e}_3} \right] - j_1(j_1-1) \bm{\Psi}_{k-1, 2, \bm{j}-2\bm{e}_1+2\bm{e}_3}           
            \end{equation*}
            where $(\bm{j}-2\bm{e}_1+2\bm{e}_3)_1<j_1$, hence $\bm{\Psi}_{k-1, 2, \bm{j}-2\bm{e}_1+2\bm{e}_3}$ %is itself already the curl of a linear combination of elements of the basis of $\homdivfreebis{k}$
           itself satisfies  $\mathcal P$.
        \end{itemize}
        
        \item All $\bm{\Psi}_{k-1, 3, \bm{i}}$ with $\bm{i}\in\mathbb N_0^3$, $\lvert\bm{i}\rvert = k-1$, $i_3=0$, satisfy $\mathcal P$ since:
        \begin{itemize}
            \item [$\circ$] $k\bm{\Psi}_{k-1, 3, (k-1)\bm{e}_2} = - \curl\left[\bm{\Psi}_{k, 1, k\bm{e}_2}\right]$;
            \item [$\circ$] $k (k-1) \bm{\Psi}_{k-1, 3, \bm{e}_1 + (k-2) \bm{e}_2} = \curl\left[ \bm{\Psi}_{k, 5, k\bm{e}_2} - k \bm{\Psi}_{k, 4, \bm{e}_1 + (k-1) \bm{e}_2} \right]$;
            \item [$\circ$] for $j_1$ increasing from $2$ to $k-1$, $\bm{j} = (j_1, k-1-k_1, 0)$ and
            \begin{equation*}
                \begin{split}
                    (j_2+1) (j_1+1) \bm{\Psi}_{k-1, 3, \bm{j}} =&\ \curl\left[ j_1 \bm{\Psi}_{k, 5, \bm{j}-\bm{e}_1+2\bm{e}_2} - (j_2+2) \bm{\Psi}_{k, 4, \bm{j}+\bm{e}_2} \right] \\
                    &\ - (j_1+1) j_1  \bm{\Psi}_{k-1, 3, \bm{j}-2\bm{e}_1+2\bm{e}_2}
                \end{split}
            \end{equation*}
            where $(\bm{j}-2\bm{e}_1+2\bm{e}_2)_1<j_1$, hence $\bm{\Psi}_{k-1, 3, \bm{j}-2\bm{e}_1+2\bm{e}_2}$ itself satisfies $\mathcal P$.
        \end{itemize}
        
        \item All $\bm{\Psi}_{k-1, 4, \bm{i}}$ with $\bm{i}\in\mathbb N_0^3$, $\lvert\bm{i}\rvert = k-1$, $i_1\geq 1$, satisfy $\mathcal P$:
        \begin{itemize}
            \item [$\circ$] for $j_1\in\indices{1}{k-1}$, $\bm{j} = (j_1, 0, k-1-j_1)$ and
            \begin{equation*}
                \bm{\Psi}_{k-1, 4, \bm{j}} = - \curl \left[ \bm{\Psi}_{k-1, 2, \bm{j}+\bm{e}_3} \right];
            \end{equation*}
            \item [$\circ$] for $j_1\in\indices{1}{k-2}$, $\bm{j} = (j_1, 1, k-2-j_1)$ and
            \begin{equation*}
                (j_3+1) \bm{\Psi}_{k-1, 4, \bm{j}} = - \curl \left[ \bm{\Psi}_{k, 5, \bm{j}+\bm{e}_3} \right] + j_1 \bm{\Psi}_{k-1, 2, \bm{j}-\bm{e}_1-\bm{e}_2+2\bm{e}_3}
            \end{equation*}
            where $\bm{\Psi}_{k-1, 2, \bm{j}-\bm{e}_1-\bm{e}_2+2\bm{e}_3}$ itself satisfies $\mathcal P$;
            \item [$\circ$] for $j_2$ increasing from $2$ to $k-3$, $\bm{j} = (k-1-j_2, j_2, 0)$ and
            \begin{equation*}
                \begin{split}
                    (j_1+2) (j_1+1) \bm{\Psi}_{k-1, 4, \bm{j}} = &\ \curl\bigl[(j_1+2) j_2 \bm{\Psi}_{k, 4, \bm{j}+\bm{e}_1-\bm{e}_2+\bm{e}_3} \\
                    &\ \hspace{-3.2cm} - (j_1+2) (j_1+1) \bm{\Psi}_{k, 5, \bm{j}+\bm{e}_3} + j_2 \bm{\Psi}_{k, 3, \bm{j}+2\bm{e}_1-\bm{e}_2} \bigr] - j_2 (j_2-1) \bm{\Psi}_{k-1, 4, \bm{j}+2\bm{e}_1-2\bm{e}_2}
                \end{split}
            \end{equation*}
            where $(\bm{j}+2\bm{e}_1-2\bm{e}_2)_2 < j_2$, hence $\bm{\Psi}_{k-1, 4, \bm{j}+2\bm{e}_1-2\bm{e}_2}$  itself satisfies $\mathcal P$;
            \item [$\circ$] for $j_2$ increasing from $2$ to $k-3$, $\bm{j} = (k-2-j_2, j_2, 1)$ and
            \begin{equation*}
                \begin{split}
                    2(j_1+3) (j_1+2) (j_1+1) \bm{\Psi}_{k-1, 4, \bm{j}} =&\ \curl\bigl[ (j_1+3) (j_1+2) j_2 \bm{\Psi}_{k, 4, \bm{j}+\bm{e}_1-\bm{e}_2+\bm{e}_3} \\
                    &\ \hspace{-3.8cm} - (j_1+3) (j_1+2) (j_1+1) \bm{\Psi}_{k, 5, \bm{j}+\bm{e}_3} - 4 j_2 \bm{\Psi}_{k, 4, \bm{j}+3\bm{e}_1-\bm{e}_2-\bm{e}_3} \bigl] \\
                    &\ \hspace{-3.8cm} - 4 j_2 (j_2-1) \bm{\Psi}_{k-1, 3, \bm{j}+3\bm{e}_1-2\bm{e}_2-\bm{e}_3} - 2 (j_1+3) j_2 (j_2-1) \bm{\Psi}_{k-1, 4, \bm{j}+2\bm{e}_1-2\bm{e}_2}
                \end{split}
            \end{equation*}
            where $(\bm{j}+2\bm{e}_1-2\bm{e}_2)_2 < j_2$, hence both $\bm{\Psi}_{k-1, 3, \bm{j}+3\bm{e}_1-2\bm{e}_2-\bm{e}_3}$ and $\bm{\Psi}_{k-1, 4, \bm{j}+2\bm{e}_1-2\bm{e}_2}$ satisfy $\mathcal P$;
            \item [$\circ$] for $j_3$ increasing from $2$ to $k-4$ and for $j_2$ increasing from $2$ to $k-2-j_3$, $\bm{j} = (k-1-j_2-j_3, j_2, j_3)$ and
            \begin{equation*}
                \begin{split}
                    (j_1+2) (j_1+1) (j_3+2) \bm{\Psi}_{k-1, 4, \bm{j}} =&\ \curl\bigl[ (j_1+2) j_2 \bm{\Psi}_{k, 4, \bm{j}+\bm{e}_1-\bm{e}_2+\bm{e}_3} \\
                    &\ \hspace{-4cm}- (j_1+2) (j_1+1) \bm{\Psi}_{k, 5, \bm{j}+\bm{e}_3} - (j_3+2) (j_3+1) \bm{\Psi}_{k, 5, \bm{j}+2\bm{e}_1-\bm{e}_3} \bigr] \\
                    &\ \hspace{-4cm} - (j_3+2) (j_3+1) j_3 \bm{\Psi}_{k-1, 4, \bm{j}+2\bm{e}_1-2\bm{e}_3} - j_2 (j_2-1) (j_3+2) \bm{\Psi}_{k-1, 4, \bm{j}+2\bm{e}_1-2\bm{e}_2}
                \end{split}
            \end{equation*}
            where $(\bm{j}+2\bm{e}_1-2\bm{e}_3)_3 < j_3$ and $(\bm{j}+2\bm{e}_1-2\bm{e}_2)_2 < j_2$, hence both $\bm{\Psi}_{k-1, 4, \bm{j}+2\bm{e}_1-2\bm{e}_3}$ and $\bm{\Psi}_{k-1, 4, \bm{j}+2\bm{e}_1-2\bm{e}_2}$ satisfy $\mathcal P$.
        \end{itemize}
        
        \item All $\bm{\Psi}_{k-1, 5, \bm{i}}$ with $\bm{i}\in\mathbb N_0^3$, $\lvert\bm{i}\rvert = k-1$, $i_2\geq 1$:
        \begin{itemize}
            \item [$\circ$] for $j_1\in\indices{0}{k-2}$, $\bm{j} = (j_1, k-1-j_1, 0)$ and
            \begin{equation*}
                (j_1+1) \bm{\Psi}_{k-1, 5, \bm{j}} = - \curl \bigl[ \bm{\Psi}_{k, 3, \bm{j}+\bm{e}_1} \bigr] + j_2 \bm{\Psi}_{k-1, 4, \bm{j}+\bm{e}_1-\bm{e}_2}
            \end{equation*}
            where $\bm{\Psi}_{k-1, 4, \bm{j}+\bm{e}_1-\bm{e}_2}$ itself satisfies $\mathcal P$;
            \item [$\circ$] for $j_1\in\indices{0}{k-3}$, $\bm{j} = (j_1, k-2-j_1, 1)$ and
            \begin{equation*}
                \begin{split}
                    (j_1+2) (j_1+1) \bm{\Psi}_{k-1, 5, \bm{j}} =&\ \curl\bigl[ 2 \bm{\Psi}_{k, 4, \bm{j}+2\bm{e}_1-\bm{e}_3} \bigr] + 2 j_2 \bm{\Psi}_{k-1, 3, \bm{j}+2\bm{e}_1-\bm{e}_2-\bm{e}_3} \\ 
                    &\ + (j_1+2) j_2 \bm{\Psi}_{k-1, 4, \bm{j}+\bm{e}_1-\bm{e}_2}
                \end{split}
            \end{equation*}
            where both $\bm{\Psi}_{k-1, 3, \bm{j}+2\bm{e}_1-\bm{e}_2-\bm{e}_3}$ and $\bm{\Psi}_{k-1, 4, \bm{j}+\bm{e}_1-\bm{e}_2}$ both satisfy $\mathcal P$;
            \item [$\circ$] for $j_3\in\indices{2}{k-2}$ and $j_1\in\indices{0}{k-2-j_3}$, $\bm{j} = (j_1, k-1-j_1-j_3, j_3)$ and
            \begin{equation*}
                \begin{split}
                    (j_1+1) (j_2+1) \bm{\Psi}_{k-1, 5, \bm{j}} =&\ \curl\bigl[ (j_3+1) \bm{\Psi}_{k, 5, \bm{j}+\bm{e}_1+\bm{e}_2-\bm{e}_3} \bigr] \\
                    &\ + (j_3+1) j_3 \bm{\Psi}_{k-1, 4, \bm{j}+\bm{e}_1+\bm{e}_2-2\bm{e}_3} + (j_2+1) j_2 \bm{\Psi}_{k-1, 4, \bm{j}+\bm{e}_1-\bm{e}_2}
                \end{split}
            \end{equation*}
            where $\bm{\Psi}_{k-1, 4, \bm{j}+\bm{e}_1+\bm{e}_2-2\bm{e}_3}$ and $\bm{\Psi}_{k-1, 4, \bm{j}+\bm{e}_1-\bm{e}_2}$ both satisfy $\mathcal P$.
        \end{itemize}
        
        \item All $\bm{\Psi}_{k-1, 1, \bm{i}}$ with $\bm{i}\in\mathbb N_0^3$, $\lvert\bm{i}\rvert = k-1$, $i_1=0$:
        \begin{itemize}
            \item [$\circ$] $k \bm{\Psi}_{k-1, 1, (k-1)\bm{e_2}} = \curl \bigl[ \bm{\Psi}_{k, 3, k\bm{e}_2} \bigr]$;
            \item [$\circ$] $k (k-1) \bm{\Psi}_{k-1, 1, (k-2)\bm{e}_2+\bm{e}_3} = - \curl \bigl[ \bm{\Psi}_{k, 5, k\bm{e}_2} \bigr]$;
            \item [$\circ$] for $j_3\in\indices{2}{k-1}$, $\bm{j} = (0, k-1-j_3, j_3)$ and
            \begin{equation*}
                (j_2+1) \bm{\Psi}_{k-1, 1, \bm{j}} = - \curl \bigl[ \bm{\Psi}_{k, 4, \bm{j}+\bm{e}_1+\bm{e}_2-\bm{e}_3} \bigr] + j_3 \bm{\Psi}_{k-1, 5, \bm{j}+\bm{e}_1+\bm{e}_2-2\bm{e}_3}
            \end{equation*}
            where $\bm{\Psi}_{k-1, 5, \bm{j}+\bm{e}_1+\bm{e}_2-2\bm{e}_3}$ itself satisfies $\mathcal P$.
        \end{itemize}
    \end{itemize}
    This concludes the proof.
\end{proof}
Finally, Algorithm \ref{alg:altQT construction} proposes an alternative approach to Algorithm \ref{alg:QT construction} for the computation of individual QT polynomial functions without the need for an SVD {or any numerical solve}. 
%\IF{In practice, I still use a linear solver for the equations involving the curl and the divergence in the code. Indeed, for what we have shown before, it is ``easy" to explicitly write the matrices representing $\curl$ and $\diver$ (acting on polynomials of degree $k$), while their inverse is much more technical to write (even if not impossible I believe). If our intention is to publish the code, I would avoid mentioning ``any numerical solve".}

\begin{algorithm}[ht]
    \caption{
    %Construction of a QT polynomial
    $\left(\funel{Q},\funmg{Q}\right) = {\tt ALTConstructQTpol}\left(p,\bm{x}_0,f_\permitt, f_\permeab, (\funel{H},\funmg{H})\right)$
    }\label{alg:altQT construction}
    %\lmig{the expilons and mus should also be input! and review how spaces are used below}\IF{I changed the description of the input adding the missing pieces and describing the coefficients as Taylor expansions rather than functions}

    {\bf Input:} {$p\in\mathbb N$,  %$\permitt, \permeab$ satisfying Assumption~\ref{ass:coefficients}}
    $\bm{x}_0\in\mathbb R^3$,\\
    \nonl\phantom{\bf Input:}
    $f_\permitt = \displaystyle\sum_{k=0}^p \permitt_k$ and $f_\permeab = \displaystyle\sum_{k=0}^p \permeab_k$, with $(\permitt_k, \permeab_k) \in\left(\hompol{k}\right)^2$  $\forall k\in\llbracket 0,p\rrbracket$ and $\varepsilon_0\mu_0\neq 0$, \\%and $\permitt_0. \permeab_0 \neq 0$;
    \nonl\phantom{\bf Input:}
    $\left(\funel{H},\funmg{H}\right) = \left(\displaystyle\sum_{k=0}^p \funel{H}_k,   \displaystyle\sum_{k=0}^p \funmg{H}_k \right)$ %\in \harm{p}\times \harm{p}$
    with $(\funel{H}_k,\funmg{H}_k )\in \homharm{k}\times \homharm{k}$ $\forall k\in\llbracket 0,p\rrbracket$
    }

    \For{$k\ {\rm from}\ 1\ {\rm to}\ p$}{
        {\bf find} 
        $\left(\funel{F}_k, \funmg{F}_k\right) \in \left(\homdivfreebis{k}\right)^2$ such that
        \begin{gather}
            \curl \funel{F}_k = {\color{white} -}\imag \kappa \,\Biggl(\, \displaystyle\sum_{j=0}^{k-2} \permeab_j \left(\funmg{F}_{k-j-1} + \funmg{G}_{k-j-1} + \funmg{H}_{k-j-1} \right) + \permeab_{k-1} \funmg{H}_0 \Biggr) %\label{eq:alg 1}
            \\
            \curl \funmg{F}_k = - \imag \kappa \,\Biggl(\, \displaystyle\sum_{j=0}^{k-2} \permitt_j \left(\funel{F}_{k-j-1} +\funel{G}_{k-j-1} + \funel{H}_{k-j-1} \right) + \permitt_{k-1} \funel{H}_0 \Biggr) %\label{eq:alg 2}
        \end{gather}

        {\bf find}
        $\left(\funel{G}_k, \funmg{G}_k \right)\in\left(\homcurlfreebis{k}\right)^2$ such that
        \begin{gather}
            \diver \funel{G}_k = - (\permitt_0)^{-1} \diver \Biggl(\, \displaystyle\sum_{j=1}^{k-1} \permitt_j \left(\funel{F}_{k-j} +\funel{G}_{k-j} + \funel{H}_{k-j} \right) + \permitt_k \funel{H}_0 \Biggr) %\label{eq:alg 3}
            \\
            \diver \funmg{G}_k = - (\permeab_0)^{-1} \diver \Biggl(\, \displaystyle\sum_{j=1}^{k-1} \permeab_j \left(\funmg{F}_{k-j} +\funmg{G}_{k-j} + \funmg{H}_{k-j} \right) + \permeab_k \funmg{H}_0 \Biggr) %\label{eq:alg 4}
        \end{gather}
    }
    
    {\bf Output: } {$\left(\funel{Q},\funmg{Q}\right) \in (\pol{p})^3 \times (\pol{p})^3$  where
    \\
    \nonl\phantom{\bf Output:}$\funel{Q} = \funel{H}_0 +  \displaystyle\sum_{k=1}^p \left(\funel{F}_k + \funel{G}_k + \funel{H}_k\right)$ and $\funmg{Q} = \funmg{H}_0 + \displaystyle\sum_{k=1}^p \left(\funmg{F}_k + \funmg{G}_k + \funmg{H}_k\right)$
    }
\end{algorithm}
\begin{remark}
    In the output of Algorithm \ref{alg:altQT construction}, the $\bm F,\bm G,\bm H$ terms are not the components of the Helmholtz decomposition.
\end{remark}

\subsection{Construction of quasi-Trefftz bases}\label{subsec:construction of quasi-Trefftz bases}
QT bases can be constructed thanks to both previous approaches for the construction of individual QT functions. Indeed, 
given any basis $\{
\varphi_j%(\varphi_j,\phi_j)
\}_{j\in \llbracket 1, 2(p+1)(p+3) \rrbracket}$ of $\harm{p}\times\harm{p}$, then  both
    \begin{equation}\label{eq:basis first algorithm}
        \left\{ 
        {\left(\funel{Q}_j, \funmg{Q}_j\right) = \tt ConstructQTpol}\left(p,\bm{x}_T,\taylortrunsc{p}{\,\bm{x}_T} [\permitt], \taylortrunsc{p}{\,\bm{x}_T} [\permeab], 
        \varphi_j%\varphi_j,\phi_j
        \right)
        \right\}_{j\in \llbracket 1, 2(p+1)(p+3) \rrbracket}
    \end{equation}
    \begin{equation}\label{eq:basis second algorithm}
        \left\{ 
        {\left(\funel{Q}_j, \funmg{Q}_j\right) = \tt ALTConstructQTpol}\left(p,\bm{x}_T,\taylortrunsc{p}{\,\bm{x}_T} [\permitt], \taylortrunsc{p}{\,\bm{x}_T} [\permeab], 
        \varphi_j%\varphi_j,\phi_j
        \right)
        \right\}_{j\in \llbracket 1, 2(p+1)(p+3) \rrbracket}
    \end{equation}
form a bases of $\QTopt{p}(T)$, see Theorem \ref{thm:QTbasis}.
In order to fully describe how to construct such bases, we only need to describe how to construct a basis for harmonic vector fields.

In order to define a basis for harmonic vector fields $\harm{p}\times\harm{p}$,
consider the space of homogeneous harmonic vector fields $\homharm{k}$, starting from the following characterization.

\begin{lemma}[Harmonic polynomials]
    Let $k\in\mathbb N$ and a point $\bm{x}_0\in\mathbb R^3$. Then,
    $
    \homharm{k} 
    = \grad \Big[ \ker (\Delta|_{\hompol{k}}) \Big]
    $.
\end{lemma}
\begin{proof}
    For any $\bm Q\in (\hompol{k})^3$, by definition
    $$
    \bm Q \in  \homharm{k}
    \Leftrightarrow
    \left\{\begin{array}{c}
            \curl\bm{Q}=\bm{0},\\
            \diver\bm{Q}=0,
    \end{array}\right.
    $$
    and therefore, from properties of the de Rham sequence
    $$
    \bm Q \in  \homharm{k}
    \Leftrightarrow
    \left\{\begin{array}{c}
        \exists q\in\hompol{k},  \bm{Q} = \grad q,\\
        \Delta q = 0.
    \end{array}\right.
    $$
    This concludes the proof.
\end{proof}

Accordingly, Algorithm \ref{alg:homogeneouns HVF construction} proposes an explicit procedure to construct a homogeneous polynomial harmonic vector field based on the construction of a scalar harmonic polynomial from \cite{Imbert-Gerard_2025-2}.
An explicit procedure for constructing a basis of harmonic vector fields $\harm{p}\times \harm{p}$ then naturally follows; it is presented in Algorithm \ref{alg:HVF basis construction}.

%\lmig{what do you think about this paragraph? another option would be to say that first and then write the Lemma and algorithms at the origin}\IF{I prefer the current option.}
These procedures have been implemented to produce the results presented in the next section. More precisely, with the context of QT-DG in mind, a single basis of harmonic vector fields $\harm{p}\times \harm{p}$ can be constructed centered at the origin, and then it can be used on each mesh element by simply translating the canonical basis terms.

\begin{algorithm}[ht]
    \caption{
    %Construction of a homogeneous polynomial harmonic vector field
    $\bm{Q} = {\tt ConstructhomHVFpol}
    \left(
    k,\bm{x}_0,
    \{ \theta_{\ell} \}_{\ell\in\indices{1}{2k+3}}
    \right)
    $
    }\label{alg:homogeneouns HVF construction}

    {\bf Input:} {$k\in\mathbb N$, $\bm{x}_0\in\mathbb R^3$,
    \\
    \nonl\phantom{\bf Input:}
    $\{ \theta_{\ell} \}_{\ell\in\indices{1}{2k+3}}\in\mathbb C^{2k+3}$ 
    }

    {\bf set} $n = 1$ \\
    \For{$i_1\in\indices{0}{1}$}{
        \For{$i_2\in\indices{0}{k+1-i_1}$}{
            {\bf set} $i_3 = k+1 - i_1 - i_2$ %\IF{Shouldn't be $k+1-i_1-i_2$? The degree of the scalar harmonic polynomial should be $k+1$}
            \\
            {\bf set} $c_{\bm{i}} = \theta_{n}$\\
            {\bf set} $n=n+1$
        }
    }

    \For{$i_1\in\indices{0}{k-1}$}{
        \For{$i_2\in\indices{0}{k-1-i_1}$}{
            {\bf set} $i_3 = k-1 - i_1 - i_2$\\
            {\bf set} $c_{\bm{i}+2\bm{e}_1} = \displaystyle\frac{1}{(i_1+2)(i_1+1)} \left[(i_2+2)(i_2+1) c_{\bm{i}+2\bm{e}_2} + (i_3+2)(i_3+1) c_{\bm{i}+2\bm{e}_3}\right]$
        }
    }

    {\bf Output: } $\bm Q(\bm{x}) = \displaystyle \grad \left[\sum_{\lvert\bm{i}\rvert = k+1} c_{\bm{i}} \left(\bm{X}-\bm{x}_0\right)^{\bm{i}}\right]$
    %It is defined in the preliminary section \IF{Is this $\bm{X}$ defined somewhere?}
    
\end{algorithm}

\begin{algorithm}[ht]
    \caption{
    %Construction of a homogeneous polynomial harmonic vector field
    $\{\bm{Q}_n\}_{n\in\indices{1}{2(p+1)(p+3)}}  = {\tt ConstructHVFpolEB}
    \left(p,\bm{x}_0,
    \left\{
    b^k%\{ b_{\ell} \}_{\ell\in\indices{1}{2k+3}}
    \right\}_{k\in\indices{0}{p}}
    \right)
    $ %{\color{red}SHOULDNT THERE BE A p HERE???????}
    }\label{alg:HVF basis construction}

    {\bf Input:} {$p\in\mathbb N$, $\bm{x}_0\in\mathbb R^3$,
    \\
    \nonl\phantom{\bf Input:}
    $\forall k\in\indices{0}{p}, b^k=\{ b_{\ell}^k \}_{\ell\in\indices{1}{2k+3}}$ a basis for $ \mathbb C^{2k+3}$
    }

    \For{$k\in\indices{1}{p}$}{
        \For{$\ell\in\indices{1}{2k+3}$}{
            $\bm{Q}_\ell^k = {\tt ConstructhomHVFpol}
            \left( k,\bm{x}_0,b^k_{\ell} \right)
            $
      }
    }
    
    {\bf Output: } $\Big\{
    \left\{(\bm Q^k_\ell,\bm 0)\right\}_{\ell\in\indices{1}{2k+3}}
    \cup
    \left\{(\bm 0,\bm Q^k_\ell)\right\}_{ \ell\in\indices{1}{2k+3}}\Big\}_{k\in\indices{0}{p}}$ 
    
\end{algorithm}

\section{Numerical results}\label{sec:numerical results}
Previous sections defined the QT space $\QTopt{p}(T)$, described and proved its best approximation properties, and presented algorithms for constructing QT bases. 
This section presents numerical experiments to illustrate the practical properties of $\QTopt{p}(T)$. These experiments involve both the implementation of QT basis construction approaches and the validation of best approximation properties.
%In this section, we show the approximation properties for two numerical test cases with a known exact solution: with constant coefficients in Subsection~\ref{subsec:test - constant coeffs}, and with variable coefficients in Subsection~\ref{subsec:test - variable coeffs test1}.

\subsection{Setup}\label{subsec:setup}
The problems of interest are governed by the first-order time-harmonic Maxwell's equations
\eqref{eq:first-order Maxwell}, with given coefficients $\permitt,\permeab$. 
Two test cases will then be defined by specifying $\permitt,\permeab$, one with constant coefficients, the other with variable coefficients, as well as an exact solution $\bm{V}_{\rm exact}=(\el_{\rm exact}, \mg_{\rm exact})$ to the corresponding PDE system.
%I kept Q only for the basis functions, then used V for "approx" and "exact" cause it is more readable to have the same letter for both\IF{In the rest of the paper, $\bm{Q}$ denoted a polynomial. Do you think it is fine to keep $\bm{Q}_{\rm exact}$ to denote the exact solution, or should we change the notation? What about $\bm{V}$ or $\bm{V}_{\rm exact}$ (for vector field)?}
%let's keep it like that for now and see what reviewers will say about that\IF{Also, if the subscripts are too long and annoying, we could change them to ex and ap}

Consider first a single open set $T\subset \mathbb R^3$ with a point $\bm x_T\in T$ and a fixed parameter $p$. 
A local QT basis for $\QTopt{p}(T)$, denoted $\{\bm{Q}_j\}_{j\in\indices{1}{2(p+1)(p+3)}}$, can be constructed from the Taylor expansion polynomial of $\permitt,\permeab$, as described in Section~\ref{subsec:Implementation consideration}. All numerical tests have been conducted using both Algorithms \ref{alg:QT construction} and \ref{alg:altQT construction}. However, the choice of algorithm did not appear to make any significant difference in the results. Therefore, all results presented here have been obtained using a single algorithm, namely Algorithm \ref{alg:altQT construction}.
Next, a QT element $\bm{V}_{\rm approx}=(\el_{\text{approx}}, \mg_{\text{approx}})\in\QTopt{p}(T)$,
\begin{equation*}
    \bm{V}_{\rm approx}=\sum_{j=1}^{2(p+1)(p+3)} \alpha_j\bm{Q}_j,
\end{equation*}
 is constructed as an approximation of the exact solution 
$
\bm{V}_{\rm exact}%=(\el_{\rm exact}, \mg_{\rm exact})
$,
%\sout{by matching their Taylor polynomials of degree $p$.}
%
% {%READ AND CHECK WHAT FOLLOWS
% \sout{To approximate a given exact solution to Maxwell's equations \eqref{eq:first-order Maxwell}, denoted $\bm{V}_{\rm exact}$, by a QT function} 
% \modTC{More specifically,}
% \begin{equation*}
%     \bm{V}_{\rm approx}=\sum_{j=1}^{2(p+1)(p+3)} \alpha_j\bm{Q}_j,
% \end{equation*}
% \modTC{and, }
following the proof of Theorem~\ref{th:approximation properties QTopt},
by matching their Taylor polynomials of degree $p$: %the goal is to compute coefficients $\{\alpha_j\}_{j\in\indices{1}{2(p+1)(p+3)}}$ by matching the Taylor polynomials of degree $p$ of $\bm{V}_{\rm approx}$ and $\bm{V}_{\rm exact}$:
$$
    \sum_{j=1}^{2(p+1)(p+3)} \alpha_j\underbrace{\taylortrun{p}{\ \bm{x}_T}\left[\bm{Q}_j\right]}_{=\,\bm{Q}_j}
    =
    \taylortrun{p}{\ \bm{x}_T}\left[\bm{V}_{\rm exact}\right]
    \Leftrightarrow
    \forall \bm{i}  \in\mathbb N_0^3, |\bm{i}|\leq p, \sum_{j=1}^{2(p+1)(p+3)} \alpha_j D^{\bm{i}}\bm{Q}_j (\bm{x}_T) = D^{\bm{i}} \bm{V}_{\rm exact} (\bm{x}_T),
$$
and a solution always exists to this $(p+1)(p+2)(p+3)\times 2(p+1)(p+3)$ %you are right \IF{Shouldn't be $(p+1)(p+2)(p+3)\times 2(p+1)(p+3)$?} 
linear system according to Proposition~\ref{prop:TaylinQT}.
The second expression above naturally follows from the standard expression of Taylor polynomials in the canonical shifted monomial polynomial basis,
% or equivalently using the standard shifted monomial polynomial basis FORMULA?
% $$
%     \Leftrightarrow
%     \forall \bm{i}  \in\mathbb N_0^3, |\bm{i}\leq p, \sum_{j=1}^{2(p+1)(p+3)} \alpha_j D^{\bm{i}}\bm{Q}_j (\bm{x}_T) = D^{\bm{i}} \bm{Q}_{\rm exact} (\bm{x}_T).
% $$
and it is possible to compute the Least-Squares solution of any corresponding matrix system numerically.
In practice though, according to Theorem~\ref{th:construction QT function}, any QT function is uniquely determined by its harmonic component; it is therefore sufficient to match the harmonic components of the Taylor polynomials of degree $p$ of $\bm{V}_{\rm approx}$ and $\bm{V}_{\rm exact}$. Thanks to a harmonic basis, this can then be written as an equivalent square $2(p+1)(p+3)\times 2(p+1)(p+3)$ system. 
Here, Algorithm \ref{alg:HVF basis construction} from canonical bases for $ \mathbb C^{2k+3}$ constructs a harmonic basis, and a convenient numbering leads to the resulting equivalent square system's matrix into the identity.%pas clair du tout\modTC{For instance, using the same basis of $\harm{p}\times \harm{p}$ to represent the harmonic coefficients of the Taylor polynomial of $\bm{V}_{\rm approx}$ and to construct the QT basis $\{\bm{Q}_j\}_{j\in\indices{1}{2(p+1)(p+3)}}$ with the canonical initialization \lmig{I dont think that anyone except you and me would understand that}, then the resulting matrix is simply the identity matrix.}\IF{What about this description to conclude the discussion about how to find the approximation?}
%\sout{For instance,  REPRENDRE ICI}
%%}
%\lmig{now what exactly do you want to say about your choice of harmonic basis? is it really so that the square system described above can be written with an identity matrix?}
%\IF{In practice, once I have the Taylor coefficients of the exact solution (computed "by hand") in the canonical basis, I solve $p+1$ square systems of dimension $\frac{3}{2} (p+1)(p+2)\lmig{k?} = \dim\hompol{k}$ to get the coefficients in the $\homharm{k}$-$\homcurlfreebis{k}$-$\homdivfreebis{k}$ basis. Then, since the order of harmoni basis functions "corresponds" to the one used for the initialization for the basis of the QT space, the matrix of the square system is trivially the identity.}

\begin{remark}
    The central role played in this computational approach by harmonic components of polynomial vector fields is inherently tied to the definition of the quasi-Trefftz operator, see Remark \ref{rmk:qTop}, and more precisely to its so-called principal part, defined as
    $$
        \mathcal P_p:(\funel{Q}, \funmg{Q}) \in (\pol{p})^3\times (\pol{p})^3 \ \mapsto\ \left(
        \begin{aligned}
            &\curl \funel{Q}  \\
            &\curl \funmg{Q} \\
            &\permitt_0\diver \funel{Q}  \\
            &\permeab_0\diver \funmg{Q} 
        \end{aligned}\right).
    $$
    The main point is that the space of harmonic vector fields is in bijection with the kernel of $\mathcal P_p$. Therefore, they can be conveniently leveraged both to construct individual quasi-Trefftz functions, in the construction of bases, and to match two quasi-Trefftz functions, such as $\bm V_{\rm approx}$ and $\taylortrunsc{p}{\,\bm{x}_T} [\bm V_{\rm exact}]$.
    %On the one hand, harmonic components are the degrees of freedom in Algorithms \ref{alg:QT construction} and \ref{alg:altQT construction} to define an individual quasi-Trefftz function, and a given harmonic basis uniquely defines a quasi-Trefftz basis from Theorem \ref{thm:QTbasis}.
    %On the other hand, harmonic components of Taylor polynomials are ... to construct a quasi-Trefftz approximation to $\bm V_{\rm exact}$ from a quasi-Trefftz basis. 
    %%DESCRIBE THE FUNDAMENTAL PROPERTY OF HARM VECT FIELDS SPACE  %convenient space in bijection with the kernel of the curl operator
    %For comparison, in the scalar case, this approach is the same as SIMILAR TO THE POLYNOMIAL CASE but not the GPW case
    %\IF{Which is the main point of this remark?}\lmig{to highlight what is parallel/different compared with the scalar work}
\end{remark}

In order to assess the high-order $h$-convergence of this approximation, the error of interest is the following %\lmig{explain before giving the formula?}
%\modTC{
\begin{equation}\label{eq:error definition}
    \begin{split}
        \bm{err}(r, T) \coloneqq&\ \lVert \el_{\rm exact} - \el_{\text{approx}} \rVert_{\infty, \mathcal B(\bm{x}_T, r)} + \lVert \mg_{\rm exact} - \mg_{\text{approx}} \rVert_{\infty, \mathcal B(\bm{x}_T, r)} \\ 
        =&\ \sup_{\lVert\bm{x} - \bm{x}_T\rVert \leq r}  \lVert \el_{\rm exact}(\bm{x}) - \el_{\text{approx}} (\bm{x}) \rVert + \sup_{\lVert\bm{x} - \bm{x}_T\rVert \leq r}  \lVert \mg_{\rm exact}(\bm{x}) - \mg_{\text{approx}} (\bm{x}) \rVert.
    \end{split}
\end{equation}
%}

In a QT-DG method, local spaces are constructed on each mesh element.
Here, for a fixed value of $p$, consider the domain $[-1, 1]^3$, and the mesh $\mathcal T_h$ with 1296 tetrahedra and $h\approx 0.5774$ depicted in Figure~\ref{subfig:Cube6-elements}. For each element $T\in\mathcal{T}_h$, the point $\bm{x}_T\in T$ is chosen as the centroid of $T$. Figure~\ref{subfig:Cube6-centroids} displays the element centroids.
The procedure described above is applied locally at each $\bm{x}_T\in T$, and the errors reported below are the maximum error over all $\bm{x}_T$: 
%\modTC{
\begin{equation}\label{eq:max error definition}
    \bm{err}(r) \coloneqq \max_{T\in\mathcal{T}_h} \bm{err}(r, T).
\end{equation}
%}
%\lmig{add an explicit formula}
% 
For values of $p$ from $1$ to $10$, all results show convergence of order $p+1$, matching the estimate from Theorem \ref{th:approximation properties QTopt}.

\begin{figure}[ht]
    \centering
    \begin{subfigure}{0.35\textwidth}
        \includegraphics[width=\textwidth]{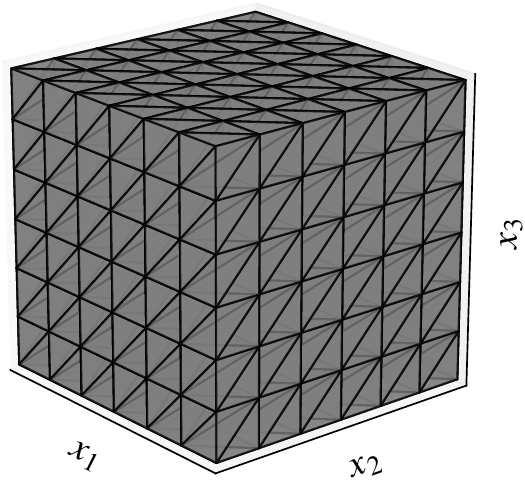}%{Figures/Meshes/Cube6/mesh_from_tetrahedrons2_crop.pdf}
        \caption{Mesh elements}
        \label{subfig:Cube6-elements}
    \end{subfigure}
    \hspace{1.5cm}
    \begin{subfigure}{0.35\textwidth}
        \includegraphics[width=\textwidth]{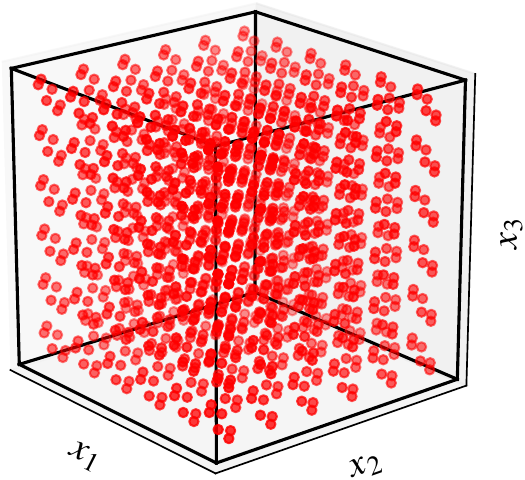}%{Figures/Meshes/Cube6/mesh_from_tetrahedrons3_crop.pdf}
        \caption{Centroids}
        \label{subfig:Cube6-centroids}
    \end{subfigure}
    \caption{Mesh for the numerical simulations on the domain $[-1, 1]^3$: elements ({\it left}) and centroids ({\it right}).}
    \label{fig:Cube6}
\end{figure}

\subsection{Test cases and approximation results}\label{subsec:test cases}

For this first numerical test case, consider the following constant coefficient setting: $\kappa = 30$, $\permitt = 1$, $\permeab = 1$, 
and the solution
\begin{equation*}
    \el_{\rm exact}(\bm{x}) = \begin{pmatrix}
        0 \\ 0 \\ e^{\imag \kappa x_2}
    \end{pmatrix}, 
    \qquad\qquad
    \mg_{\rm exact}(\bm{x}) = \begin{pmatrix}
        e^{\imag \kappa x_2} \\ 0 \\ 0
    \end{pmatrix}.
\end{equation*}

\begin{figure}[htbp]
    \centering
    \begin{subfigure}{0.55\textwidth}
        \centering
        \resizebox{!}{0.6\textwidth}{%
        \input{maxwell_errors_sum_3}%{Figures/Test_Mesh_3/maxwell_errors_sum_3}
        }
        \caption{Constant coefficient test.}
        \label{fig:convergence test_mesh_3 plots}
    \end{subfigure}
    \begin{subfigure}{0.4\textwidth}
        \centering
        \resizebox{!}{0.8\textwidth}{%
        \input{maxwell_errors_sum_4}%{Figures/Test_Mesh_4/maxwell_errors_sum}
        }
        \caption{Variable coefficient test.}
        \label{fig:convergence test_mesh_4 plots}
    \end{subfigure}
    \\
    \begin{subfigure}{\textwidth}
        \bigskip
        \centering
        \resizebox{!}{0.18\textwidth}{%
        \begin{tabular}{|c|c|c|c|c|c|c|c|c|c|c|}
            \hline
            Radius $r$ & $p=1$ & $p=2$ & $p=3$ & $p=4$ & $p=5$ & $p=6$ & $p=7$ & $p=8$ & $p=9$ & $p=10$ \\
            \hline\hline
            $\lambda/2$ & 0.76 & 2.27 & 3.50 & 4.63 & 5.71 & 6.77 & 7.81 & 8.84 & 9.86 & 10.88 \\
            $\lambda/2^2$ & 1.69 & 2.80 & 3.86 & 4.90 & 5.92 & 6.94 & 7.95 & 8.96 & 9.96 & 10.97 \\
            $\lambda/2^3$ & 1.93 & 2.95 & 3.96 & 4.97 & 5.98 & 6.98 & 7.99 & 8.99 & 9.99 & 10.99 \\
            $\lambda/2^4$ & 1.98 & 2.99 & 3.99 & 4.99 & 5.99 & 7.00 & 8.00 & 9.00 & 10.00 & 10.99 \\
            $\lambda/2^5$ & 2.00 & 3.00 & 4.00 & 5.00 & 6.00 & 7.00 & 8.00 & 9.00 & 9.97 & - \\
            $\lambda/2^6$ & 2.00 & 3.00 & 4.00 & 5.00 & 6.00 & 7.00 & 8.00 & - & - & - \\
            $\lambda/2^7$ & 2.00 & 3.00 & 4.00 & 5.00 & 6.00 & 6.98 & - & - & - & - \\
            $\lambda/2^8$ & 2.00 & 3.00 & 4.00 & 5.00 & 6.00 & - & - & - & - & - \\
            % $\lambda/2^9$ & 2.00 & 3.00 & 4.00 & 5.00 & 5.69 & - & - & - & - & - \\
            \hline
        \end{tabular}
        }
        % \caption{Table with convergence rates for the total error.}
        \caption{Constant coefficient test.}
        \label{fig:table test_mesh_3}
    \end{subfigure}
    \\
    \begin{subfigure}{\textwidth}
        \bigskip
        \centering
        \resizebox{!}{0.18\textwidth}{%
        \begin{tabular}{|c|c|c|c|c|c|c|c|c|c|c|}
            \hline
            Radius $r$ & $p=1$ & $p=2$ & $p=3$ & $p=4$ & $p=5$ & $p=6$ & $p=7$ & $p=8$ & $p=9$ & $p=10$ \\
            \hline\hline
            % $\lambda/2$ & 0.81 & 1.96 & 2.95 & 4.01 & 5.00 & 6.15 & 7.15 & 8.27 & 9.32 & 10.36 \\
            % $\lambda/2^2$ & 0.88 & 1.81 & 3.02 & 4.20 & 5.39 & 6.46 & 7.62 & 8.62 & 9.69 & 10.63 \\
            $\lambda/2^3$ & 0.98 & 2.45 & 3.68 & 4.76 & 5.86 & 6.88 & 7.84 & 8.83 & 9.86 & 10.89 \\
            $\lambda/2^4$ & 1.86 & 2.92 & 3.98 & 5.00 & 5.95 & 6.93 & 7.96 & 8.96 & 9.97 & 10.98 \\
            $\lambda/2^5$ & 2.04 & 3.05 & 4.00 & 4.97 & 6.00 & 6.98 & 8.00 & 8.99 & 10.00 & 11.00 \\
            $\lambda/2^6$ & 2.03 & 2.99 & 4.01 & 4.99 & 6.01 & 7.00 & 8.00 & 9.00 & 10.00 & 10.75 \\
            $\lambda/2^7$ & 2.02 & 3.00 & 4.01 & 5.00 & 6.00 & 7.00 & 8.00 & 8.80 & - & - \\
            $\lambda/2^8$ & 2.01 & 3.00 & 4.00 & 5.00 & 6.00 & 7.00 & - & - & - & - \\
            $\lambda/2^9$ & 2.00 & 3.00 & 4.00 & 5.00 & 6.00 & - & - & - & - & - \\
            $\lambda/2^{10}$ & 2.00 & 3.00 & 4.00 & 5.00 & 5.90 & - & - & - & - & - \\
            \hline
        \end{tabular}
        }
        % \caption{Table with convergence rates for the total error.}
        \caption{Variable coefficient test.}
        \label{fig:table mesh_test_4}
    \end{subfigure}
    \caption{Error convergence for the two numerical tests of Subsection~\ref{subsec:test cases} with the setup described in Subsection~\ref{subsec:setup}: convergence plots ({\it top}) and convergence rates tables ({\it middle} and {\it bottom}). %$\kappa = 30$, $\permitt = 1$, $\permeab = 1$, $\el(\bm{x}) = (0, 0, e^{\imag \kappa x_2})^T$, $\mg(\bm{x}) = (e^{\imag \kappa x_2} , 0, 0)^T$. 
    In the tables, $\lambda = 2\pi/\kappa$ is the wavelength. The results have been obtained by implementing Algorithm~\ref{alg:altQT construction}.}
    \label{fig:convergence test_mesh_3}
\end{figure}

For the second test case, we consider the following setting with $\kappa = 5$, $\permeab(\bm{x}) = \permeab = 1$ and a variable coefficient: $\permitt(\bm{x}) = 11 - x_1$, and the exact solution
\begin{equation}\label{eq:exact sol-variable coeffs}
     \el_{\rm exact}(\bm{x}) = \begin{pmatrix}
         0 \\ 0 \\ \airy\left(\kappa^{2/3} (x_1-10)\right) e^{\imag \kappa x_2}
     \end{pmatrix},
     \qquad\quad
     \mg_{\rm exact}(\bm{x}) = \begin{pmatrix}
         \airy\left(\kappa^{2/3} (x_1-10)\right) e^{\imag \kappa x_2} \\ \imag \kappa^{-1/3} \airy'\left(\kappa^{2/3} (x_1-10)\right) e^{\imag \kappa x_2} \\ 0
     \end{pmatrix}.
\end{equation}
Here, $\airy(x)$ is the Airy function, i.e., the solution of the differential equation $y''(x) - xy(x) = 0$.

%\subsection{Convergence results}

Figure~\ref{fig:convergence test_mesh_3} shows the results of the numerical simulations using Algorithm~\ref{alg:altQT construction}. In both cases, for the different values of the degree $p$, the error $\bm{err}(r)$ \eqref{eq:max error definition} decreases as $r$ decreases until it reaches machine precision. Additionally, the tables contain the convergence rates, matching the $p+1$ asymptotic from the theoretical result, see Theorem~\ref{th:approximation properties QTopt}.

\begin{remark}[Comparing two Computational time comparison for QT bases construction] 
    While the choice of algorithm did not appear to make any significant difference in the errors described above, a minor difference appeared in terms of computational time. 
    During the preprocessing step to construct bases for the harmonic, solenoidal and irrotationl spaces, 
    the computational time to construct the spaces $\mathbb S^+_p/\mathbb I^+_p$, required for Algorithm~\ref{alg:altQT construction}, is consistently lower than that required for $\divfree{p}/\curlfree{p}$, required for Algorithm~\ref{alg:QT construction}. This is due to the fact that the latter involves computing an SVD for each $k$ from 1 to $p$.
    Yet, as highlighted in Remark~\ref{rem:HSI independence of x0}, these spaces are independent of the center $\bm{x}_0$, hence their bases are computed only once in the context of a QT-DG method as opposed to the quasi-Trefftz bases that are constructed on each individual element.   
    As for the computational times recorded from running all tests comparing Algorithms~\ref{alg:altQT construction} and ~\ref{alg:QT construction} for the construction of quasi-Trefftz bases, they did not show any significant difference, so they are not reported here.
\end{remark}
\begin{table}[htbp]
    \centering
    \begin{tabular}{|c||c|c|c|c|c|}
    \hline
        $p$ & 1 & 2 & 3 & 4 & 5 \\ \hline
        $\harm{p}$ & $9.53\times10^{-4}$ & $1.51\times10^{-3}$  & $2.52\times10^{-3}$ & $4.45\times10^{-3}$ & $2.05\times10^{-2}$ \\ \hline
        $\mathbb S^+_p$ & $1.00\times10^{-5}$ & $3.60\times10^{-5}$
        & $7.00\times10^{-5}$ & $1.83\times10^{-4}$ & $7.86\times10^{-4}$ \\%& $1.18\times10^{-3}$ & $1.40\times10^{-3}$ & $2.43\times10^{-3}$ & $4.44\times10^{-3}$ & $1.08\times10^{-2}$ \\ 
        $\divfree{p}$ & $5.87\times10^{-4}$ & $1.68\times10^{-3}$
        & $7.10\times10^{-3}$ & $6.31\times10^{-2}$
        & $4.79\times10^{-2}$ \\% & $5.97\times10^{-2} & $1.45\times10^{-1}$ & $3.69\times10^{-1}$ & $5.47\times10^{-1}$ & $1.04\times10^{0}$ \\
        \hline
        $\mathbb I^+_p$ & $1.00\times10^{-5}$
        & $2.60\times10^{-5}$ & $3.90\times10^{-5}$ & $9.00\times10^{-5}$ & $3.78\times10^{-4}$ \\ %& $4.32\times10^{-4}$ & $6.24\times10^{-4}$ & $1.25\times10^{-3}$ & $2.32\times10^{-3}$ & $3.77\times10^{-3}$ \\
        $\curlfree{p}$ & $5.68\times10^{-4}$
        & $1.62\times10^{-3}$ & $4.54\times10^{-3}$ & $5.81\times10^{-2}$ & $1.02\times10^{-1}$ \\% & $4.39\times10^{-2}$ & $9.76\times10^{-2}$ & $1.69\times10^{-1}$ & $4.51\times10^{-1}$ & $6.42\times10^{-1}$ \\

        \hline\hline
        $p$ & 6 & 7 & 8 & 9 & 10 \\ \hline
        $\harm{p}$ & $2.83\times10^{-2}$ & $2.11\times10^{-2}$ & $5.83\times10^{-2}$ & $7.64\times10^{-2}$ & $7.36\times10^{-2}$ \\ \hline
        $\mathbb S^+_p$ & %$1.00\times10^{-5}$ & $3.60\times10^{-5}$ & $7.00\times10^{-5}$ & $1.83\times10^{-4}$ & $7.86\times10^{-4}$ & 
        $1.18\times10^{-3}$ & $1.40\times10^{-3}$ & $2.43\times10^{-3}$ & $4.44\times10^{-3}$ & $1.08\times10^{-2}$ \\ 
        $\divfree{p}$ & %$5.87\times10^{-4}$ & $1.68\times10^{-3}$ & $7.10\times10^{-3}$ & $6.31\times10^{-2}$ & $4.79\times10^{-2}$ & 
        $5.97\times10^{-2}$ & $1.45\times10^{-1}$ & $3.69\times10^{-1}$ & $5.47\times10^{-1}$ & $1.04\times10^{0}$ \\ \hline
        $\mathbb I^+_p$ & %$1.00\times10^{-5}$ & $2.60\times10^{-5}$ & $3.90\times10^{-5}$ & $9.00\times10^{-5}$ & $3.78\times10^{-4}$ & 
        $4.32\times10^{-4}$ & $6.24\times10^{-4}$ & $1.25\times10^{-3}$ & $2.32\times10^{-3}$ & $3.77\times10^{-3}$ \\
        $\curlfree{p}$ & %$5.68\times10^{-4}$ & $1.62\times10^{-3}$ & $4.54\times10^{-3}$ & $5.81\times10^{-2}$ & $1.02\times10^{-1}$ & 
        $4.39\times10^{-2}$ & $9.76\times10^{-2}$ & $1.69\times10^{-1}$ & $4.51\times10^{-1}$ & $6.42\times10^{-1}$ \\ \hline
    \end{tabular}
    \caption{Computation time to construct bases for the harmonic $\harm{p}$, solenoidal $\mathbb S^+_p$ and $\divfree{p}$, and irrotational $\mathbb I^+_p$ and $\curlfree{p}$ spaces, then leveraged in Algorithm \ref{alg:QT construction} for $*$ spaces and in Algorithm \ref{alg:altQT construction} for $+$ spaces. All times are reported in seconds.}
    \label{tab:computational time}
\end{table}

\section{Conclusion}
Given a PDE $\mathcal L u = 0$,
QT operators are defined as the composition of Taylor truncation and the differential operator $\mathcal L$. In the case of a scalar PDE, the principal part of QT operators is automatically surjective \cite{Imbert-Gerard_2025-2}; in the case of a vector-valued PDE, this is not true in general, and in particular it is not true for Maxwell's equations due to properties of the curl operator.
This is the fundamental difference between the scalar and vector cases.

In this work on time-harmonic Maxwell's equations, 
the QT space definition leverages divergence conditions to define QT operators whose principal part is surjective. Together with Helmholtz decompositions and exact sequences, this plays a central role in the construction of both individual QT functions and QT bases.
%optimality?
The optimal dimension of the proposed space also relies strongly on divergence properties.

In conclusion, rather than a systematic approach as in the scalar case, it is a deep understanding of the structure of Maxwell's equations, based on divergence conditions and polynomial de Rham complexes, that enables this study of optimal QT space.

\section{Acknowledgments}
This material is based upon work supported by the U.S. Department of Energy, Office of Science, Office of Advanced Scientific Computing Research, under Award Number DE-SC0024246.

 \appendix

\section{Properties of Taylor polynomials}

The study of the QT space $\QTopt{p}(T)$ in Sections~\ref{sec:quasi-Trefftz} and \ref{sec:optimal properties} relies heavily on the following simple yet fundamental relations between Taylor polynomial operators and differential operators.
%
%The study of quasi-Trefftz spaces in Sections~\ref{sec:quasi-Trefftz} and \ref{sec:optimal properties} heavily relies on various relations between Taylor polynomial operators and differential operators.
%%The aim of this subsection is to summarize I TRY TO AVOID SENTENCES LIKE THAT, THEY ARE A LITTLE HEAVY
%Here are summarized some of these simple yet fundamental properties. % of Taylor polynomials \eqref{eq:scalar Taylor polynomial} and \eqref{eq:vector Taylor polynomial} useful for studying quasi-Trefftz spaces. %that will be used for proving the results of Sections~\ref{sec:quasi-Trefftz} and \ref{sec:optimal properties}.
%
The first result focuses on single partial derivative operators;
the second one focuses on the differential operators from the de Rham complex;
the last one focuses on variable-coefficient operators.

\begin{proposition}[Derivatives and Taylor polynomials]\label{prop:multi-index derivative}
    Let $U\subseteq \mathbb R^3$ be an open set, $\bm{x}_0\in U$, and $s\in\mathbb N$.
    \begin{enumerate}[(a)]
        \item If %$f\in\mathcal C^s(D)$ is a complex-valued scalar function, then
        $f\in C^s(U; \mathbb C)$, then for all $\bm{i}\in\mathbb N_0^3$ \label{it:Taylor+derivation}
        % \begin{equation*}
        %     D^{\bm{i}} \taylortrunsc{s}{\,\bm{x}_0} [f] = \taylortrunsc{s-\lvert\bm{i}\rvert}{\,\bm{x}_0} [D^{\bm{i}} f]
        %     \qquad\qquad
        %     \forall \bm{i}\in\mathbb N_0^3, \lvert\bm{i}\rvert \leq s.
        % \end{equation*}
        \begin{equation*}
            D^{\bm{i}} \taylortrunsc{s}{\,\bm{x}_0} [f] = 
            \begin{cases}
                \taylortrunsc{s-\lvert\bm{i}\rvert}{\,\bm{x}_0} \left[D^{\bm{i}} f \right]
                &\qquad \text{ if } \lvert\bm{i}\rvert \leq s, \\
                0 &\qquad \text{ if } \lvert\bm{i}\rvert > s.
            \end{cases}
        \end{equation*}
        
        \item If $f \in C^s(U; \mathbb C)$, then for all $\bm{i}\in\mathbb N_0^3$ \label{it:derivative scalar Taylor}
        \begin{equation*}%\label{eq:derivative values-scalar}
            D^{\bm{i}} \taylortrunsc{s}{\,\bm{x}_0}[f](\bm{x}_0) = \begin{cases}
                D^{\bm{i}} f(\bm{x}_0) &\quad \text{if}\ \lvert\bm{i}\rvert \leq s, \\
                0 &\quad \text{if}\ \vert\bm{i}\rvert > s.
            \end{cases}
        \end{equation*}

        \item If $\bm{F}\in C^s(U;\mathbb C^3)$, then for all $\bm{i}\in\mathbb N_0^3$ \label{it:derivative vector Taylor}
        \begin{equation*}%\label{eq:derivative values-vector}
            D^{\bm{i}} \taylortrun{s}{\,\,\bm{x}_0}[\bm{F}](\bm{x}_0) = \begin{cases}
                D^{\bm{i}} \bm{F}(\bm{x}_0) &\quad \text{if}\ \lvert\bm{i}\rvert \leq s, \\
                \bm{0} &\quad \text{if}\ \vert\bm{i}\rvert > s.
            \end{cases}
        \end{equation*}
    \end{enumerate}
\end{proposition}
% \lmig{I am wondering if we should skip the proof(s)}
% \begin{proof}
%     Properties (a) and (b) directly follow from the explicit  \eqref{eq:scalar Taylor polynomial} of the Taylor polynomial, see e.g. \cite[Subsection 2.2]{Imbert-Gerard:25-polyonmial}.
%     Property (c) follows immediately from (b).
% \end{proof}
%\IF{I removed the proof}
Of course, these results hold in any dimension and for more general function spaces.
For reference, see \cite{BOOK_Brenner_2008} for result (a) on Sobolev spaces.

\begin{proposition}[Divergence, curl and gradient properties]\label{prop:divergence+taylor}
    Let $U\subseteq \mathbb R^3$ be an open set, $\bm{x}_0\in U$, and $s\in\mathbb N$.
    \begin{enumerate}[(a)]
        \item\label{DIVnTAYL} If %$\bm{F}\in C^s(D)$ is a vector field with values in $\mathbb C^3$, then
        $\bm{F} \in C^{s}(U; \mathbb C^3)$, then \label{it:Taylor-divergence condition}
        \begin{equation*}
            \diver \taylortrun{s}{\ \bm{x}_0} [\bm{F}] = \taylortrunsc{s-1}{\,\bm{x}_0} [\diver \bm{F}].
        \end{equation*}

        \item \label{it:curl-Taylor action} If %$\bm{F}\in C^{s}(D)$ is a vector field with values in $\mathbb C^3$, then 
        $\bm{F}\in C^s(U; \mathbb C^3)$, then for all $\bm{i}\in\mathbb N_0^3$
        % \begin{equation*}
        %     D^{\bm{i}} \curl \taylortrun{s}{\,\,\bm{x}_0} [\bm{F}](\bm{x}_0) = D^{\bm{i}} \curl \bm{F} (\bm{x}_0)
        %     \qquad\qquad \forall \lvert\bm{i}\rvert \leq s-1.
        % \end{equation*}
        \begin{equation*}
            D^{\bm{i}} \curl \taylortrun{s}{\,\,\bm{x}_0} [\bm{F}](\bm{x}_0) = 
            \begin{cases}
                D^{\bm{i}} \curl \bm{F} (\bm{x}_0)
                &\qquad \text{ if } \lvert\bm{i}\rvert \leq s-1, \\
                \bm{0} &\qquad \text{ if } \vert\bm{i}\rvert > s-1.
            \end{cases}
        \end{equation*}

        \item If %$\bm{F}\in C^{s}(D)$ is a vector field with values in $\mathbb C^3$, then 
        $\bm{F}\in C^s(U; \mathbb C^3)$, then for all $\bm{i}\in\mathbb N_0^3$ \label{it:divergence-Taylor action}
        % \begin{equation*}
        %     D^{\bm{i}} \diver \taylortrun{s}{\,\,\bm{x}_0} [\bm{F}](\bm{x}_0) = D^{\bm{i}} \diver \bm{F} (\bm{x}_0)
        %     \qquad\qquad \forall \lvert\bm{i}\rvert \leq s-1.
        % \end{equation*}
        \begin{equation*}
            D^{\bm{i}} \diver \taylortrun{s}{\,\,\bm{x}_0} [\bm{F}](\bm{x}_0) = 
            \begin{cases}
                D^{\bm{i}} \diver \bm{F} (\bm{x}_0)
                &\qquad \text{ if } \lvert\bm{i}\rvert \leq s-1, \\
                0 & \qquad \text{ if } \lvert\bm{i}\rvert > s-1.
            \end{cases}
        \end{equation*}
    \end{enumerate}
\end{proposition}

%\IF{We may add something with variable coefficients and properties of Taylor operator}

\begin{proof}
Property \ref{DIVnTAYL} follows from Proposition~\ref{prop:multi-index derivative}-\ref{it:Taylor+derivation} applied to each $\bm{i}\in\{\bm{e}_1, \bm{e}_2, \bm{e}_3\}$ component of the vector field $\bm{F}$:
    %(a) From Proposition~\ref{prop:multi-index derivative}-\ref{it:Taylor+derivation} applied to each component of $\bm{F}$ with $\bm{i}\in\{\bm{e}_1, \bm{e}_2, \bm{e}_3\}$, we get
    \begin{equation*}
        \begin{split}
            \diver \taylortrun{s}{\ \bm{x}_0} [\bm{F}] 
            =&\ \partial_x \taylortrunsc{s}{\,\bm{x}_0} [F_1] + \partial_y \taylortrunsc{s}{\,\bm{x}_0} [F_2] + \partial_z \taylortrunsc{s}{\,\bm{x}_0} [F_3] \\
            =&\ \taylortrunsc{s-1}{\,\bm{x}_0} \left[\partial_x F_1\right] + \taylortrunsc{s-1}{\,\bm{x}_0} \left[\partial_y F_2\right] + \taylortrunsc{s-1}{\,\bm{x}_0} \left[\partial_z F_3\right]
            = \taylortrunsc{s-1}{\,\bm{x}_0} [\diver \bm{F}].
        \end{split}
    \end{equation*}
%
    %\noindent
    Since
    \begin{equation*}
            D^{\bm{i}} \curl \taylortrun{s}{\,\,\bm{x}_0} [\bm{F}] (\bm{x}_0)
            =\ \begin{pmatrix}
                D^{\bm{i} + \bm{e}_2} \taylortrunsc{s}{\,\bm{x}_0} [F_3] (\bm{x}_0) - D^{\bm{i} + \bm{e}_3} \taylortrunsc{s}{\,\bm{x}_0} [F_2] (\bm{x}_0) \\
                D^{\bm{i} + \bm{e}_3} \taylortrunsc{s}{\,\bm{x}_0} [F_1] (\bm{x}_0) - D^{\bm{i} + \bm{e}_1} \taylortrunsc{s}{\,\bm{x}_0} [F_3] (\bm{x}_0) \\
                D^{\bm{i} + \bm{e}_1} \taylortrunsc{s}{\,\bm{x}_0} [F_2] (\bm{x}_0) - D^{\bm{i} + \bm{e}_2} \taylortrunsc{s}{\,\bm{x}_0} [F_1] (\bm{x}_0)
            \end{pmatrix},
    \end{equation*}
    
Property \ref{it:curl-Taylor action} follows from Proposition~\ref{prop:multi-index derivative}-\ref{it:derivative scalar Taylor} applied to each component of the vector field $\bm{F}$ as, given $\lvert\bm{i}\rvert\leq s-1$, $\lvert\bm{i} + \bm{e}_j\rvert \leq s$ so then:
    %(b) Fix $\bm{i}\in\mathbb N_0^3$ such that $\lvert\bm{i}\rvert\leq s-1$. Then, $\lvert\bm{i} + \bm{e}_j\rvert \leq s$ for all $j\in\{1, 2, 3\}$, and, applying \eqref{eq:derivative values-scalar} to each component of $\bm{F}$, we have
    \begin{equation*}
            D^{\bm{i}} \curl \taylortrun{s}{\,\,\bm{x}_0} [\bm{F}] (\bm{x}_0)
            =\ \begin{pmatrix}
                D^{\bm{i} + \bm{e}_2} F_3 (\bm{x}_0) - D^{\bm{i} + \bm{e}_3} F_2 (\bm{x}_0) \\
                D^{\bm{i} + \bm{e}_3} F_1 (\bm{x}_0) - D^{\bm{i} + \bm{e}_1} F_3 (\bm{x}_0) \\
                D^{\bm{i} + \bm{e}_1} F_2 (\bm{x}_0) - D^{\bm{i} + \bm{e}_2} F_1 (\bm{x}_0)
            \end{pmatrix}
            = D^{\bm{i}} \curl \bm{F} (\bm{x}_0).
    \end{equation*}
%
    %\noindent
    Since
    \begin{equation*}
            D^{\bm{i}} \diver \taylortrun{s}{\,\,\bm{x}_0} [\bm{F}](\bm{x}_0) =\ D^{\bm{i} + \bm{e}_1} \taylortrunsc{s}{\, \bm{x}_0} [F_1] (\bm{x}_0) + D^{\bm{i} + \bm{e}_2} \taylortrunsc{s}{\, \bm{x}_0} [F_2] (\bm{x}_0) + D^{\bm{i} + \bm{e}_3} \taylortrunsc{s}{\, \bm{x}_0} [F_3] (\bm{x}_0) ,
    \end{equation*}
    
    Property \ref{it:divergence-Taylor action} also follows from Proposition~\ref{prop:multi-index derivative}-\ref{it:derivative scalar Taylor} applied to each component of the vector field $\bm{F}$ as, given $\lvert\bm{i}\rvert\leq s-1$, $\lvert\bm{i} + \bm{e}_j\rvert \leq s$, so then:
    %(c) Fix $\bm{i}\in\mathbb N_0^3$ such that $\lvert\bm{i}\rvert\leq s-1$. Then, $\lvert\bm{i} + \bm{e}_j\rvert \leq s$, $j=1, 2, 3$, and, applying \eqref{eq:derivative values-scalar} to each component of $\bm{F}$, we have
    \begin{equation*}
            D^{\bm{i}} \diver \taylortrun{s}{\,\,\bm{x}_0} [\bm{F}](\bm{x}_0)
            =\ D^{\bm{i} + \bm{e}_1} F_1(\bm{x}_0) + D^{\bm{i} + \bm{e}_2} F_2(\bm{x}_0) + D^{\bm{i} + \bm{e}_3} F_3(\bm{x}_0) = D^{\bm{i}} \diver \bm{F} (\bm{x}_0).
    \end{equation*}
\end{proof}

\begin{proposition}\label{prop:opertors+taylor w/ coefficients}
    Let $U\subseteq \mathbb R^3$ be an open set, $\bm{x}_0\in U$, and $s\in\mathbb N$.
    If $g\in C^s(U; \mathbb C)$ and $\bm{F}\in  C^s(U; \mathbb C^3)$, then %\label{it:gradient-Taylor action}
        \begin{equation*}
            D^{\bm{i}} \bigl(\diver (g \taylortrun{s}{\ \bm{x}_0} [\bm{F}]) \bigr) (\bm{x}_0) = D^{\bm{i}} \bigl(\diver (g \bm{F}) \bigr) (\bm{x}_0) \qquad\qquad \forall \bm{i}\in\mathbb N_0^3, \lvert\bm{i}\rvert\leq s-1.
        \end{equation*}
\end{proposition}
\begin{proof}
    From the simple fact that for smooth scalar functions $f_1$ and $f_2$
     \begin{equation*}%\label{eq:derivative B}
        \sum_{\substack{\bm{m}\in \mathbb N_0^3,\\ \bm{m}\leq\bm{i}}} \binom{\bm{i}}{\bm{m}} D^{\bm{m}} f_1(\bm{x}_T) D^{\bm{i}-\bm{m}} \taylortrunsc{p}{\ \bm{x}_T} [f_2] (\bm{x}_T) 
        =\sum_{\substack{\bm{m}\in \mathbb N_0^3,\\ \bm{m}\leq\bm{i}}} \binom{\bm{i}}{\bm{m}} D^{\bm{m}} f_1(\bm{x}_T) D^{\bm{i}-\bm{m}} f_2 (\bm{x}_T),
    \end{equation*}
    we get 
    \begin{equation*}%\label{eq:derivative divergence product}
        \begin{split}
            D^{\bm{i}} \bigl(\diver (g \taylortrun{s}{\ \bm{x}_0} [\bm{F}]) \bigr) (\bm{x}_0) =&\ D^{\bm{i}} \bigl(\grad g \cdot \taylortrun{s}{\ \bm{x}_0} [\bm{F}] \bigr) (\bm{x}_0) + D^{\bm{i}}\bigl(g \diver \taylortrun{s}{\ \bm{x}_0} [\bm{F}] \bigr) (\bm{x}_0) \\
            =&\ D^{\bm{i}} \bigl(\grad g \cdot \bm{F} \bigr) (\bm{x}_0) + D^{\bm{i}}\bigl(g \diver \bm{F} \bigr) (\bm{x}_0) = D^{\bm{i}} \bigl(\diver (g\bm{F})\bigr) (\bm{x}_0).
        \end{split}
    \end{equation*}
\end{proof}

\section{Curl and divergence-free bases}
\label{sec:CDbases}
%ADD PROOF THAT THESE ARE BASES

%\IF{I believe this appendix section can be removed, since the proof that these are bases is explicitly shown in \cite{Imbert-Gerard_2025-3}. Otherwise, we can state the results as follows and mention the paper as reference for the proof.}

{
The following theorems establish explicit bases for the irrotational and solenoidal spaces $\homcurlfreeall{k}$ and $\homdivfreeall{k}$, respectively. 
For a complete proof of these results, see Lemmas~2 and 3 of \cite{Imbert-Gerard_2025-3}.
\begin{theorem}[Basis for $\homcurlfreeall{k}$]
    Given $k\geq 1$, a basis for $\homcurlfreeall{k}$ is given by
    \begin{equation*}
        \left\{ \bm{\Phi}_{k, \bm{i}} ,{\bm{i}\in\mathbb N_0^3, \lvert\bm{i}\rvert = k+1} \right\},
    \end{equation*}
    where $\bm{\Phi}_{k, \bm{i}}$ are the curl-free functions defined in Definition~\ref{def:curl-free and divergence-free functions}.
\end{theorem}
\begin{theorem}[Basis for $\homdivfreeall{k}$]
    Given $k\geq 1$, a basis for $\homdivfreeall{k}$ is given by
    \begin{equation*}
    \begin{split}
        \cup_{{\bm{i}\in\mathbb N_0^3, \lvert\bm{i}
        \rvert = k}}\Big\{ &
        \left\{\bm{\Psi}_{k,1, \bm{i}} ,%{\bm{i}\in\mathbb N_0^3, \lvert\bm{i}\rvert = k},
        i_1=0\right\}
        \cup
        \left\{\bm{\Psi}_{k,2, \bm{i}} ,
        %{\bm{i}\in\mathbb N_0^3, \lvert\bm{i}\rvert = k}, 
        i_2=0\right\}
        \cup
        \left\{\bm{\Psi}_{k,3, \bm{i}} ,
        %{\bm{i}\in\mathbb N_0^3, \lvert\bm{i}\rvert = k}, 
        i_3=0\right\}
        \\&
        \cup
        \left\{\bm{\Psi}_{k,4, \bm{i}} ,
        %{\bm{i}\in\mathbb N_0^3, \lvert\bm{i}\rvert = k}, 
        i_1\geq 1\right\}
        \cup
        \left\{\bm{\Psi}_{k,5, \bm{i}} ,
        %{\bm{i}\in\mathbb N_0^3, \lvert\bm{i}\rvert = k}, 
        i_2\geq 1\right\}
        \Big\},
    \end{split}
    \end{equation*}
    where $\bm{\Psi}_{k, \bm{i}}$ are the divergence-free functions defined in Definition~\ref{def:curl-free and divergence-free functions}.
\end{theorem}
}

\bibliographystyle{siam}
\bibliography{refs-Maxwell}

\end{document}